\documentclass[a4paper, english]{amsart}
\usepackage[utf8]{inputenc}

\usepackage{amsthm}
\usepackage{amsfonts}
\usepackage{mathtools}
\usepackage{amssymb}
\usepackage{paralist}
\usepackage[all,cmtip]{xy}
\usepackage{enumitem}
\usepackage{kantlipsum}

\usepackage[colorlinks=true, linkcolor=blue, citecolor=blue, linktocpage=true]{hyperref}
\usepackage[capitalise, nameinlink]{cleveref}
\crefname{subsection}{Subsection}{Subsections}
\crefname{equation}{}{}

\renewcommand\le\leqslant
\renewcommand\ge\geqslant

\calclayout

\usepackage[maxalphanames=4, minalphanames=4, maxbibnames=99,backend=bibtex, style=alphabetic]{biblatex} 
\AtBeginBibliography{\footnotesize}

\setlist[itemize]{align=parleft,left=0pt..12pt,topsep=0pt,label={\(\bullet\)}}
\setlist[enumerate]{align=parleft,left=0pt..18pt,topsep=0pt}

\newcommand{\fg}[0]{\mathfrak{g}}

\newcommand{\bC}[0]{\mathbb{C}}
\newcommand{\bN}[0]{\mathbb{N}}
\newcommand{\bZ}[0]{\mathbb{Z}}

\newcommand{\sheafA}[0]{\mathcal{A}}

\newcommand{\sheafF}[0]{\mathcal{F}}
\newcommand{\sheafG}[0]{\mathcal{G}}
\newcommand{\sheafO}[0]{\mathcal{O}}

\newcommand{\compO}[0]{\widehat{\mathcal{O}}}
\newcommand{\compA}[0]{\widehat{\mathcal{A}}}
\DeclareMathOperator{\sheafHom}{\mathcal{H}\hspace{-1pt}\textit{om}}
\DeclareMathOperator{\sheafEnd}{\mathcal{E}\hspace{-1pt}\textit{nd}}

\DeclareMathOperator{\End}{End}
\DeclareMathOperator{\Hom}{Hom}
\DeclareMathOperator{\Aut}{Aut}

\usepackage{amsthm}

\numberwithin{equation}{section}
\renewcommand*{\theequation}{%
  \ifnum\value{section}>0 %
    \thesection.%
  \fi
    \arabic{equation}%
}

\makeatletter
\def\th@plain{%
  \thm@notefont{}
  \itshape 
}
\def\th@definition{%
  \thm@notefont{}
  \normalfont 
}
\makeatother

\newtheorem{theoremcounter}{theoremcounter}[section]
\newtheorem{maintheoremcounter}{maintheoremcounter}

\theoremstyle{plain}

\newtheorem{corollary}[theoremcounter]{Corollary}
\newtheorem{lemma}[theoremcounter]{Lemma}
\newtheorem{proposition}[theoremcounter]{Proposition}
\newtheorem{theorem}[theoremcounter]{Theorem}

\newtheorem{maintheorem}[maintheoremcounter]{Theorem}

\newtheoremstyle{example_style}
  {5pt} 
  {5pt} 
  {} 
  {} 
  {\itshape} 
  {.} 
  {10pt} 
  {} 

\theoremstyle{example_style}

\theoremstyle{definition}

\theoremstyle{example_style}
\newtheorem{remarkx}[theoremcounter]{Remark}
\newenvironment{remark}
  {\pushQED{\qed}\remarkx}
 {\popQED\endremarkx}

\title{Classification of \(D\)-bialgebra structures on power series algebras}
\author{Raschid Abedin}\address{
ETH Z\"urich\\
Department of Mathematics \\
R\"amistrasse 101\\ 8092 Zurich\\
Switzerland
}
\email{raschid.abedin@math.ethz.ch}

\begin{document}

\maketitle

\begin{abstract}
In this paper, we use algebro-geometric methods in order to derive classification results for so-called \(D\)-bialgebra structures on the power series algebra \(A[\![z]\!]\) for certain central simple non-associative algebras \(A\). These structures are closely related to a version of the classical Yang-Baxter equation (CYBE) over \(A\). 

If \(A\) is a Lie algebra, we obtain new proofs for pivotal steps in the known classification of non-degenerate topological Lie bialgebra structures on \(A[\![z]\!]\) as well as of non-degenerate solutions of the usual CYBE. 

If \(A\) is associative, we achieve the classification of non-triangular topological balanced infinitesimal bialgebra structures on \(A[\![z]\!]\) as well as of all non-degenerate solutions of an associative version of the CYBE.
\end{abstract}

\section{Introduction}

\subsection*{Background and motivation}
A Lie bialgebra \((L,\delta)\) over a field \(\Bbbk\) consists of a Lie algebra \(L\) over \(\Bbbk\) equipped with a skew-symmetric 1-cocycle \(\delta \colon L \to L \otimes L\) such that the dual map \(\delta^*\colon (L \otimes L)^* \to L^*\) restricted to \(L^*\otimes L^* \subseteq (L\otimes L)^*\) is a Lie bracket. 
Equivalently, given a pair \((L,\delta)\) of any (not necessarily Lie) algebra \(L\) over \(\Bbbk\) and any linear map \(\delta \colon L \to L \otimes L\), there is a canonical \(\Bbbk\)-algebra structure on \(D(L,\delta) \coloneqq L \oplus L^*\) and \((L,\delta)\) is a Lie bialgebra if and only if \(D(L,\delta)\) is a Lie algebra. The algebra \(D(L,\delta)\) is called classical double of \((L,\delta)\). 

Lie bialgebras first appeared in \cite{drinfeld_hamiltonian_structures}, where Drinfeld noticed that the Lie algebra tangent to a Poisson-Lie group, i.e.\ a Lie group equipped with a Poisson structure compatible with the group operation, is a Lie bialgebra. Later in \cite{drinfeld_quantum_groups}, he proposed to approach the quantization of Poisson-Lie group structures by deforming the universal enveloping algebra of the associated Lie bialgebras. This made Lie bialgebras integral to the field of quantum groups which arose from these considerations. Lie bialgebras have also seen application in the theory of classical integrable systems and are closely related to the classical Yang-Baxter equation; see e.g.\ \cite{chari_pressley}.

Let {\tt Alg}\(_\Bbbk\) be the category of non-associative \(\Bbbk\)-algebras, i.e.\ of vector spaces \(A\) over \(\Bbbk\) equipped with any linear map \(A \otimes A \to A\).
The classical double construction can be used to define analogs of Lie bialgebras in any full subcategory {\tt C} of {\tt Alg}\(_\Bbbk\) which is closed under taking subalgebras. Namely, consider a pair \((A,\delta)\) of an algebra \(A \in {\tt Alg}_\Bbbk\) and a linear map \(\delta \colon A \to A \otimes A\). As mentioned above, \(D(A,\delta) = A \oplus A^*\) can be equipped with a canonical \(\Bbbk\)-algebra structure and \((A,\delta)\) is called \(D\)-bialgebra in {\tt C}  if \(D(A,\delta) \in {\tt C}\); see \cite{zhelyabin}. Observe that in this case \(A, A^* \in {\tt C}\) holds automatically. In particular, \(D\)-bialgebras in the category of Lie algebras are precisely Lie bialgebras. 

The \(D\)-bialgebras in the category of associative algebras, or associative \(D\)-bialgebras for short, turn out to be precisely the co-opposites of balanced infinitesimal bialgebras in the sense of Aguiar \cite{aguiar_associative}. The latter see applications in the study of compatible associative multiplications (see \cite{odesskii_sokolov}) and in combinatorics (see \cite{aguiar_combinatoris}). Moreover, they induce Lie bialgebra structures and are related to an associative version of the classical Yang-Baxter equation; see \cite{aguiar_associative}.

The \(D\)-bialgebras in the category of Jordan algebras, called Jordan bialgebras, were introduced and discussed in the works of Zhelyabin \cite{zhelyabin,zhelyabin_symmetric_elements,zhelyabin_symplectic,zhelyabin_jordan_yang_baxter}, where they were related to associative \(D\)-bialgebras, Lie bialgebras, symplectic forms on Jordan algebras, and a Jordan version of the classical Yang-Baxter equation.

Let \(\fg\) be a finite-dimensional simple Lie algebra over an algebraically closed field \(\Bbbk\) of characteristic 0.
Several important infinite-dimensional Lie bialgebras, e.g.\ arbitrary Lie bialgebra structures on the polynomial Lie algebra \(\fg[z]\) or on a twisted loop algebra \(\mathfrak{L}(\fg,\sigma)\), can be completed to so-called topological Lie bialgebra structures on the Lie algebra \(\fg[\![z]\!]\) of formal power series with coefficients in \(\fg\). A topological Lie bialgebra \((\fg[\![z]\!],\delta)\) thereby consists of a skew-symmetric 1-cocycle \(\delta \colon \fg[\![z]\!] \to (\fg \otimes \fg)[\![x,y]\!]\) whose continuous dual 
\[\delta^\vee \colon \fg[\![z]\!]^\vee \otimes \fg[\![z]\!]^\vee\cong (\fg \otimes \fg)[\![x,y]\!]^\vee \to \fg[\![z]\!]^\vee\]
is a Lie bracket. Similar to the non-topological setting, there is a canonical multiplication on \(D(\fg[\![z]\!],\delta) = \fg[\![z]\!]\oplus \fg[\![z]\!]^\vee\) and \((\fg[\![x]\!],\delta)\) is a topological Lie bialgebra if and only if this multiplication is a Lie bracket. In \cite{montaner_stolin_zelmanov}, all possibilities for \(D(\fg[\![z]\!],\delta)\) were classified. More precisely, they proved that either \(D(\fg[\![z]\!],\delta) \cong D(\fg[\![z]\!],0)\) or \(D(\fg[\![z]\!],\delta)\) is isomorphic to \(\fg(\!(z)\!) \times \fg[z]/z^n\fg[z]\) for some \(n \in \{0,1,2\}\) as a Lie algebra. 

Topological Lie bialgebras are closely related 
to solutions
\begin{equation}\label{eq:standard_form_intro}
    r(x,y) = \frac{\lambda(x)y^n\gamma}{x-y} + t(x,y)
\end{equation}
of the classical Yang-Baxter equation (CYBE)
\begin{equation}\label{eq:cybe_intro}
    [r^{12}(z_1,z_2),r^{13}(z_1,z_3)] + 
    [r^{12}(z_1,z_2),r^{23}(z_2,z_3)] + 
    [r^{13}(z_1,z_3),r^{23}(z_2,z_3)] = 0.
\end{equation}
Here, \(\gamma\) is a the quadratic Casimir element of \(\fg\), \(\lambda \in \Bbbk[\![x]\!]^\times\) and \(t \in (\fg \otimes \fg)[\![x,y]\!]\). 
In particular, the assignment \(a(z) \mapsto [a(x) \otimes 1 + 1 \otimes a(y),r(x,y)]\) defines a 1-cocycle \(\delta_r \colon \fg[\![z]\!]\to (\fg \otimes \fg)[\![x,y]\!]\) for \(r\) of the form \eqref{eq:standard_form}. Furthermore, \(r \mapsto \delta_r\) defines a bijection between topological Lie bialgebra structures on \(\fg[\![z]\!]\) with double \(\fg(\!(z)\!) \times \fg[z]/z^n\fg[z]\) and solutions of the CYBE of the form \eqref{eq:standard_form_intro}. In \cite{abedin_maximox_stolin_zelmanov}, it is shown that a topological Lie bialgebra structure is determined up to isomorphism by one of 6 types of solutions of the CYBE: degenerate, elliptic, trigonometric, rational, quasi-trigonometric, or quasi-rational. All non-degenerate solutions  
of the CYBE have been classified.

In this paper, we introduce topological \(D\)-bialgebras structures on \(A[\![z]\!]\) for any \(\Bbbk\)-algebra \(A \in {\tt Alg}_{\Bbbk}\) and extend some results from \cite{montaner_stolin_zelmanov,abedin_maximox_stolin_zelmanov} to these using the methods from \cite{abedin_universal_geometrization,abedin_maximox_stolin_zelmanov}. Moreover, we establish a connection between topological \(D\)-bialgebras and a generalization of the classical Yang-Baxter equation. 

\subsection*{Content and results}
The notion of \(D\)-bialgebras as well as supplementary objects such as the classical double and isomorphisms of \(D\)-bialgebras will be discussed in Section \ref{sec:intro_D_bialgebras}. In Section \ref{sec:manin_triples_over_series}, these notions are adapted to the topological setting for power series algebras and we will introduce the main object of study: non-degenerate topological \(D\)-bialgebras. 

Let \(A\) be a finite-dimensional, central, simple \(\Bbbk\)-algebra equipped with a non-degenerate, associative, symmetric bilinear form \(\beta\colon A \times A \to \Bbbk\). By definition, a non-degenerate topological \(D\)-bialgebra \((A[\![z]\!],\delta)\) has a double isomorphic to \(D_n(A) \coloneqq A(\!(z)\!) \times A[z]/z^nA[z]\), where \(n \in \bN\). In Section \ref{sec:general_categorization}, we will prove that the cases \(n > 2\) cannot occur for so-called geometrically admissible algebras \(A\). In particular, this result holds for finite-dimensional central simple associative, Lie or Jordan algebras; see Section \ref{sec:geometrically_admissible_metrics}. More precisely, we have the following result; see Theorem \ref{thm:categorization_of_manin_triples}.

\begin{maintheorem}\label{mainthmA}
Let \(A\) be a geometrically admissible algebra over a field \(\Bbbk\) of characteristic 0, e.g.\ a finite-dimensional central simple Lie, associative, or Jordan algebra over \(\Bbbk\). 

The double \(D(A[\![z]\!],\delta)\) of a non-degenerate topological \(D\)-bialgebra \((A[\![z]\!],\delta)\) is isomorphic to \(A(\!(z)\!) \times A[z]/z^nA[z]\) for some \(n \in \{0,1,2\}\). \qed
\end{maintheorem}

\noindent
Let us note that, if \(A\) is a Lie algebra, Theorem \ref{mainthmA} is one of the main results in \cite{montaner_stolin_zelmanov}. However, our proof is independent of the proof in \cite{montaner_stolin_zelmanov} and is based on the geometrization of \(A\)-lattices (see Section \ref{sec:geometrization}). This method was already used in order to give a new proof of the Belavin-Drinfeld trichotomy of non-degenerate solutions of the CYBE \eqref{eq:cybe_intro} in \cite{abedin_universal_geometrization} and the classification of topological Lie bialgebras in \cite{abedin_maximox_stolin_zelmanov}. 

In Section \ref{sec:A_CYBE}, we will show that every non-degenerate topological \(D\)-bialgebra \((A[\![z]\!],\delta)\) is of the form \(a(z) \mapsto r(x,y)a(x)^{(1)} - a(y)^{(2)}r(x,y)\) for a solution \begin{equation}\label{eq:intro_standard_form}
    r(x,y) = \frac{\lambda(x)y^n\gamma}{x-y} + t(x,y) \in (A \otimes A)[\![x,y]\!][(x-y)^{-1}]
\end{equation}
of the so-called \(A\)-classical Yang-Baxter equation (\(A\)-CYBE) 
\begin{equation}\label{eq:intro_A_CYBE}
    r^{13}(z_1,z_3)r^{12}(z_1,z_2) - r^{12}(z_1,z_2)r^{23}(z_2,z_3) + r^{23}(z_2,z_3)r^{13}(z_1,z_3) = 0.
\end{equation}
Here, \(\lambda \in \Bbbk[\![x]\!]^\times, t \in (\fg \otimes \fg)[\![x,y]\!]\) and \(\gamma \in A\otimes A\) is a canonical \(A\)-invariant element determined by the algebra metric \(\beta\) of \(A\); see Section \ref{sec:series_in_standard_form} for details. In particular, the classification of non-degenerate topological \(D\)-bialgebras \((A[\![z]\!],\delta)\) up to isomorphism is equivalent to the classification of solutions of the \(A\)-CYBE \eqref{eq:intro_A_CYBE} of the form \eqref{eq:intro_standard_form} up to a certain type of equivalence relation.

The \(A\)-CYBE already appeared in several special cases in literature.
If \(A = \fg\) is a Lie algebra, \eqref{eq:intro_A_CYBE} is exactly the usual CYBE \eqref{eq:cybe_intro}. For constant \(r \in A \otimes A\), \eqref{eq:intro_A_CYBE} was examined in \cite{aguiar_associative} for an associative algebra \(A\) and in \cite{zhelyabin_jordan_yang_baxter} for a Jordan algebra \(A\). An endomorphism version of \eqref{eq:intro_A_CYBE} for an associative algebra \(A\) and meromorphic functions \(r = r(x,y)\) in two complex variables was related to pairs of compatible associative algebra structures in \cite{odesskii_sokolov}. We also point out that solutions of the associative Yang-Baxter equation introduced by Polishchuk in \cite[Eq. (0.1)]{polishchuk_CYBE} which do not depend on the first parameter are precisely difference depending meromorphic solutions of the \(A\)-CYBE for associative algebras \(A\).

In Section \ref{sec:categorization_general_whole_section}, we will see that non-degenerate topological \(D\)-bialgebras can be categorized rather explicitly by the form of their associated solution of the \(A\)-CYBE if \(A\) is a so-called strongly geometrically admissible \(\Bbbk\)-algebra; see Subsection \ref{sec:geometrically_M_admissible}. The most important examples are finite-dimensional simple Lie, Jordan, and associative algebras over an algebraically closed field of characteristic 0; see Section \ref{sec:etale_locally_trivial_sheaves_of_algebras}. More precisely, we have the following result; see Theorem \ref{thm:categorization_refined}.

\begin{maintheorem}\label{mainthmB}
Let \(\Bbbk\) be an algebraically closed of characteristic 0 and \(A\) be a unital strongly geometrically admissible algebra (e.g.\ a finite-dimensional, central, simple associative or Jordan algebra). Furthermore, let \((A[\![z]\!],\delta)\) be a non-degenerate topological \(D\)-bialgebra in some category of \(\Bbbk\)-algebras closed under taking subalgebras. 

Up to isomorphism of topological \(D\)-bialgebras, \(\delta = \delta_r\) for a solution of the \(A\)-CYBE of precisely one of the following forms:
\begin{enumerate}
    \item \(r\) is \emph{trigonometric} in the sense that there exists \(\sigma \in \textnormal{Aut}_{\Bbbk\textnormal{-alg}}(A)\) of order \(m \in \bN\) and \(t \in L(A,\sigma) \otimes L(A,\sigma)\) such that
    \begin{equation*}
    r(x,y) = \frac{1}{\exp\left(x-y\right)-1}\sum_{j = 0}^{m-1}\textnormal{exp}\left(\frac{x-y}{m}\right) \gamma_j + t\left(\exp\left(\frac{x}{m}\right),\exp\left(\frac{y}{m}\right)\right).
    \end{equation*}
    Here, \(\gamma_j \in A \otimes A\) is uniquely determined by \(\gamma = \sum_{j = 0}^{m-1} \gamma_j\) and \((\sigma \otimes 1)\gamma_j = \varepsilon^j\gamma_j\) for some primitive \(m\)-th root of unity \(\varepsilon \in \Bbbk\), where \(\gamma \in A\otimes A\) is the canonical \(A\)-invariant element;
    
        \item \(r\) is \emph{rational} in the sense that there exists \(t \in (A\otimes A)[x,y]\) such that \(r(x,y) = \frac{\gamma}{x-y} + t(x,y)\);
    
        \item \(r\) is \emph{quasi-trigonometric} in the sense that there exists a polynomial \(t \in (A\otimes A)[x,y]\) such that \(r(x,y) =\frac{y\gamma}{x-y} + t(x,y)\);
    
        \item \(r\) is \emph{quasi-rational} in the sense that there exists a polynomial \(t \in (A\otimes A)[x,y]\) such that \(r(x,y) = \frac{y^2\gamma}{x-y} + t(x,y).\)
    \end{enumerate}
In particular, every solution of the \(A\)-CYBE \eqref{eq:intro_A_CYBE} of the form \eqref{eq:intro_standard_form} is, up to equivalence, of one of the above forms. \qed
\end{maintheorem}

\noindent
The analog of Theorem \ref{mainthmB} for a Lie algebra \(A= \fg\) was proven in \cite{abedin_maximox_stolin_zelmanov} and can be seen as a generalization of the Belavin-Drinfeld trichotomy for non-degenerate \(r\)-matrices from \cite{belavin_drinfeld_solutions_of_CYBE_paper}. A consequence of this result is, that all topological Lie bialgebras \((\fg[\![z]\!],\delta)\) are classified. The proof of Theorem \ref{mainthmB} is, under consideration of Theorem \ref{mainthmA}, similar to the proof of its Lie algebra analog in \cite{abedin_maximox_stolin_zelmanov}. More precisely, it proceeds by refining the algebro-geometric methods already used to proof Theorem \ref{mainthmA}. Namely, we can assign a particular type of geometric data, called geometric \(A\)-CYBE datum (see Section \ref{sec:geometric_A_CYBE_data}), to any Lagrangian subalgebra \(W \subseteq D_n(A)\) complementary to \(A[\![z]\!]\); see Section \ref{sec:geometrization_of_manin_triples}.  If this assignment is done, Theorem \ref{mainthmB} is a consequence of the classification results for sheaves of algebras on the (punctured) affine line presented in Proposition \ref{thm:classification_sheaves_of_algebras}, which is a consequence of the results from \cite{pianzola}; see \cite[Theorem 6.1.1]{abedin_thesis} for details.

Let us point out that the unitality assumption in Theorem \ref{mainthmB} is actually rather weak, since strongly geometrically admissible power associative algebras which are not anti-commutative are automatically unital; see Remark \ref{rem:unitality}. The most interesting strongly geometrically admissible anti-commutative algebras are precisely Lie algebras, where the analog of Theorem \ref{mainthmB} is already known as mentioned above.

We conclude this paper, by using Theorem \ref{mainthmB} to classify all topological associative \(D\)-bialgebras \((A[\![z]\!],\delta)\) for \(A\) associative and \(D(A[\![z]\!],\delta) \ncong D(A[\![z]\!],0)\). It turns out that the trigonometric and quasi-trigonometric cases do not occur. More precisely, we obtain the following result in Section \ref{sec:classification_associatice_Dbialgebras}.

\begin{maintheorem}\label{mainthmC}
Let \(A\) be a finite-dimensional simple associative \(\Bbbk\)-algebra over an algebraically closed field \(\Bbbk\) of characteristic 0, i.e.\ \(A \cong M_n(\Bbbk)\), and \((A[\![z]\!],\delta)\) be a  topological associative \(D\)-bialgebra such that \(D(A[\![z]\!],\delta) \ncong D(A[\![z]\!],0)\).

Then, up to isomorphism, \(\delta = \delta_r\) where \(r\) is either the rational or the quasi-rational solution of the \(A\)-CYBE determined by an associative Stolin pair \((S,B)\) of class \(k \in \overline{0,n-1}\); see Section \ref{sec:rational_A_CYBE_soltuions} and Section \ref{sec:qrational_A_CYBE_soltuions} for details.

In particular, every solution of the \(A\)-CYBE \eqref{eq:intro_A_CYBE} of the form \eqref{eq:intro_standard_form} is, up to equivalence, of one of the above forms.\qed 
\end{maintheorem}

\noindent
Let us remark that meromorphic solutions of the \(\textnormal{M}_n(\bC)\)-CYBE which depend on the difference of their variables and have diagonal residue \(\gamma\) were already shown to be rational up to equivalence in \cite[Theorem 0.2]{polishchuk_trigonometric}. 

We were unable to provide examples of quasi-trigonometric and trigonometric solutions of the \(A\)-CYBE for unital non-associative strongly geometrically admissible algebras \(A\) as well. We conjecture that, similar to the associative case in Theorem \ref{mainthmC}, the unitality obstructs the existence of these solutions. In particular, this would apply to the case that \(A\) is Jordan. We conclude this work with a brief outlook to the classification problem in the Jordan case in Section \ref{sec:outlook_jordan}.

\subsection*{Acknowledgments} I thank Ivan Shestakov for explaining several facts about non-associative algebras to me. I also want to thank the anonymous referee for reading this paper carefully and for proposing further interesting research perspectives, which were partially included into the outlook in Section \ref{sec:outlook_jordan}. This work was supported by the DFG grant AB 940/1--1. It was also supported as a part of NCCR SwissMAP, a National
Centre of Competence in Research, funded by the Swiss
National Science Foundation (grant number 205607).

\subsection*{Notation and conventions} If the reader is unsure about the meaning of symbols or names which are ambiguously used in literature, we refer to the Appendix \ref{sec:notation}. There we have tried to give an overview on our conventions.

\section{Introduction to \(D\)-bialgebras} \label{sec:intro_D_bialgebras}

\subsection{Survey on Manin triples} \label{sec:survey_manin_triples}
Throughout this paper \(\Bbbk\) is a field and all algebras, vector spaces, tensor products, etc.\ are understood over \(\Bbbk\) if not stated otherwise. Let us remark that for us an \(R\)-algebra over an (unital, commutative, associative) ring \(R\) satisfies no additional assumptions: an \(R\)-algebra \(A = (A,\mu)\) consists of an \(R\)-module \(A\) equipped with an \(R\)-linear map \(\mu\colon A \otimes_R A \to A\), called multiplication map. We write \(ab \coloneqq \mu(a\otimes b)\) for \(a,b\in A\) if no confusion arises.  

\subsubsection{Metric algebras}\label{sec:metric_algebras}
Let \(R\) be a ring and \(A\) be an \(R\)-algebra. We call a map \(\beta \colon A \times A \to R\) \emph{algebra metric} if it is non-degenerate, symmetric, associative, and \(R\)-bilinear. Thereby, ``non-degenerate'' means that the canonical map \(A \to \textnormal{Hom}_R(A,R)\) defined by \(a \mapsto \beta(a,-)\) is injective and ``associative'' means that
\begin{equation}\label{eq:associativity_of_algebra_metric}
    \beta(ab,c) = \beta(a,bc)
\end{equation}
holds for all \(a,b,c \in A\). We call a pair \((A,\beta)\) \emph{metric \(R\)-algebra} if \(A\) is an \(R\)-algebra equipped with an algebra metric \(\beta \colon A \times A \to R\). Moreover, two metric algebras \((A_1,\beta_1)\) and \((A_2,\beta_2)\) are called \emph{isomorphic}, written \((A_1,\beta_1) \cong (A_2,\beta_2)\), if there exists an \(R\)-algebra isomorphism \(\varphi \colon A_1 \to A_2\) such that \(\beta_2(\varphi(a),\varphi(b)) = \beta_1(a,b)\) for all \(a,b\in A_1\).

\subsubsection{Manin pairs and Manin triples}\label{sec:manin_triples} 
A \emph{Manin pair} \(((M,\beta),N)\) consists of a metric \(\Bbbk\)-algebra \((M,\beta)\) and a Lagrangian subalgebra \(N \subseteq M\). In other words, \(N^\bot = N \subseteq M\) is a subalgebra. 

A \emph{Manin triple} \(((M,\beta),M_+,M_-)\) consists of a metric \(\Bbbk\)-algebra \((M,\beta)\) and subalgebras \(M_\pm \subseteq M\) such that \(M = M_+ \oplus M_-\) and \(M_\pm \subseteq M_\pm^\bot\).
It is easy to see that for any Manin triple \(((M,B),M_+,M_-)\), \(M_\pm^\bot = M_\pm\) already holds. In particular, \(((M,\beta),M_+)\) and \(((M,\beta),M_-)\) are automatically Manin pairs.

\begin{remark}
In literature, Manin pairs and triples are usually only defined for Lie algebras. The definition given here is a straight-forward generalization to arbitrary algebras. 
\end{remark}

\subsubsection{Manin triples and comultiplication maps}\label{sec:manin_triples_and_comultiplication}
Recall that a \emph{\(\Bbbk\)-coalgebra} \(C\) is a \(\Bbbk\)-vector space equipped with a \(\Bbbk\)-linear map \(\delta \colon C \to C \otimes C\), called \emph{comultiplication map}. The restriction of \(\delta^* \colon (C \otimes C)^* \to C^*\) to \(C^* \otimes C^* \subseteq (C \otimes C)^*\) always defines \(\Bbbk\)-algebra structure on \(C^*\), hence the name. Explicitly, the multiplication \(f_1f_2 \in C^*\) of two maps \(f_1,f_2 \in C^*\) is defined by
\begin{equation}
    f_1 f_2(a) \coloneqq (f_1 \otimes f_2)\delta(a)
\end{equation}
for all \(a \in C\).

Let us note that for an infinite-dimensional algebra \(A\), \(A^*\) is not necessarily a coalgebra, since the dual \(A^* \to (A\otimes A)^*\) of the multiplication map might fail to have values in \(A^* \otimes A^* \subseteq (A \otimes A)^*\). 

Two \(\Bbbk\)-coalgebras \((C_1,\delta_1)\) and \((C_2,\delta_2)\) are called \emph{isomorphic}, written \(C_1\cong C_2\), if there exists an isomorphism \(\varphi \colon C_1 \to C_2\) of vector spaces satisfying \((\varphi \otimes \varphi)\delta_1 = \delta_2\varphi\).

We say that a comultiplication map \(\delta \colon M_+ \to M_+ \otimes M_+\) is \emph{determined} by a Manin triple \(((M,\beta),M_+,M_-)\) if 
\begin{equation}\label{eq:delta_B_compatibility}
    \beta^{\otimes 2}(\delta(a),b_1 \otimes b_2) = B(a,b_1b_2)
\end{equation}
holds for all \(a \in M_+\) and \(b_1,b_2 \in M_-\). Here, \(\beta^{\otimes 2}(a_1\otimes a_2,b_1\otimes b_2) \coloneqq \beta(a_1,b_1)\beta(a_2,b_2)\). The name stems from the fact that, if \(\widetilde{\delta} \colon M_+ \to M_+ \otimes M_+\) is another comultiplication determined by \(((M,\beta),M_+,M_-)\), we have \(\delta = \widetilde{\delta}\).

\subsubsection{Isomorphism of Manin pairs and Manin triples} \label{sec:isomorphism_of_manin_triples}
We call two Manin pairs (resp.\ Manin triples) 
\begin{equation*}
    ((M_1,\beta_1), N_1) \textnormal{ and }((M_2,\beta_2),N_2)\qquad (\textnormal{resp. } ((M_1,\beta_1),M_{1,+},M_{1,-})\textnormal{ and }((M_2,\beta_2),M_{2,+},M_{2,-}))   
\end{equation*}
\emph{isomorphic} if there exists an isomorphism \(\varphi \colon (M_1,\beta_1) \to (M_2,\beta_2)\) of metric algebras such that \(\varphi(N_1) = N_2\) (resp.\ \(\varphi(M_{1,\pm}) = M_{2,\pm}\)). In this case, we write \(((M_1,\beta_1), N_1) \cong ((M_2,\beta_2),N_2)\) (resp.\ \(((M_1,\beta_1),M_{1,+},M_{1,-}) \cong ((M_2,{\beta_2}),M_{2,+},M_{2,-})\)).

Assume that \(((M_i,\beta_i),M_{i,+},M_{i,-})\) determines a comultiplication \(\delta_i\) on \(M_{i,+}\) for \(i \in \{1,2\}\). Then \(((M_1,\beta_1),M_{1,+},M_{1,-}) \cong ((M_2,\beta_2),M_{2,+},M_{2,-})\) via an isomorphism \(\varphi \colon (M_1,\beta_1) \to (M_2,{\beta_2})\) implies that
\begin{align*}
    &{\beta}_2((\varphi \otimes \varphi)\delta_1(a),b_1 \otimes b_2) = \beta_1(\delta_1(a),\varphi^{-1}(b_1)\otimes \varphi^{-1}(b_2)) = \beta_1(a,\varphi^{-1}(b_1)\varphi^{-1}(b_2) ) \\&= \beta_1(a,\varphi^{-1}(b_1b_2)) = {\beta}_2(\varphi(a),b_1b_2) = {\beta}_2({\delta}_2(\varphi(a)),b_1\otimes b_2)
\end{align*}
holds for all \(a \in M_{1,+}\) and \(b_1,b_2 \in {M}_{2,-}\). Consequently, \((\varphi \otimes \varphi)\delta_1 = {\delta}_2\varphi\) holds, so \(M_{1,+}\cong {M}_{2,+}\) holds both as algebras and coalgebras.

\subsection{D-bialgebras}\label{sec:D_bialgebras} 
Let us call a pair \((A,\delta)\) consisting of a \(\Bbbk\)-algebra \(A\) and a comultiplication map \(\delta \colon A \to A \otimes A\) \emph{bialgebra}. In particular, we do not assume any compatibility conditions between multiplication and comultiplication of \(A\) in this definition.

To any bialgebra \((A,\delta)\),
there is a unique \(\Bbbk\)-algebra structure on \(D(A,\delta) \coloneqq A \oplus A^*\) such that \(((D(A,\delta),\textnormal{ev}),A,A^*)\) is a Manin triple determining \(\delta\), where:
\begin{itemize}
    \item The multiplication of \(A^{*}\) is defined by the comultiplication \(\delta\) in the sense of Subsection \ref{sec:manin_triples_and_comultiplication};
    
    \item \(\textnormal{ev}\colon D(A,\delta) \times D(A,\delta) \to \Bbbk\) is the \emph{evaluation pairing}    \begin{equation}\label{eq:evaluation_pairing} \textnormal{ev}(a + f,b + g) = f(b) + g(a)\,,\qquad a,b \in A \textnormal{ and }f,g \in A^*.
    \end{equation}
\end{itemize}
Explicitly, the multiplication on \( D(A,\delta)\) is determined by: 
\begin{itemize}
    \item \(A,A^* \subseteq D(A,\delta)\) are subalgebras;

    \item The identities
    \begin{equation}\label{eq:double_multiplication_derivation}
    \begin{split}
        &\textnormal{ev}(a f, b) = \textnormal{ev}(f,ba) = f(ba) = \textnormal{ev}(f R_a,b) \\&\textnormal{ev}(a f, g) = \textnormal{ev}(a, fg) = (f \otimes g)\delta(a) = \textnormal{ev}((f \otimes 1)\delta(a),g)
    \end{split}
\end{equation}
for \(a,b \in A, f,g \in A^*\) yield \(af = (f \otimes 1)\delta(a) + f R_a\) and similarly \(fa = (1 \otimes f)\delta(a) + f L_a\) holds.
\end{itemize}
Here, \(R, L \colon A \to \textnormal{End}(A)\) denote the right and left multiplication maps respectively, i.e.
\begin{equation}\label{eq:let_right_multiplication}
    R_ab = ba = L_ba 
\end{equation}
for all \(a,b \in A\).

The \(\Bbbk\)-algebra \(D(A,\delta)\) associated to a bialgebra \((A,\delta)\) is called called \emph{classical double} of \((A,\delta)\).

Let \({\tt Alg}_\Bbbk\) be the category of \(\Bbbk\)-algebras, i.e.\ the category with \(\Bbbk\)-algebras as objects and \(\Bbbk\)-algebra homomorphisms as morphisms. Furthermore, let \({\tt C}\) be a full subcategory of \({\tt Alg}_\Bbbk\) closed under taking subalgebras. For instance, \({\tt C}\) can be any subcategory of equation based \(\Bbbk\)-algebras like the category of Lie algebras, associative algebras or Jordan algebras.

We call a bialgebra \((A,\delta)\) \emph{\(D\)-bialgebra in {\tt C}} if \(D(A,\delta)\) is an algebra in {\tt C}.
Observe that, by construction, \(A,A^{*} \in {\tt C}\) and \(((D(A,\delta),\textnormal{ev}),A,A^{*})\) is a Manin triple determining \(\delta\).

Let us point out that \(D\)-bialgebras in  \({\tt Alg}_\Bbbk\) are exactly bialgebras.

\subsubsection{Isomorphism of \(D\)-bialgebras}\label{sec:isomorphism_of_Dbialgebras} Let \({\tt C}\) be a full subcategory of {\tt Alg}\(_\Bbbk\) closed under taking subalgebras. Two \(D\)-bialgebras \((A_1,\delta_1)\) and \((A_2,\delta_2)\) in {\tt C} are called \emph{isomorphic}, written \((A_1,\delta_1) \cong (A_2,\delta_2)\), if there exists a \(\Bbbk\)-linear map \(\varphi \colon A_1 \to A_2\) which is both an isomorphism of \(\Bbbk\)-algebras and \(\Bbbk\)-coalgebras, i.e.\ if the identities
\begin{equation}\label{eq:isomorphism_of_Dbialgebras}
    \varphi(a_1a_2) = \varphi(a_1)\varphi(a_2) \textnormal{ and }(\varphi \otimes \varphi)\delta_1(a) = \delta_2(\varphi(a))
\end{equation}
hold for all \(a,a_1,a_2 \in A_1\) .

\begin{lemma}\label{lem:isomorphism_of_Dbialgebras}
Let \({\tt C}\) be a full subcategory of {\tt Alg}\(_\Bbbk\) closed under taking subalgebras and \((A_1,\delta_1)\),\((A_2,\delta_2)\) be two \(D\)-bialgebras in {\tt C}. Then
\begin{equation*}
    (A_1,\delta_1) \cong (A_2,\delta_2) \iff ((D(A_1,\delta_1),\textnormal{ev}),A_1,A_1^*) \cong ((D(A_2,\delta_2),\textnormal{ev}),A_2,A_2^*).
\end{equation*}
\end{lemma}

\begin{proof}
The fact that \(((D(A_1,\delta_1),\textnormal{ev}),A_1,A_1^*) \cong ((D(A_2,\delta),\textnormal{ev}),A_2,A_2^*)\) implies \((A_1,\delta_1) \cong (A_2,\delta_2)\) was already mentioned in Section \ref{sec:isomorphism_of_manin_triples}. On the other hand, let \(\varphi \colon A_1 \to A_2\) define the isomorphism \((A_1,\delta_1) \cong (A_2,\delta_2)\). It is easy to see that \(\widetilde{\varphi}(a + f) \coloneqq \varphi(a) + f \varphi^{-1}\), where \(a\in A_1, f \in A_1^*\), defines an isomorphism \(((D(A_1,\delta_1),\textnormal{ev}),A_1,A_1^*) \cong ((D(A_2,\delta),\textnormal{ev}),A_2,A_2^*)\).
\end{proof}

\subsection{Examples}\label{sec:survey_manin_triples_examples}
In \cite{zhelyabin}, the \(D\)-bialgebra structures for the most important categories \({\tt C}\) of algebras where discussed. Let us give a short outline of their explicit descriptions. In the following, we write for any elements \(a,a_1,\dots,a_n\) in some \(\Bbbk\)-algebra \(A\)
\begin{equation}
    \begin{split}
        &a^{(i)}(a_1\otimes \dots \otimes a_n) = a_1\otimes \dots \otimes aa_i \otimes \dots \otimes a_n\\&(a_1\otimes \dots \otimes a_n)a^{(i)} = a_1\otimes \dots \otimes a_ia \otimes \dots \otimes a_n.    
    \end{split}
\end{equation}

\subsubsection{\(D\)-bialgebras in the category of Lie algebras}\label{sec:survey_manin_triples_Lie_bialgebras} Recall that a Lie bialgebra \((L,\delta)\) consists of a Lie algebra \(L\) equipped with a linear map \(\delta \colon L \to L \otimes L\) such that:
\begin{itemize}
    \item \(\delta\) is a 1-cocycle, i.e.\ for all \(a,b \in L\)
    \begin{equation*}\label{eq:Lie_cocycle_condition}
        \delta(ab) = (a^{(1)} + a^{(2)})\delta(b) + \delta(a)(b^{(1)} + b^{(2)});
    \end{equation*}
    
    \item The restriction of \(\delta^*\) to \(L^* \otimes L^* \subseteq (L \otimes L)^*\) is a Lie bracket.
\end{itemize}
It is well-known that for a linear map \(\delta \colon L \to L \otimes L\) on a Lie algebra \(L\) the double \(D(L,\delta)\) is again a Lie algebra if and only if \((L,\delta)\) is a Lie bialgebra. Therefore, a \(D\)-bialgebra in the category of Lie algebras is exactly a Lie bialgebra.

\subsubsection{\(D\)-bialgebras in the category of associative algebras}\label{sec:survey_manin_triples_balanced_infinitesimal_bialgebras}
An infinitesimal bialgebra \((A,\delta)\) consists of a Lie algebra \(A\) equipped with a cobracket \(\delta \colon A \to A \otimes A\) such that:
\begin{itemize}
    \item \(\delta\) is a 1-cocycle, i.e.\ for all \(a,b \in A\) 
    \begin{equation*}
        \delta(ab) = a^{(1)}\delta(b) + \delta(a)b^{(2)};
    \end{equation*}
    
    \item The restriction of \(\delta^*\) to \(A^* \otimes A^* \subseteq (A \otimes A)^*\) is an associative multiplication.
\end{itemize}
It was shown by Aguiar \cite{aguiar_associative} that there is classical double like construction for infinitesimal algebras, this time by giving \(\overline{D}(A,\delta) \coloneqq (A \oplus A^*) \oplus (A \otimes A^*)\) an associative algebra structure. In general, these are not related to Manin triples. However, under the condition that \(\delta\) is \emph{balanced}, i.e. if for all \(a_1,a_2 \in A\)
\begin{equation}\label{eq:balance}
    a^{(1)}_1\tau \delta(a_2) + a_2^{(2)}\delta(a_1) = \delta(a_1) a_2^{(1)} + \tau \delta(a_2) a_1^{(2)}
\end{equation}
holds, the double can be reduced to \(\overline{D}_{\textnormal{red}}(A,\delta) = A \oplus A^*\). It turns out that \(\overline{D}_{\textnormal{red}}(A,\delta) = D(A,\tau \delta)\), so \(\tau \delta\) is an associative \(D\)-bialgebra structure, i.e.\ a \(D\)-bialgebra in the category of associative \(\Bbbk\)-algebras. Here and in the following \(\tau(a \otimes b) = b \otimes a\). 

For every bialgebra \((A,\delta)\), we call \((A,\tau \delta)\) the \emph{co-opposite} bialgebra of \((A,\delta)\). It is shown in \cite{zhelyabin} that indeed all associative \(D\)-bialgebra structures are of the above form, i.e.\ the co-opposites of balanced infinitesimal bialgebras are exactly associative \(D\)-bialgebras.

\subsubsection{\(D\)-bialgebras in the category of Jordan algebras}\label{sec:survey_manin_triples_jordan_bialgebras}
The \(D\)-bialgebra structures in the category of Jordan algebras, which we will simply call \emph{Jordan bialgebras}, where found in \cite[Theorem 2]{zhelyabin}: a Jordan bialgebra \((J,\delta)\) consists of a Jordan algebra \(J\) and a linear map \(\delta \colon J \to J \otimes J\) such that \(J^*\) is a Jordan algebra and the following identities hold:

\begin{align*}
    &\frac{1}{2}((\delta \otimes 1) - (1 \otimes \delta))\delta(a^2) = a^{(2)}(\delta \otimes 1 - 1 \otimes \delta)\delta(a) + (a^{(3)} - a^{(1)})(1 \otimes \tau)(\delta \otimes 1)\delta(a) \\&+ (\delta(a) \otimes 1 - 1 \otimes \delta(a))(1\otimes \tau)(\delta(a) \otimes 1);\\ 
    &(\delta \otimes 1+ 1 \otimes \delta + (1 \otimes \tau)(\delta \otimes 1))(1 \otimes a + a \otimes 1)\delta(a) = 2a^{(2)}(1 \otimes \delta)\delta(a) + a^{(1)}(1 \otimes \tau)(\delta \otimes 1)\delta(a) \\&+ (1 \otimes \delta(a))(1 \otimes \tau)(\delta(a) \otimes 1) + (\delta \otimes 1)\delta(a^2); \\
    &\delta(a^2b) - \delta(a^2)b^{(1)} - \delta(b)(a^2)^{(2)} + 2 \delta(b)(a\otimes a) - 2\delta(ab)a^{(1)} + 2 (\delta(a)b^{(1)})a^{(1)} 
    \\&+ 2 (\delta(a)b^{(2)})a^{(2)} - 2\delta(a)(ab)^{(2)} = 0.
\end{align*}

\section{Non-degenerate topological \(D\)-bialgebra structures on power series algebras}\label{sec:manin_triples_over_series}

\subsection{Topological \(D\)-bialgebra structures on power series algebras}\label{sec:survey_topological_bialgebras}
Let \(A\) be a finite-dimensional \(\Bbbk\)-algebra and let us equip
\(A[\![z]\!]\) (resp.\ \((A\otimes A)[\![x,y]\!]\)) with the \((z)\)-adic (resp.\ \((x,y)\)-adic) topology; see Appendix \ref{sec:notation} for definition of \((\cdot)[\![z]\!]\) and \((\cdot)[\![x,y]\!]\). Note that if \(\delta \colon A[\![z]\!] \to (A\otimes A)[\![x,y]\!]\) is a continuous linear map and if \((\cdot)^\vee\) denotes taking the continuous dual space, \(A[\![z]\!]^\vee\) is naturally a \(\Bbbk\)-algebra with the multiplication defined by
\begin{equation}\label{eq:topological_dual_of_delta}
    A[\![z]\!]^\vee \otimes A[\![z]\!]^\vee \cong (A \otimes A)[\![x,y]\!]^\vee \stackrel{\delta^\vee}\longrightarrow A[\![z]\!]^\vee.
\end{equation}
We call a pair \((A[\![z]\!],\delta)\) as above \emph{topological bialgebra}.

For any topological bialgebra \((A[\![z]\!],\delta)\), there is a unique \(\Bbbk\)-algebra structure on \(D(A[\![z]\!],\delta) = A[\![z]\!] \oplus A[\![z]\!]^\vee\) such that \(((D(A[\![z]\!],\delta),\textnormal{ev}),A[\![z]\!],A[\![z]\!]^\vee)\) is a Manin determining \(\delta\). Here, the \emph{evaluation pairing} \(\textnormal{ev}\colon D(A[\![z]\!],\delta) \times D(A[\![z]\!],\delta) \to \Bbbk\) is defined analogous to \eqref{eq:evaluation_pairing}.
The multiplication map of \(D(A,\delta)\) satisfying these conditions can be explicitly determined in the same way as in the non-topological setting in Section \ref{sec:D_bialgebras}. The algebra \(D(A[\![z]\!],\delta)\) is called \emph{classical 
 double} of \((A[\![z]\!],\delta)\).

Let \({\tt C}\) be a full subcategory of \({\tt Alg}_\Bbbk\) closed under taking subalgebras. We call a topological bialgebra \((A[\![z]\!],\delta)\) \emph{topological \(D\)-bialgebra in {\tt C}} if \(D(A[\![z]\!],\delta)\) is an algebra in \({\tt C}\). Observe that if \((A[\![z]\!],\delta)\) is a topological \(D\)-bialgebra in {\tt C}, we have \(A,A[\![z]\!],A[\![z]\!]^\vee \in {\tt C}\).

It is easy to see that \((A[\![z]\!],\delta)\) is a topological \(D\)-bialgebra in {\tt C} if and only if \((A[\![z]\!]^\vee,\mu^\vee)\) is a usual \(D\)-bialgebra in {\tt C}, where \(\mu \colon (A\otimes A)[\![x,y]\!]\to A[\![z]\!]\) is the multiplication map. Therefore, we can describe topological \(D\)-bialgebras in the most important categories of algebras using the same axioms as in Section \ref{sec:survey_manin_triples_examples}.

\subsubsection{Isomorphism of topological \(D\)-bialgebras on series}\label{sec:isomorphism_of_top_Dbialgebras} Let \(A\) be a \(\Bbbk\)-algebra and  \({\tt C}\) be a full subcategory of \({\tt Alg}_\Bbbk\) closed under taking subalgebras and such that \(A[\![z]\!] \in {\tt C}\). Two topological \(D\)-bialgebras \((A_1[\![z]\!],\delta_1)\) and \((A_2[\![z]\!],\delta_2)\) in {\tt C} are called \emph{isomorphic}, written \((A_1[\![z]\!],\delta_1) \cong (A_2[\![z]\!],\delta_2)\), if there exists a continuous linear map \(\varphi \colon A_1[\![z]\!] \to A_2[\![z]\!]\) which is both an isomorphism of algebras and coalgebras, i.e.\ if for every \(a,a_1,a_2 \in A_1\) the identities
\begin{equation}\label{eq:isomorphism_of_top_Dbialgebras}
    \varphi(a_1a_2) = \varphi(a_1)\varphi(a_2) \textnormal{ and }(\varphi \otimes \varphi)\delta_1(a) = \delta_2(\varphi(a))
\end{equation}
hold. Here, in the latter equation \(\varphi \otimes \varphi\) was continuously extended from an automorphism of \(A[\![z]\!]\otimes A[\![z]\!]\) to an automorphism of \((A\otimes A)[\![x,y]\!]\).

\begin{lemma}\label{lem:isomorphism_of_top_Dbialgebras}
Let \(A\) be a \(\Bbbk\)-algebra and  \({\tt C}\) be a full subcategory of \({\tt Alg}_\Bbbk\) closed under taking subalgebras and \((A_1[\![z]\!],\delta_1)\),\((A_2[\![z]\!],\delta_2)\) be two \(D\)-bialgebras in {\tt C}. Then \((A_1[\![z]\!],\delta_1) \cong (A_2[\![z]\!],\delta_2)\) if and only if
\[((D(A_1[\![z]\!],\delta_1),\textnormal{ev}),A_1[\![z]\!],A_1[\![z]\!]^\vee) \cong ((D(A_2[\![z]\!],\delta_2),\textnormal{ev}),A_2[\![z]\!],A_2[\![z]\!]^\vee)\]
via an isomorphism \(D(A_1[\![z]\!],\delta_1) \cong D(A_2[\![z]\!],\delta_2)\) which restricts to a continuous isomorphism \(A_1[\![z]\!] \to A_2[\![z]\!]\).
\end{lemma}

\begin{proof}
Repeat the arguments in the proof of Lemma \ref{lem:isomorphism_of_top_Dbialgebras} under consideration of the fact that any continuous linear isomorphism \(A_1[\![z]\!] \to A_2[\![z]\!]\) has a continuous inverse, since \(A_1[\![z]\!]\) is linearly compact.
\end{proof}

\begin{remark}\label{rem:isomorphism_of_top_Dbialgebras}
Let \(A\) be a finite-dimensional central simple \(\Bbbk\)-algebra. Then \cite[Theorem 3.3]{abedin_maximox_stolin_zelmanov} states that for every \(\phi \in \textnormal{Aut}_{\Bbbk\textnormal{-alg}}(A[\![z]\!])\) exists a \(\varphi \in \textnormal{Aut}_{\Bbbk[\![z]\!]\textnormal{-alg}}(A[\![z]\!]) \subseteq \textnormal{End}(A)[\![z]\!]\) and \(u \in z\Bbbk[\![z]\!]^\times\) such that
\begin{equation}
    \phi(a)(z) = \varphi(z)a(u(z))
\end{equation}
for all \(a \in A[\![z]\!]\). As a consequence, every \(\Bbbk\)-algebra automorphism of \(A[\![z]\!]\) is continuous in the \((z)\)-adic topology.
In particular, in this case Lemma \ref{lem:isomorphism_of_top_Dbialgebras} can be refined for \(A = A_1 = A_2\) as 
\begin{equation*}
    (A[\![z]\!],\delta_1) \cong (A[\![z]\!],\delta_2) \iff ((D(A[\![z]\!],\delta_1),\textnormal{ev}),A[\![z]\!],A[\![z]\!]^\vee) \cong ((D(A[\![z]\!],\delta_2),\textnormal{ev}),A[\![z]\!],A[\![z]\!]^\vee).
\end{equation*}
\end{remark}

\subsection{Non-degenerate topological \(D\)-bialgebra structures}\label{sec:manin_triples_over_series_explicit} 
Consider Manin triples of the form
\begin{equation}\label{eq:manin_triples_over_series_explicit} 
    ((D_n(A),\beta_{(n,\lambda)}),A[\![z]\!],W),
\end{equation}
where:
\begin{itemize}
    \item \((A,\beta)\) is a finite-dimensional metric \(\Bbbk\)-algebra (recall the definition from Section \ref{sec:survey_manin_triples});
    
    \item \( n \in \bN\) and \(D_n(A) \coloneqq A(\!(z)\!) \times A[z]/z^nA[z]\);
    
    \item \(A[\![z]\!]\) is identified with the image of the diagonal embedding \(A[\![z]\!] \to D_n(A)\) defined by 
    \[a \longmapsto (a,[a]).\]
    Here, for any \(a \in A[\![z]\!]\), \([a] \coloneqq a + z^nA[\![z]\!] \in A[\![z]\!]/z^nA[\![z]\!] = A[z]/z^nA[z]\).
    \item \(\lambda \in \Bbbk[\![z]\!]^\times\) and \(\beta_{(n,\lambda)}\) is given by 
    \begin{equation}\label{eq:beta_nlambda}
        \beta_{(n,\lambda)}((a_1,[a_2]),(b_1,[b_2])) = \textnormal{res}_0 \frac{1}{z^n\lambda }( \beta(a_1,b_1)- \beta(a_2,b_2)),
    \end{equation}
    where \(\beta\) was extended to a \(\Bbbk(\!(z)\!)\)-bilinear form \(A(\!(z)\!)\times A(\!(z)\!) \to \Bbbk(\!(z)\!)\) on the right-hand side. 
\end{itemize}

\smallskip
\noindent
It is easy to see that all triples of the form \eqref{eq:manin_triples_over_series_explicit} are indeed Manin triples in the sense of Definition \ref{sec:manin_triples}.

Let \({\tt C}\) be a full subcategory of \({\tt Alg}_\Bbbk\) that is closed under taking subalgebras.
We call a topological \(D\)-bialgebra \((A[\![z]\!],\delta)\) in \({\tt C}\) \emph{non-degenerate} if and only if there exist \(n \in \bN\) and \(\lambda \in \Bbbk[\![z]\!]^\times\) such that \(((D(A[\![z]\!],\delta),\textnormal{ev}),A[\![z]\!]) \cong ((D_n(A),\beta_{(n,\lambda)}),A[\![z]\!])\) as Manin pairs (see Section \ref{sec:isomorphism_of_manin_triples}). In other words, \((A[\![z]\!],\delta)\) is non-degenerate if and only if there exist \(n \in \bN\) and \(\lambda \in \Bbbk[\![z]\!]^\times\) such that \(D_n(A) \in {\tt C}\) and \(\delta\) is determined by the Manin triple \(((D_n(A),\beta_{(n,\lambda)}),A[\![z]\!],W)\) for an appropriate \(W \subseteq D_n(A)\).

Remark \ref{rem:isomorphism_of_top_Dbialgebras} implies that, if \(A\) is central and simple. the classification of non-degenerate topological \(D\)-bialgebra structures on \(A[\![z]\!]\) up to isomorphisms of topological \(D\)-bialgebras is equivalent to the classification of Manin triples of the form \eqref{eq:manin_triples_over_series_explicit} up to isomorphisms of Manin triples.

\subsection{Connection to trace extensions of \(\Bbbk[\![z]\!]\)}\label{sec:trace_extension}
A \emph{trace extension} \((R,t)\) of \(\Bbbk[\![z]\!]\) consists of a commutative and associative \(\Bbbk\)-algebra extension \(R \supseteq \Bbbk[\![z]\!]\) equipped with a linear map \(t \colon R \to \Bbbk\), called \emph{trace map}, such that:
\begin{enumerate}
    \item \((a,b) \mapsto \beta_t(a,b) \coloneqq t(ab)\) is an algebra metric making \(\Bbbk[\![z]\!]\subseteq R\) a Lagrangian subalgebra;
    
    \item For all continuous (in the \((z)\)-adic topology) linear maps \(f \colon \Bbbk[\![z]\!]\to \Bbbk\) exists an \(a \in R\) such that \(f(b) = t(ab)\) for all \(b \in \Bbbk[\![z]\!]\).
\end{enumerate}
In other words, trace extensions \((R,t)\) are in bijection with Manin pairs \(((R,\beta_t),\Bbbk[\![z]\!])\) for which \(R\) is associative and commutative and \(\beta_t\) satisfies (2).

Two trace extensions \((R_1,t_1)\) and \((R_2,t_2)\) are called \emph{isomorphic}, written \((R_1,t_1) \cong (R_2,t_2)\), if their associated Manin pairs are isomorphic. In other words, \((R_1,t_1) \cong (R_2,t_2)\) if there exists an algebra isomorphism \(\varphi \colon R_1 \to R_2\) such that \(\varphi(\Bbbk[\![z]\!]) = \Bbbk[\![z]\!]\) and \(t_2 \varphi = t_1\). Observe that \(\varphi|_{\Bbbk[\![z]\!]}\) is automatically continuous.

Trace extensions were classified up to isomorphism in \cite[Proposition 2.9]{montaner_stolin_zelmanov}:

\begin{proposition}\label{prop:trace_extensions}
Let \((R,t)\) be a trace extension of \(\Bbbk[\![z]\!]\). Then precisely one of the following cases occurs:
\begin{enumerate}
    \item \((R,t) \cong (R_\infty, t_\infty)\), where \(R_{\infty} \coloneqq \Bbbk[\![z]\!] \oplus \textnormal{Span}_{\Bbbk}\{a_k\mid k \in \bN\}\) with multiplication defined by \(a_ja_k = 0\) and
    \begin{equation*}
        a_jz^k = \begin{cases}a_{j-k}&\textnormal{if }k \le j,\\
        0&\textnormal{otherwise}
        \end{cases}
    \end{equation*}
    and \(t_\infty\) is the unique trace map on \(R_\infty\) defined by \(t(a_{j}) = \delta_{j0}\) for \(j \in \bN\).
    
    \item There exists \(n \in \bN\) and \(\lambda \in \Bbbk[\![z]\!]^\times\) such that \((R,t) \cong (R_n,t_{(n,\lambda)})\), where \(R_n \coloneqq \Bbbk(\!(z)\!) \times \Bbbk[z]/(z^n)\) and
    \begin{equation*}
        t_{(n,\lambda)}(a,[b]) \coloneqq \textnormal{res}_0\frac{1}{z^n\lambda}(a - b).
    \end{equation*}
    Here, \(a \in \Bbbk(\!(z)\!), b\in \Bbbk[\![z]\!]\), \([b] \coloneqq b + z^n\Bbbk[z] \in \Bbbk[\![z]\!]/z^n\Bbbk[\![z]\!]\) and \(\Bbbk[\![z]\!]\) is identified with its image via the embedding \(a \mapsto (a, [a])\). \qed
\end{enumerate}
\end{proposition}
A trace extension \((R,t)\) of \(\Bbbk[\![z]\!]\) is called \emph{trivial} if \((R,t) \cong (R_\infty,t_\infty)\).

If \((A,\beta)\) is a finite-dimensional metric \(\Bbbk\)-algebra and \((R,t)\) is a trace extension of \(\Bbbk[\![z]\!]\), then \((A \otimes R,\beta \otimes t)\) is a metric \(\Bbbk\)-algebra. Here,
\begin{equation}
    (\beta \otimes t)(a_1 \otimes b_1,a_2 \otimes b_2) \coloneqq \beta(a_1,a_2)t(b_1b_2)
\end{equation}
for all \(a_1,a_2 \in A, b_1,b_2 \in R\). 

Observe that \(((A\otimes R_n,\beta \otimes t_{(n,\lambda)}),A\otimes \Bbbk[\![z]\!]) \cong ((D_n(A),\beta_{(n,\lambda)}),A[\![z]\!])\) holds for all \(n \in \bN_0\) and \(\lambda \in \Bbbk[\![z]\!]^\times\). 
Therefore, Proposition \ref{prop:trace_extensions} states that, if \((R,t)\) is a non-trivial trace extension of \(\Bbbk[\![z]\!]\), there exists \(n \in \bN\) and \(\lambda \in \Bbbk[\![z]\!]^\times\) such that \(((A\otimes R,\beta \otimes t),A\otimes \Bbbk[\![z]\!]) \cong ((D_n(A),\beta_{(n,\lambda)}),A[\![z]\!])\) as Manin pairs. In particular, the Manin triples considered in Section \ref{sec:manin_triples_over_series_explicit} are exactly those which arise by finding Lagrangian subalgebras \(W\) in \((A \otimes R,\beta \otimes t)\) complementary to \(A \otimes \Bbbk[\![z]\!]\) for any non-trivial trace extension \((R,t)\).

\subsection{Non-triangular topological Lie \(D\)-bialgebra structures are non-degenerate} 
Let \(A\) be a finite-dimensional \(\Bbbk\)-algebra and {\tt C} be a full subcategory of {\tt Alg}\(_\Bbbk\) closed under taking subalgebras such that \(A \otimes R_\infty \in {\tt C}\). Then the zero map
\[\delta = 0 \colon A[\![z]\!] \to (A\otimes A)[\![x,y]\!]\] 
defines a topological \(D\)-bialgebra structure in {\tt C} with double \((D(A[\![z]\!],\delta),\textnormal{ev}) \cong (A \otimes R_\infty,\beta \otimes t_\infty)\).
We say that a topological \(D\)-bialgebra structure \(\delta \colon A[\![z]\!] \to (A\otimes A)[\![x,y]\!]\) is \emph{triangular} if
\begin{equation}
    ((D(A[\![z]\!],\delta),\textnormal{ev}),A[\![z]\!]) \cong ((A \otimes R_\infty,\beta \otimes t_\infty),A\otimes \Bbbk[\![z]\!]).
\end{equation}
The origin of the name is explained in Remark \ref{rem:triangular_bialgebras_and_rmatrices}.

Let \(\Bbbk\) be algebraically closed of characteristic 0. If {\tt C} is the category of Lie algebras over \(\Bbbk\) and \(A = \fg \in {\tt C}\) is simple, it was shown in \cite[Corollary 2.2, Lemma 2.3 and Proposition 2.8]{montaner_stolin_zelmanov} (see also \cite[Corollary 3.10]{abedin_maximox_stolin_zelmanov}) that any non-triangular topological Lie bialgebra \((\fg[\![z]\!],\delta)\) is non-degenerate in the sense of Section \ref{sec:manin_triples_over_series_explicit}. 

We will prove an analog of this results for the case that {\tt C} is the category of associative algebras in Section \ref{sec:associative_topological_Dbialgebras_nondegenerate}.

\section{Categorization of non-degenerate \(D\)-bialgebra structures}\label{sec:general_categorization}
In this section, we will show that, up to isomorphism of \(D\)-bialgebras and for a large class of central simple \(\Bbbk\)-algebras \(A\), all non-degenerate topological \(D\)-bialgebras \((A[\![z]\!],\delta)\)
are determined by a Manin triple of the form \eqref{eq:manin_triples_over_series_explicit} for some \(n \in \{0,1,2\}\) and \(\lambda = 1\). The main method we use to prove this result is the geometrization of \(A\)-lattices developed in \cite[Section 2.3]{abedin_universal_geometrization} (see also \cite[Section 1.3]{abedin_thesis}), which we will recall in Subsection \ref{sec:geometrization}. The precise formulation of the above mentioned result is then given in Subsection \ref{sec:geometrically_admissible_metrics} and the reminder of this section will be devoted to its proof.

Throughout the reminder of this paper, we assume that \(\Bbbk\) is a field of characteristic 0.

\subsection{Geometrization of lattices.} \label{sec:geometrization}
Let \(A\) be a finite-dimensional, central, simple \(\Bbbk\)-algebra.
We call a subalgebra \(W \subseteq A(\!(z)\!)\) satisfying
\begin{equation}\label{eq:lattice}
    \dim(A[\![z]\!] \cap W) < \infty \textnormal{ and }\dim(A(\!(z)\!)/(A[\![z]\!] + W)) < \infty
\end{equation}
\emph{\(A\)-lattice}. Furthermore, we call a pair \((O,W)\) consisting of an \(A\)-lattice \(W\subseteq A(\!(z)\!)\) and a unital subalgebra \(O \subseteq \{f \in \Bbbk(\!(z)\!)\mid fW \subseteq W\}\) of finite codimension \emph{ringed \(A\)-lattice}.

Let us fix a ringed \(A\)-lattice \((O,W)\). The graded \(\Bbbk\)-algebra
\begin{equation}
    \textnormal{gr}(O) \coloneqq \bigoplus_{j = 0}^\infty t^j\left(O\cap z^{-j}\Bbbk[\![z]\!]\right) \subseteq O[t]
\end{equation}
defines an irreducible projective curve \(X \coloneqq \textnormal{Proj}(\textnormal{gr}(O))\) over \(\Bbbk\) of arithmetic genus
\begin{equation}\label{eq:arithmetic_genus}
    \textnormal{h}^1(\sheafO_X) = \dim(\Bbbk(\!(z)\!)/(\Bbbk[\![z]\!] + O)).
\end{equation}
The \(\Bbbk\)-rational smooth point \(p = (t)\) of \(X\) satisfies \(\textnormal{D}_+(t) = X\setminus\{p\}\). Furthermore, there is canonical isomorphism \(c \colon \compO_{X,p} \to \Bbbk[\![z]\!]\) such that the induced isomorphism \( \textnormal{Q}(\compO_{X,p}) \to \Bbbk(\!(z)\!)\) on quotient fields, which will be denoted again by \(c\), has the property \(c(\Gamma(X\setminus \{p\},\mathcal{O}_X)) = O\).

Consider the graded \(\textnormal{gr}(O)\)-algebra 
\begin{equation}
    \textnormal{gr}(W) \coloneqq \bigoplus_{j \in \bZ} t^j(W \cap z^{-j}A[\![z]\!]) \subseteq W[t,t^{-1}]
\end{equation}
defined by \(W\). Then the quasi-coherent sheaf \(\mathcal{A}\) on \(X = \textnormal{Proj}(\textnormal{gr}(O))\) associated to \(\textnormal{gr}(W)\) is a coherent torsion-free \(\mathcal{O}_X\)-algebra. This sheaf comes equipped with an \(c\)-equivariant isomorphism \(\zeta \colon \compA_p \to A[\![z]\!]\) such that the induced isomorphism \(\textnormal{Q}(\compA_{p}) \to A(\!(z)\!)\), which will be denoted again by \(\zeta\), has the property \(\zeta(\Gamma(X\setminus \{p\},\mathcal{A})) = W\). The dimensions of the cohomology of \(\sheafA\) can be calculated by
\begin{equation}\label{eq:cohomology}
    \textnormal{h}^0(\sheafA) = \dim(A[\![z]\!] \cap W) \textnormal{ and }\textnormal{h}^1(\sheafA) = \dim(A(\!(z)\!)/(A[\![z]\!] + W)).
\end{equation}

\subsection{Geometrically admissible algebra metrics and the main theorem.}\label{sec:geometrically_admissible_metrics}
Let \((A,\beta)\) be a metric \(\Bbbk\)-algebra and let us denote the \(\Bbbk(\!(z)\!)\)-bilinear extension of \(\beta\) by the same symbol, i.e.\
\begin{equation}\label{eq:beta_extension}
\beta \colon A(\!(z)\!) \times A(\!(z)\!) \to \Bbbk(\!(z)\!)\,,\qquad \left(\sum_{k \in \bZ}a_kz^k,\sum_{k \in \bZ}b_kz^k\right) \longmapsto \sum_{k,\ell \in \bZ} \beta(a_k,b_\ell)z^{k+\ell}.
\end{equation}
We call \((A,\beta)\) \emph{geometrically admissible} if:
\begin{enumerate}
    \item \(A\) is finite-dimensional, central, and simple;

    \item For any ringed \(A\)-lattices \((O,W)\) and any maximal ideal \(\mathfrak{m} \subseteq O\) such that \(W_{\mathfrak{m}}\) is free as \(O_{\mathfrak{m}}\)-module, we have \(\beta(W_{\mathfrak{m}},W_{\mathfrak{m}}) \subseteq O_{\mathfrak{m}}\). 
\end{enumerate}

\subsubsection{Examples} \label{sec:geom_admissible_metrics_examples}
Let \((A,\beta)\) be a finite-dimensional, central, simple, metric \(\Bbbk\)-algebra, \((O,W)\) be a ringed \(A\)-lattice, and \(\mathfrak{m} \subseteq O\) be a regular maximal ideal of \(O\) such that \(W_{\mathfrak{m}}\) is a free \(O_{\mathfrak{m}}\)-module. Then \(W_{\mathfrak{m}}\) is of rank \(d \coloneqq \dim(A)\), so we can choose an \(O_{\mathfrak{m}}\)-basis \(\{b_i\}_{i = 1}^d \subseteq W_{\mathfrak{m}}\) and write \(b_ib_j = \sum_{k = 1}^dC_{ij}^kb_k\) for \(\{C_{ij}^k\}_{i,j,k = 1}^d \subseteq O_{\mathfrak{m}}\). Observe that \(\{b_i\}_{i = 1}^d \subseteq W_{\mathfrak{m}} \subseteq A(\!(z)\!)\) is also a \(\Bbbk(\!(z)\!)\)-basis of \(A(\!(z)\!)\).
\begin{enumerate}
    \item Assume \(A\) is a Lie algebra. Then \(\beta\) is a scalar multiple of the Killing form of \(A\) since \(A\) is simple. As a consequence, the extension \eqref{eq:beta_extension} of \(\beta\) is equal to \(\lambda K\) for the Killing form \(K\) of \(A(\!(z)\!)\) and some \(\lambda \in \Bbbk^\times\). Therefore,
    \begin{equation*}
        \beta(b_i,b_j) = \lambda C_{ik}^\ell C_{j\ell}^k \in O_{\mathfrak{m}}
    \end{equation*}
    holds, so \(\beta\) is geometrically admissible.
    
    \item Assume that \(A\) is power associative and not anti-commutative, e.g.\ if \(A\) is associative or Jordan. Then the existence of an algebra metric \(\beta\) on \(A\) implies that \(A\) is a non-commutative Jordan algebra; see e.g.\ \cite[Kapitel I, Satz 6.5]{braun_koecher}. Moreover, \cite[Theorem 1]{shestakov} implies that \(A\) is not nil, so there exists \(\lambda \in \Bbbk^\times\) such that
    \begin{equation*}
        \beta(a,b) = \frac{\lambda}{2}(\textnormal{Tr}(R_{ab}) + \textnormal{Tr}(L_{ab})),
    \end{equation*}
    for all \(a,b\in A\); see e.g.\ \cite{schafer}. Here, \(R,L \colon A \to \textnormal{End}(A)\) are the right and left multiplication maps respectively. Therefore,
    \begin{equation*}
        \beta(b_i,b_j) = \frac{\lambda}{2}\sum_{k,\ell = 1}^d C_{ij}^\ell (C_{k\ell}^k + C_{\ell k}^k) \in O
    \end{equation*}
    holds, so \(\beta\) is geometrically admissible.
\end{enumerate}

\subsubsection{Geometrically admissible metrics and geometrization of lattices}\label{sec:properties_geometriaclly_admissible_pairing}
Let \((A,\beta)\) be a geometrically admissible metric \(\Bbbk\)-algebra and \((O,W)\) be a ringed \(A\)-lattice. The following results are true:
\begin{enumerate}
\item For all regular maximal ideals \(\mathfrak{m} \subseteq O\), we have \(\beta(W_\mathfrak{m},W_\mathfrak{m}) \subseteq O_{\mathfrak{m}}\).

\item Let \(N\) be the integral closure of \(O\). Then \(N\) can be understood as a subalgebra of \(\Bbbk(\!(z)\!)\) and \(V \coloneqq NW \subseteq A(\!(z)\!)\) is an \(A\)-lattice.
Consider the geometric datum \(((X,\sheafA),(p,z,\zeta))\) constructed from the ringed \(A\)-lattice \((N,V)\) in Section \ref{sec:geometrization}. Then there exists a unique pairing
\(\beta_\sheafA \colon \sheafA \times \sheafA \to \sheafO_X\) such that 
\begin{equation}\label{eq:beta_A}
    \xymatrix{\Gamma(U,\sheafA) \times \Gamma(U,\sheafA) \ar[r]^-{\beta_\sheafA}\ar[d]_{\zeta \times \zeta}& \Gamma(U,\sheafO_X)\ar[d]^c\\A(\!(z)\!) \times A(\!(z)\!) \ar[r]_-{\beta} &\Bbbk(\!(z)\!)
    }
\end{equation}
commutes for all \(U \subseteq X\) open. 

\item The pairing \(\beta_\sheafA\) gives rise to a short exact sequence
\begin{equation*}
    0 \longrightarrow \sheafA \longrightarrow \sheafA^* \longrightarrow \mathcal{C} \longrightarrow 0
\end{equation*}
for a torsion sheaf \(\mathcal{C}\). Here, \(\sheafA^* = \sheafHom_{\sheafO_X}(\sheafA,\sheafO_X)\) is the sheaf of morphisms from \(\sheafA\) to \(\sheafO_X\).
\end{enumerate}

\subsubsection{Proof of Subsection \ref{sec:properties_geometriaclly_admissible_pairing}.(1)}\label{sec:proof_of_properties_admissible_pairing}
By definition, \(O_{\mathfrak{m}}\) is a regular local ring. Therefore, the torsion-free \(O_{\mathfrak{m}}\)-module \(W_{\mathfrak{m}}\) is free, so \(\beta(W_{\mathfrak{m}},W_{\mathfrak{m}}) \subseteq O_{\mathfrak{m}}\) holds since \(\beta\) is geometrically admissible.

\subsubsection{Proof of Subsection \ref{sec:properties_geometriaclly_admissible_pairing}.(2)}\label{sec:proof_properties_of_geometric_admissible_metric1} Since the quotient field of \(O\) is a subalgebra of \(\Bbbk(\!(z)\!)\), we have \(N \subseteq \Bbbk(\!(z)\!)\). Furthermore, since \(O\) has Krull dimension one, \(\dim(N/O) < \infty\) and \(\dim(V/W) < \infty\).
In particular, \(V\) is an \(A\)-lattice and \((N,V)\) is a ringed \(A\)-lattice.

Every closed point \(q \in X\setminus \{q\} \cong \textnormal{Spec}(N)\) corresponds to a maximal ideal \(\mathfrak{m}_q \subseteq N\). Since \(N\) is integrally closed of dimension one, \(\mathfrak{m}_q\) is regular. Combined with \(c(\sheafO_{X,q}) = O_{\mathfrak{m}_q}\) and \(\zeta(\sheafA_q) = W_{\mathfrak{m}_q}\) for all \(q \in X \setminus \{p\}\), we obtain
\[\beta(\zeta(\sheafA_q),\zeta(\sheafA_q)) \subseteq \sheafO_{X,q}\]
from (1).
This implies that 
for all \(U \subseteq X \setminus \{p\}\) open
\[\beta(\zeta(\Gamma(U,\sheafA)),\zeta(\Gamma(U,\sheafA))) \subseteq c(\Gamma(U,\sheafO_X))\]
holds. Combined with the fact that \(\zeta(\Gamma(U,\sheafA)) = \zeta(\Gamma(U\setminus\{p\},\sheafA)) \cap A[\![z]\!]\) holds for all open neighbourhoods \(U\subseteq X\) of \(p\), we can simply define \(\beta_\sheafA\) via the diagrams \eqref{eq:beta_A}.

\subsubsection{Proof of Subsection \ref{sec:properties_geometriaclly_admissible_pairing}.(3)}\label{sec:proof_properties_of_geometric_admissible_metric2} Note that the fiber of \(\beta_\sheafA\) at \(p\) can be identified with \(\beta\). In particular, this fiber is non-degenerate. Let \(\sheafA \to \sheafA^*\) be the canonical morphism induced by \(\beta_\sheafA\). Then the fact that \(\beta_\sheafA|_p\) is non-degenerate translates to the fact that \(\sheafA|_p \to \sheafA^*|_p\) is an isomorphism. In particular, the kernel and cokernel of \(\sheafA \to \sheafA^*\) are torsion. The observation that the kernel, as a torsion subsheaf of the torsion-free sheaf \(\sheafA\), is vanishing concludes the proof.

\subsection{The categorization theorem}
The rest of this section is dedicated to the proof of the following theorem.

\begin{theorem}\label{thm:categorization_of_manin_triples}
Let \((A,\beta)\) be a geometrically admissible algebra over a field \(\Bbbk\) of characteristic 0, \(n \in \bN\), and \(\lambda \in \Bbbk[\![z]\!]^\times\). Furthermore, let \(((D_n(A),\beta_{(n,\lambda)}),A[\![z]\!],W)\) be the Manin triple associated to this datum in Subsection \ref{sec:manin_triples_over_series_explicit}.

Then \(n \in \{0,1,2\}\) and \(\lambda = 1\) up to isomorphism in the sense that
\begin{equation*}
    ((D_n(A),\beta_{(n,\lambda)}),A[\![z]\!],W)  \cong ((D_n(A),\beta_{(n,1)}),A[\![z]\!],\widetilde{W})
\end{equation*}
for an appropriate \(\widetilde{W} \subseteq D_n(A)\).

In particular, if {\tt C} is a full subcategory of {\tt Alg}\(_\Bbbk\) closed under taking subalgebras, a non-degenerate topological \(D\)-bialgebra \((A[\![z]\!],\delta)\) in {\tt C} satisfies \[((D(A[\![z]\!],\delta),\textnormal{ev}),A[\![z]\!]) \cong ((D_n(A),\beta_{(n,1)}),A[\![z]\!])\] 
for an appropriate \(n \in \{0,1,2\}\).
\end{theorem}

\noindent
The proof proceeds in several steps. We being by collecting several algebraic properties of Manin triples of the form \eqref{eq:manin_triples_over_series_explicit} in Section \ref{lem:manin_triples}. Using the geometrization method from Section \ref{sec:geometrization}, we pass from these Manin triples to certain geometric data. The application of algebro-geometric tools then concludes the proof of the refinement of Theorem \ref{thm:categorization_of_manin_triples} in Section \ref{sec:geometric_categorization_general}. 

\begin{remark}
    If \(A\) is a Lie algebra and \(\Bbbk\) is algebraically closed, Theorem \ref{thm:categorization_of_manin_triples} coincides with \cite[Theorem 2.10]{montaner_stolin_zelmanov}. However, our proof is independent of the proof of \cite[Theorem 2.10]{montaner_stolin_zelmanov}. In other words, we give a new proof of this result.
\end{remark}

\begin{remark}
The assumption on characteristic could be weakened by careful analysis of the following steps. For instance, the geometrization in Section \ref{sec:geometrization} works over fields where the characteristic does not divide the dimension of \(A\). Furthermore, most geometric methods used below are adaptable to fields of non-zero characteristic. However, for sake of clarity, we shall not pursue this level of generality.    
\end{remark}

\subsection{Algebraic properties of Manin triples of the form \eqref{eq:manin_triples_over_series_explicit}}\label{lem:manin_triples}
Let \((A,\beta)\) be a finite-dimensional metric \(\Bbbk\)-algebra, \(n \in \bN_0\), \(\lambda \in \Bbbk[\![z]\!]^\times\), and
\[((D_n(A),\beta_{(n,\lambda)}),A[\![z]\!],W)\]
be the Manin triple associated to this datum in Subsection \ref{sec:manin_triples_over_series_explicit}.
Furthermore, let \(W_+ \) (resp.\ \( W_- \)) be the projection
of 
\[W \subseteq D_n(A) = A(\!(z)\!) \times A[z]/z^nA[z]\]
onto \(A(\!(z)\!)\) (resp.\ \( A[z]/z^{n}A[z]\)).

The following results are true:
\begin{enumerate}
    \item \(W_\pm^\bot \subseteq W_\pm\) with respect to the bilinear forms \(\beta^\pm_{(n,\lambda)}\) defined by
    \begin{equation}\label{eq:beta_nlambda_pm}
    \beta_{(n,\lambda)}^+(a_1,a_2) \coloneqq \textnormal{res}_0\frac{1}{z^n\lambda}\beta(a_1,a_2) \textnormal{ and }\beta_{(n,\lambda)}^-([b_1],[b_2]) \coloneqq \textnormal{res}_0\frac{1}{z^n\lambda}\beta(b_1,b_2),
    \end{equation}
    where \(a_1,a_2 \in A(\!(z)\!)\) and \([b_1],[b_2] \in A[z]/z^nA[z] = A[\![z]\!]/z^nA[\![z]\!]\) are the classes of \(b_1,b_2 \in A[\![z]\!]\).
    
    \item \(A(\!(z)\!) = A[\![z]\!] + W_+\) and \(\dim(A[\![z]\!]\cap W_+) < \infty\);
    
    \item \(W_+/W_+^\bot \times W_-/W_-^\bot = (A[\![z]\!] \cap (W_+\times W_-)) \oplus W/(W_+^\bot \times W_-^\bot)\) is a finite-dimensional Manin triple, so \(\dim(W_+/W_+^\bot) = \dim(W_-/W_-^\bot) < \infty\). Here, we recall that \(A[\![z]\!]\) is considered as a subalgebra of \(D_n(A) = A(\!(z)\!)\times A[z]/z^nA[z]\) via the diagonal embedding;
    
    \item If \(n > 0\), we have \(\dim(A[\![z]\!]\cap W_+) > 0\).
\end{enumerate}

\subsubsection{Proof of Subsection \ref{lem:manin_triples}.(1)}
Follows immediately from the fact that \eqref{eq:beta_nlambda} and \eqref{eq:beta_nlambda_pm} implies
\begin{equation}\label{eq:Wbot}
    W_+^\bot \times W_-^\bot = (W_+\times W_-)^\bot \subseteq W^\bot = W\subseteq W_+\times W_-.
\end{equation}

\subsubsection{Proof of Subsection \ref{lem:manin_triples}.(2)}
Observe that \(A[\![z]\!] + W = A(\!(z)\!) \times A[z]/z^nA[z]\) implies 
\(A(\!(z)\!) = A[\![z]\!] + W_+\). Therefore, \(\{0\} = (A[\![z]\!] + W_+)^\bot = z^n A[\![z]\!] \cap W_+^\bot\) since 
\[A[\![z]\!]^\bot = z^n\lambda A[\![z]\!] = z^nA[\![z]\!]\] 
with respect to \(\beta_{(n,\lambda)}^+\).
This implies that \(A[\![z]\!] \cap W_+^\bot\) can be embedded into \(A[\![z]\!]/z^n A[\![z]\!]\) and is therefore finite-dimensional.
Consequently, the dimension of
\(A[\![z]\!]\cap W_+\) is finite if
the quotient
\((A[\![z]\!]\cap W_+) / (A[\![z]\!]\cap W_+^{\perp}) \)
is finite-dimensional. The latter space can be identified with a subspace of \(W_+/W_+^\bot\).
Therefore, Subsection \ref{lem:manin_triples}.(2) follows from Subsection \ref{lem:manin_triples}.(3).

\subsubsection{Proof of Subsection \ref{lem:manin_triples}.(3)}
The kernel \(K\) of the projection \(W \to W_+\) contains \(\{0\} \times W_-^\bot\) by virtue of \eqref{eq:Wbot}. On the other hand, any element of \(K\) is of the form
\((0, a)\) for some \(a \in W_-\), so for all \( (w_+,w_-)\in W \)
\begin{equation}
    0 = \beta_{(n,\lambda)}((0,a),(w_+,w_-)) = -\beta_{(n,\lambda)}^-(a,w_-)
\end{equation}
holds, implying \( a \in W_-^{\perp} \)
and hence
\( K = \{ 0 \} \times W_-^{\perp}\).
We obtain an isomorphism
\[W/(W_+^\bot\times W_-^\bot) \longrightarrow W_+/W_+^\bot.\]
A similar argument yields 
\(W/(W_+^\bot\times W_-^\bot) \cong W_-/W_-^\bot\). Therefore, we obtain an isomorphism
\(W_+/W_+^\bot \to W_-/W_-^\bot\).
In particular, 
\[\dim(W_+/W_+^\bot) = \dim(W_-/W_-^\bot) \le \dim(A[z]/z^nA[z]) < \infty.\] 
Considering \(W \subseteq W_+ \times W_-\), the identity \(A[\![z]\!] \oplus W = A(\!(z)\!) \times  A[z]/z^nA[z]\) is equivalent to
\begin{equation}\label{eq:W+W-W}
    W_+ \times W_- = (A[\![z]\!] \cap (W_+\times W_-)) \oplus W.    
\end{equation}
Quoiting out \(W_+^\bot \times W_-^\bot\) concludes the proof.

\subsubsection{Proof of Subsection \ref{lem:manin_triples}.(4)}
Assume that \(n > 0\) and \(A[\![z]\!] \cap W_+ = \{0\}\). Then
\begin{equation*}
    A[\![z]\!] \cap (W_+\times W_-) = \{0\}
\end{equation*}
and \eqref{eq:W+W-W} imply
\begin{equation}\label{eq:W_pm_is_W_pmbot}
    W = W_+ \times W_- = W_+^\bot \times W_-^\bot.
\end{equation}
For any \(a \in A[z]\) exists \(b \in A[\![z]\!]\) and \(w_\pm \in W_\pm\) such that
\[(0,[a]) = (b,[b])+(w_+,w_-) \in A[\![z]\!] \oplus (W_+ \times W_-).\]
Therefore, \(w_+ = -b \in A[\![z]\!] \cap W_+ = \{0\}\) results in \([a] = w_- \in W_-\). Since \(a \in A[z]\) was arbitrary, we conclude \(W_- = A[z]/z^nA[z]\), which contradicts \(W_-^\bot = W_-\) in \eqref{eq:W_pm_is_W_pmbot}.

\subsection{Geometric properties of Manin triples over series}\label{sec:geometric_categorization_general}
We are now in the position to proof Theorem \ref{thm:categorization_of_manin_triples}. More precisely, we proof the following refinement of this theorem.

Let \((A,\beta)\) be a geometrically admissible algebra over a field \(\Bbbk\) of characteristic 0, \(n \in \bN_0\), \(\lambda \in \Bbbk[\![z]\!]^\times\), and \(((D_n(A),\beta_{(n,\lambda)}),A[\![z]\!],W)\) be the Manin triple constructed in Subsection \ref{sec:manin_triples_over_series_explicit}. Furthermore, let \(W_+ \subseteq A(\!(z)\!)\) be the image of \(W\) under the projection \(D_n(A) \to A(\!(z)\!)\) and consider \(M \coloneqq \{f\in \Bbbk(\!(z)\!)\mid f W_+ \subseteq W_+\}\). 

Then \(n \in \{0,1,2\}\) and, up to isomorphism of Manin triples, \(\lambda = 1\)
and precisely one of the following cases occurs:
\begin{enumerate}
    \item \(n = 0\) and \(M\) is integrally closed satisfying \(\dim(\Bbbk(\!(z)\!)/(\Bbbk[\![x]\!] + M)) = 1\);
    \item \(n = 0\) and \(\Bbbk[u',uu'] \subseteq M\) for \(u \in z^{-1}\Bbbk[\![z]\!]^\times\) satisfying \(u \neq z^{-1}\); 
    \item \(n = 0\) and \(\Bbbk[z^{-2},z^{-3}] \subseteq M\);
    \item \(n = 1\) and \(M = \Bbbk[z^{-1}]\);
    \item \(n = 2\) and \(M = \Bbbk[z^{-1}]\).
\end{enumerate}
    
\subsection{Proof of the results in Section \ref{sec:geometric_categorization_general} (and by proxy of Theorem \ref{thm:categorization_of_manin_triples})}\label{sec:proof_of_categorization_theorem}
We use similar arguments as in the proof of \cite[Theorem 3.6]{abedin_universal_geometrization} or \cite[Section 8.4]{abedin_maximox_stolin_zelmanov}. 

Let \(N\) be the integral closure of \(M\). Subsection \ref{sec:properties_geometriaclly_admissible_pairing}.(2) states that \(V \coloneqq NW_+\) is an \(A\)-lattice.
Let \(((X,\mathcal{A}),(p,c,\zeta))\) be the geometric datum associated to the ringed \(A\)-lattice \((N,V)\) in Section \ref{sec:geometrization}. By virtue of Subsection \ref{sec:properties_geometriaclly_admissible_pairing}.(3), we have a short exact sequence
\begin{align}\label{eq:tildeKexactseq}
    0 \longrightarrow \mathcal{A} {\longrightarrow} \mathcal{A}^* \longrightarrow \mathcal{C}\longrightarrow0,
\end{align}
where \(\mathcal{C}\) is a torsion sheaf.
The associated long exact sequence in cohomology reads
\begin{equation}
    0 \longrightarrow \textnormal{H}^0(\mathcal{A}) {\longrightarrow} \textnormal{H}^0(\mathcal{A}^*) \longrightarrow \textnormal{H}^0(\mathcal{C})\longrightarrow \textnormal{H}^1(\mathcal{A}) {\longrightarrow} \textnormal{H}^1(\mathcal{A}^*) \longrightarrow \textnormal{H}^1(\mathcal{C}) \longrightarrow 0.
\end{equation}
The identities \(\textnormal{H}^1(\mathcal{A}) = 0 = \textnormal{H}^1(\mathcal{C})\) imply that \(\textnormal{H}^1(\mathcal{A}^*) = 0\). 

The Riemann-Roch theorem for \(\mathcal{A}\) and \(\mathcal{A}^*\) combined with the fact that \(\textnormal{h}^1(\mathcal{O}_X) = g\) reads 
\begin{align*}
    &0 \le \textnormal{h}^0(\mathcal{A})-\textnormal{h}^1(\mathcal{A}) = \textnormal{deg}(\textnormal{det}(\mathcal{A})) + (1-g)\textnormal{rank}(\mathcal{A}),\\
    &0 \le \textnormal{h}^0(\mathcal{A}^*)-\textnormal{h}^1(\mathcal{A}^*) = -\textnormal{deg}(\textnormal{det}(\mathcal{A})) + (1-g)\textnormal{rank}(\mathcal{A}),
\end{align*}
where we used that \(\textnormal{det}(\mathcal{A}^*) = \textnormal{det}(\mathcal{A})^*\) implies \(\textnormal{deg}(\textnormal{det}(\mathcal{A}^*)) = -\textnormal{deg}(\textnormal{det}(\mathcal{A}))\). We conclude \( g \le 1\).

\subsubsection{The case \(g = 1\)} Assume \(g = 1\), then \(X\) is an elliptic curve. Then the sheaf \(\Omega^1_{X}\) of regular 1-forms on \(X\) satisfies \(\Omega^1_X \cong \sheafO_X\). Therefore, \(0 = \textnormal{h}^1(\mathcal{A}^*) = \textnormal{h}^0(\mathcal{A})\) because of Serre duality.
In particular, by \eqref{eq:cohomology}, \(W_+ \cap \fg[\![z]\!] \subseteq V \cap \fg[\![z]\!] = \{0\}\), so Subsection \ref{lem:manin_triples}.(4) implies \(n = 0\).
Moreover, \(W_+ \oplus \fg[\![z]\!] = \fg(\!(z)\!) =  V \oplus \fg[\![z]\!]\)
and \(W_+ \subseteq V\) imply \(V = W_+\), so \(M = N\) is integrally closed. Since \(\textnormal{h}^1(\sheafO_X) = \dim(\Bbbk(\!(z)\!)/(\Bbbk[\![z]\!] + M)) = 1\) by virtue of \eqref{eq:arithmetic_genus}, we are in case (1).

\subsubsection{The case \(g = 0\)} 
Note that \(g= 0\) means \(\Bbbk(\!(z)\!) = \Bbbk[\![z]\!] + N\) by virtue of \eqref{eq:arithmetic_genus}. Since 
\[N \cap \Bbbk[\![z]\!] = \textnormal{H}^0(\mathcal{O}_X) = \Bbbk,\]
we can see that \(N = \Bbbk[u]\) for the unique \(u \in (z^{-1}+z\Bbbk[\![z]\!]) \cap N \neq \{0\}\). Let \(N^\bot\) be the orthogonal complement of \(N\) with respect to the bilinear form \(R_{(n,\lambda)} \colon \Bbbk(\!(z)\!) \times \Bbbk(\!(z)\!) \to \Bbbk\) defined by
\begin{equation}
    (f,g) \longmapsto \textnormal{res}_{0} \frac{1}{z^n\lambda} fg.
\end{equation}
Since \((A,\beta)\) is geometrically admissible, we have \(\beta(V,V) \subseteq N\) for \(\beta\) from \eqref{eq:beta_extension}, so for all \(f \in N^\bot\) and \(a,b\in W_+ \subseteq V\) we have
\begin{equation}
    \beta_{(n,\lambda)}^+(fa,b) = \textnormal{res}_0 \frac{1}{z^n\lambda}f\beta(a,b) = R_{(n,\lambda)}(f, \beta(a,b)) = 0.
\end{equation}
Here, we recall that \(\beta^+_{(n,\lambda)}\) was defined in \eqref{eq:beta_nlambda_pm}.
In particular, \(fa \in W_+^\bot\) so \(fa \in W_+\) since \(W_+^\bot \subseteq W_+\) by Subsection \ref{lem:manin_triples}.(1). Therefore, \(N^\bot \subseteq \{f \in \Bbbk(\!(x)\!)\mid f W_+ \subseteq W_+\} = M \subseteq N\). The identity 
\begin{equation}\label{eq:Rnlambda_and_multipliers}
    R_{(n,\lambda)}(z^n \lambda u',u^k) = \textnormal{res}_0 u'u^k = \frac{1}{k+1}\textnormal{res}_0\left(u^{k+1}\right)' = 0
\end{equation}
for all \(k \in \bN\) yields
\(z^n \lambda u' \in N^\bot \subseteq M\).

\subsubsection{The case \((n,g) = (0,0)\).} Since \(R_{(n,\lambda)}\) is associative 
and \( v\coloneqq \lambda u' \in N^\bot\) we have 
the inclusion \(v N \subseteq N^\bot\). Furthermore, since
\(v \in N^\bot \subseteq N = \Bbbk[u]\), we obtain 
\(\Bbbk[v,vu] \subseteq \Bbbk + vN \subseteq \Bbbk + N^\bot\). Since all three spaces are of codimension one in \(N = \Bbbk[u]\) we conclude
\begin{equation}
    \Bbbk[v,vu] = \Bbbk + vN = \Bbbk+N^\bot \subseteq \{f \in \Bbbk(\!(z)\!)\mid fW_+ \subseteq W_+\} = M.
\end{equation}

\subsubsection{Case \((n,g) = (1,0)\).} Since \(v \coloneqq z\lambda u' \in N^\bot \cap z^{-1}\Bbbk[\![z]\!]^\times \subseteq N \cap z^{-1}\Bbbk[\![z]\!]^\times = \Bbbk^\times u + \Bbbk\), we have \(M = N = \Bbbk[v]\). 

\subsubsection{Case \((n,g) = (2,0)\).} Since
\(z^2\lambda u' \in N^\bot \cap \Bbbk[\![z]\!] \subseteq N \cap \Bbbk[\![z]\!] = \Bbbk\) 
we have \(au' = -z^{-2}\lambda^{-1}\) for some \(a \in \Bbbk^\times\).  Consequently, \(\textnormal{res}_0 z^{-2}\lambda^{-1} = \textnormal{res}_0 au' = 0\) and \(N = \Bbbk[u] = N^\bot \subseteq M\) by \eqref{eq:Rnlambda_and_multipliers}.

\subsubsection{Case \(n \ge 3\)} The fact that \(z^3\lambda u' \in N^\bot \cap z\Bbbk[\![z]\!] = \{0\}\) is a contradiction. In particular, there cannot exist any Manin triple of the form \(((D_n(A),\beta_{(n,\lambda)}),A[\![z]\!],W)\) for \(n \le 3\).

\subsubsection{Concluding the proof} 
As a metric algebra \((D_n(A),\beta_{(n,\lambda)}) \cong (A \otimes R_n,\beta \otimes t_{(n,\lambda)})\); see Section \ref{sec:trace_extension}.
It is shown in \cite[Proposition 3.12]{abedin_maximox_stolin_zelmanov} that \((R_n,t_{(n,\lambda)}) \cong (R_n,t_{(n,1/(1+az^{n-1}))})\) as trace extensions for \(a = \textnormal{res}_0z^{-n}\lambda^{-1} \in \Bbbk\). In particular, since \(n \le 2\) and for \(n = 2\) the identity \(\textnormal{res}_0 z^{-n}\lambda^{-1} = 0\) holds, we obtain \(\lambda = 1\) up to isomorphism in all cases.

If \((n,g) = (0,0)\), this means that \(\Bbbk[u',uu'] \subseteq M \subseteq N = \Bbbk[u]\). In particular, since by definition \(u \in (z^{-1} + zb + z^2\Bbbk[\![z]\!])\) for some \(b \in \Bbbk\), we see that \(u' + u^2 - 3b \in \Bbbk[u] \cap z\Bbbk[\![z]\!] = \{0\}\). If \(b = 0\) this formal differential equation has the unique solution \(u = z^{-1}\) and we are in case (3) and if \(b \neq 0\) we are in case (2). 

If \((n,g) = (1,0)\), we have \(zu' \in M = N = \Bbbk[u]\), so \(zu' = -u\). The only solution to this equation is again \(z^{-1}\) and we are in case (4). Finally, if \((n,g) = (2,0)\), we have \(u' = -z^{-2}\) so \(u = z^{-1}\) again and we are in case (5).

\section{Non-degenerate \(D\)-bialgebra structures and the classical Yang-Baxter equation} \label{sec:A_CYBE}

\subsection{Series of type \((n,\lambda)\)}\label{sec:series_in_standard_form}
Let \((A,\beta)\) be a finite-dimensional metric \(\Bbbk\)-algebra. Choose a basis \(\{b_i\}_{i = 1}^d\) of \(A\) and consider its dual basis \(\{b_i^*\}_{i = 1}^d\), i.e.\ \(\beta(b_i,b_j^*) = \delta_{ij}\). The tensor \(\gamma = \sum_{i = 1}^d b^*_i \otimes b_i\) is independent of the choice of \(\{b_i\}_{i = 1}^d\subseteq A\), symmetric, and satisfies 
\begin{equation}\label{eq:A_invariance_of_gamma}
    a^{(1)}\gamma = \gamma a^{(2)} \textnormal{ or, equivalently, }a^{(2)}\gamma = \gamma a^{(1)}.    
\end{equation}
The first identity follows from the fact that
\begin{align*}
    &\beta^{\otimes2} (a^{(1)}\gamma,b_j \otimes b_k^*) = \sum_{i = 1}^d \beta(ab_i^*,b_j)\beta(b_i,b_k^*) = \beta(ab_k^*,b_j) = \beta(b_k^*,b_ja) = \beta(\gamma a^{(2)},b_j\otimes b_k^*).
\end{align*}
holds for all \(j,k \in \overline{1,d}\), and the second follows from the first by using the symmetry of \(\gamma\). We call \(\gamma\) \emph{canonical \(A\)-invariant element of \((A,\beta)\)}.

Let us note that the canonical embedding \((A \otimes A)[\![x,y]\!] \to (D_n(A) \otimes A)[\![y]\!]\) extends to 
\begin{equation}\label{eq:rational_r_as_series}
    (A \otimes A)[\![x,y]\!][(x-y)^{-1}] \longrightarrow (D_n(A) \otimes A)(\!(y)\!)
\end{equation}
by writing
\begin{equation}
    \frac{1}{x-y} = \sum_{k = 0}^{n-1} (0,-[x]^{n-k-1}) y^{k-n} + \sum_{k = 0}^\infty (x^{-k-1},0) y^k \in (\Bbbk(\!(x)\!) \times \Bbbk[x]/(x^n))(\!(y)\!).  
\end{equation}
Indeed, this is appropriate since we can calculate
\begin{equation}
    \begin{split}
        &((x,[x]) - y)\left(\sum_{k = 0}^{n-1} (0,-[x]^{n-k-1}) y^{k-n} + \sum_{k = 0}^\infty (x^{-k-1},0) y^k\right) \\&= \sum_{k = 0}^{n-1} (0,-[x]^{n-k}) y^{k-n} - \sum_{k = 1}^{n} (0,-[x]^{n-k}) y^{k-n} + (1,0) = (1,1)
    \end{split}
\end{equation}
inside \((\Bbbk(\!(x)\!) \times \Bbbk[x]/(x^n))(\!(y)\!)\).

In particular, for any \(n \in \bN\) we obtain
\begin{equation}\label{eq:def_wki}
    \begin{split}
        \frac{y^n\gamma}{x-y} &= \sum_{k = 0}^{n-1}\sum_{i = 1}^d (0,-[b_i^* x^{n-1-k}]) \otimes b_i y^k + \sum_{k = n}^\infty \sum_{i = 1}^d (b_i^*x^{n-1-k},0) \otimes b_i y^k \\&= \sum_{k = 0}^{\infty}\sum_{i = 1}^dw_{k,i} \otimes b_iy^k \in (D_n(A) \otimes A)[\![y]\!].
    \end{split}
\end{equation}
For any \(\lambda\in \Bbbk[\![z]\!]^\times\) and \(s \in (A\otimes A)[\![x,y]\!]\), we identify the expression
\begin{align}\label{eq:standard_form}
    r(x,y) = \frac{y^n\lambda(x) \gamma}{x-y} + s(x,y) \in (A \otimes A)[\![x,y]\!][(x-y)^{-1}]
\end{align}
with its series in \((D_n(A) \otimes A)[\![y]\!]\) and say that \(r\) is a series of type \((n,\lambda)\).

\begin{remark}\label{rem:series_nlambda_type}
Every \(r(x,y) = \frac{a(x,y)\gamma}{x-y} + s(x,y)\) for \(a \in \Bbbk[\![x,y]\!]\) such that \(a(z,z) \neq 0\) and any \(s \in (A\otimes A)[\![x,y]\!]\) has a unique representation as a series of type \((n,\lambda)\). Indeed, chose \((n,\lambda)\) such that \(a(z,z) = z^n \lambda(z)\). Then \(a(x,y) - y^n\lambda(x) = (x-y)b(x,y)\) for some \(b \in \Bbbk[\![x,y]\!]\), so
\begin{equation}
    r(x,y) = \frac{y^n\lambda(x)\gamma}{x-y} + b(x,y)\gamma + s(x,y)
\end{equation}
is a series of type \((n,\lambda)\).

In the construction of \(b\) we used the following easy fact: for any \(\Bbbk\)-vector space \(V\) 
\begin{equation}\label{eq:series_vanishing_at_diagonal}
    f \in V[\![x,y]\!], f(z,z) = 0 \implies f(x,y) = (x-y)g(x,y)\textnormal{ for some }g\in V[\![x,y]\!]
\end{equation}
holds. 
\end{remark}

\smallskip
\noindent 
Note that we have a linear automorphism of \((A \otimes A)[\![x,y]\!][(x-y)^{-1}]\) defined by 
\begin{equation}\label{eq:skew-series}
    a(x,y) \longmapsto \overline{a}(x,y) \coloneqq -\tau(a(y,x))
\end{equation}
where \(\tau(a\otimes b) = b\otimes a\) is applied coefficient-wise. 
For any series \(r\) of type \((n,\lambda)\), \(\overline{r}\) is again a series of type \((n,\lambda)\) by Remark \ref{rem:series_nlambda_type}. We call \(r\) \emph{skew-symmetric} if \(r = \overline{r}\).

\subsection{The (generalized) classical Yang-Baxter equation with coefficients in arbitrary algebras}\label{sec:definition_CYBE}
As in the last section, \((A,\beta)\) is a finite-dimensional metric \(\Bbbk\)-algebra, \(\{b_i\}_{i = 1}^d\) and \( \{b_i^*\}_{i = 1}^d\) are basis of \(A\) satisfying \(\beta(b_i^*,b_j) = \delta_{ij}\), and \(\gamma \coloneqq \sum_{i = 1}^d b_i^* \otimes b_i\). Furthermore, let \(U\) be the unitalization of \(A\), i.e.\ \(U = A \oplus \Bbbk\) with multiplication
\begin{equation}
    (a_1,u_1)(a_2,u_2) = (a_1a_2 + u_1a_2 + u_2a_1,u_1u_2)
\end{equation}
for all \(a_1,a_2 \in A\) and \(u_1,u_2 \in \Bbbk\).

For any \(s \in (A \otimes A)[\![x,y]\!][(x-y)^{-1}]\), let us define the expressions
\begin{equation}\label{eq:ij_notations_1}
    s^{12}(z_1,z_2),s^{13}(z_1,z_3),s^{23}(z_2,z_3) \in (U \otimes U \otimes U)[\![z_1,z_2,z_3]\!]\left[\frac{1}{(z_1-z_2)(z_1-z_3)(z_2-z_3)}\right]
\end{equation}
coefficient-wise via
\begin{equation}\label{eq:ij_notations_constant}
    t^{12} = t \otimes 1, t^{13} = a \otimes 1 \otimes b, t^{23} = 1 \otimes t \in U \otimes U \otimes U 
\end{equation}
for \(t = a\otimes b \in A \otimes A\).

Let us point out that for example \((a_1\otimes a_2)^{13}(b_1 \otimes b_2)^{12} = a_1b_1 \otimes b_2 \otimes a_2 \in A \otimes A \otimes A\). This and similar identities imply that for all \(s_1,s_2 \in (A \otimes A)[\![x,y]\!][(x-y)^{-1}]\)
\begin{equation}\label{eq:ij_notations_2}
    s_1^{13}(z_1,z_3)s_2^{12}(z_1,z_2), s_1^{12}(z_1,z_2)s_2^{23}(z_2,z_3)\textnormal{, and } s_1^{23}(z_2,z_3)s_2^{13}(z_1,z_3)
\end{equation}
are elements of 
\[(A \otimes A \otimes A)[\![z_1,z_2,z_3]\!]\left[\frac{1}{(z_1-z_2)(z_1-z_3)(z_2-z_3)}\right].\]
Furthermore, if \(s_1,s_2\) are of the form \eqref{eq:standard_form} and we write
\[s_\epsilon = \sum_{k \in \bZ}\sum_{i = 1}^d s_{\epsilon;k,i}(x) \otimes b_iy^k\in (D_n(A) \otimes A)(\!(y)\!),\]
for \(\epsilon \in \{1,2\}\) for the associated series via \eqref{eq:rational_r_as_series}, we get
\begin{align*}
    &s_1^{13}(z_1,z_3)s_2^{12}(z_1,z_2) = \sum_{k,\ell \in \bN} \sum_{i,j = 1}^d s_{1;k,i}(z_1)s_{2;\ell,j}(z_1) \otimes b_jz_2^\ell \otimes b_iz_3^k \in (D_n(A) \otimes A \otimes A)[\![z_2,z_3]\!].
\end{align*}
Similar formulas for \(s_1^{12}(z_1,z_2)s_2^{23}(z_2,z_3)\) and \(s_1^{23}(z_2,z_3)s_2^{13}(z_1,z_3)\) hold.

For any \(r \in (A \otimes A)[\![x,y]\!][(x-y)^{-1}]\), we call the equation \(\textnormal{GCYB}(r) = 0\) the \emph{\(A\)-generalized classical Yang-Baxter equation} (short: \(A\)-GCYBE), where 
\begin{equation}\label{eq:gcybe}
    \textnormal{GCYB}(r) \coloneqq r^{13}r^{12}-r^{12}r^{23}+\overline{r}^{23}r^{13}
\end{equation}
Here, \(\overline{(\cdot)}\) was defined in \eqref{eq:skew-series}.

Similarly, we call the equation \(\textnormal{CYB}(r) = 0\) the \emph{\(A\)-classical Yang-Baxter equation} (short: \(A\)-CYBE), where 
\begin{equation}\label{eq:cybe}
    \textnormal{CYB}(r) \coloneqq r^{13}r^{12}-r^{12}r^{23}+r^{23}r^{13}.
\end{equation}
If \(A\) is a Lie algebra, these are exactly the usual (generalized) classical Yang-Baxter in two-formal parameter. If \(A\) is associative, the \(A\)-classical Yang-Baxter equation is a formal variant of the associative version of the CYBE used in \cite{odesskii_sokolov}, which is itself a spectral parameter generalization of the associative CYBE discussed in e.g.\ \cite{aguiar_associative}.

\subsection{Solutions of the \(A\)-(G)CYBE and subspaces of \(D_n(A)\)}\label{sec:solutions_of_CYBE_and_manin_triples}
Series of type \((n,\lambda)\) can be seen as generating series of certain subspaces of \(D_n(A)\). More precisely, we have the following result, which is a generalization of known statements in the Lie algebra case; see e.g.\ \cite{gelfand_cherednik,skrypnyk_infinite_dimensional_Lie_algebras,abedin_maximov_stolin_quasibialgebras}.

\begin{theorem}\label{thm:series_and_subspaces}
    Let \((A,\beta)\) be a finite-dimensional metric \(\Bbbk\)-algebra, \(\{b_i\}_{i = 1}^d\) and \(\{b^*_i\}_{i = 1}^d\) be basis of \(A\) satisfying \(\beta(b_i^*,b_j) = \delta_{ij}\), \(\gamma \coloneqq \sum_{i = 1}^d b_i^* \otimes b_i\), and \(n \in \bN\). To any series \(r(x,y) = \sum_{k = 0}^\infty\sum_{i = 1}^d r_{k,i}(x) \otimes b_iy^k \in (D_n(A) \otimes A)[\![y]\!]\), we can define a linear subspace
\begin{equation}
    A(r) \coloneqq \textnormal{Span}_{\Bbbk}\{ r_{k,i} \mid k\in\bN,i\in \overline{1,d}\} \subseteq D_n(A).
\end{equation}
For any fixed \(\lambda \in \Bbbk[\![z]\!]^\times\) the following results are true:
\begin{enumerate}
    \item \(r \mapsto A(r)\) defines a bijection between series \(r\) of type \((n,\lambda)\) (i.e.\ of the from \eqref{eq:standard_form}) and subspaces \(W \subseteq D_n(A)\) satisfying \(D_n(A) = A[\![z]\!] \oplus W\). 
    
    \item For any series \(r\) of type \((n,\lambda)\), the identity \(A(r)^\bot = A(\overline{r})\) holds, where \((\cdot)^\bot\) is meant with respect to \(\beta_{(n,\lambda)}\) from \eqref{eq:beta_nlambda} and \(\overline{(\cdot)}\) is defined in \eqref{eq:skew-series}.

    \item For any series \(r\) of type \((n,\lambda)\), the identity
    \begin{equation}
        \textnormal{GCYB}(r) = \varphi 
    \end{equation}
    holds
    for the unique element \(\varphi \in (A \otimes A \otimes A)[\![z_1,z_2,z_3]\!]\)
    determined by
    \begin{equation}        \beta_{(n,\lambda)}^{\otimes 3}(v_1\otimes v_2 \otimes v_3,\varphi) = \beta_{(n,\lambda)}(v_1,v_3v_2)
    \end{equation}
    for all \(v_1 \in A(\overline{r}), v_2,v_3 \in A(r)\).
    \end{enumerate}
\end{theorem}

\smallskip
\noindent
The proof of Theorem \ref{thm:series_and_subspaces} is postponed to Subsection \ref{sec:proof_solutions_of_CYBE_and_manin_triples}.

A direct consequence of Theorem \ref{thm:series_and_subspaces}.(1)\&(3) is that \(r \mapsto A(r)\) defines a bijection between solutions \(r\) of the \(A\)-GCYBE \eqref{eq:gcybe} of type \((n,\lambda)\) and subalgebras \(W \subseteq D_n(A)\) satisfying \(D_n(A) = A[\![z]\!] \oplus W\).
Combined with Theorem \ref{thm:series_and_subspaces}.(2), we obtain a bijection between skew-symmetric solutions \(r\) of the \(A\)-CYBE \eqref{eq:cybe} of type \((n,\lambda)\) and Manin triples \(((D_n(A),\beta_{(n,\lambda)}),A[\![z]\!],W)\).
We will see that, if \(A\) is simple, any solution \(r\) of the \(A\)-CYBE of type \((n,\lambda)\) is already skew-symmetric. Therefore, we have the following consequence of Theorem \ref{thm:series_and_subspaces}.

\smallskip
\noindent

\begin{corollary}\label{thm:solutions_of_CYBE_and_manin_triples}
Let \((A,\beta)\) be a finite-dimensional simple metric \(\Bbbk\)-algebra, \(n \in \bN\), and \(\lambda \in \Bbbk[\![z]\!]^\times\). Then \(r \mapsto A(r)\) defines a bijection between solutions of the \(A\)-CYBE \eqref{eq:cybe} \(r\) of type \((n,\lambda)\) and Manin triples \(((D_n(A),\beta_{(n,\lambda)}),A[\![z]\!],W)\).
\end{corollary}

\smallskip
\noindent
The proof will be given in Subsection \ref{sec:proof_of_corollary_solutions_of_CYBE_and_subspaces}.

Since Manin triples of the form \eqref{eq:manin_triples_over_series_explicit} exist only for \(n \le 2\) by virtue of Theorem \ref{thm:categorization_of_manin_triples}, Theorem \ref{thm:solutions_of_CYBE_and_manin_triples} gives the same restriction for solutions of the \(A\)-CYBE for any geometrically admissible \(\Bbbk\)-algebra \((A,\beta)\). To be precise, we have the following result.

\begin{corollary}
Let \((A,\beta)\) be a finite-dimensional simple metric \(\Bbbk\)-algebra, \(n \in \bN\) and \(\lambda \in \Bbbk[\![z]\!]^\times\). If \(r \in (D_n(A) \otimes A)[\![y]\!]\) is a solution of the \(A\)-CYBE of type \((n,\lambda)\), we have \(n \in \{0,1,2\}\).
\end{corollary}

\subsubsection{Proof of Theorem \ref{thm:series_and_subspaces}}\label{sec:proof_solutions_of_CYBE_and_manin_triples}
The proof of (1) and (2) is completely analogous to the proof in the case that \(A\) is a Lie algebra in \cite[Theorem 3.6]{abedin_maximov_stolin_quasibialgebras}, so it remains to prove (3).

Let us begin by proving, that 
\begin{equation}\label{eq:GCYBE_regular}
    \textnormal{GCYB}(r) \in (A\otimes A \otimes A)[\![z_1,z_2,z_3]\!]\textnormal{ for all series }r \in (A\otimes A)[\![x,y]\!][(x-y)^{-1}]\textnormal{ of type }(n,\lambda).
\end{equation}
To this end, let \(r(x,y) = \frac{y^n\lambda(x) \gamma}{x-y} + s(x,y) \) be a series of type \((n,\lambda)\).
Clearly, \(T_1 \coloneqq \textnormal{GCYB}(s)\) is an element of \( (A \otimes A \otimes A)[\![z_1,z_2,z_3]\!]\). Since \(a^{(1)}\gamma = \gamma a^{(2)}\) for all \(a \in A\)
we have \(\gamma^{13}\gamma^{12} = \gamma^{12}\gamma^{23} = \gamma^{23}\gamma^{13}\). Therefore, if we write \(w \coloneqq \frac{y^n\lambda(x) \gamma}{x-y}\), we have
\begin{align*}
    &(z_1-z_2)(z_1-z_3)(z_2-z_3)\textnormal{GCYB}\left(w\right) \\&= (z_2z_3)^n\lambda(z_1)(\lambda(z_1)(z_2-z_3)-\lambda(z_2)(z_1-z_3)+ \lambda(z_3)(z_1-z_2))\gamma^{13}\gamma^{12}.
\end{align*}
This expression is zero if \(z_1 = z_2, z_1 = z_3\) or \(z_2 = z_3\), so 
\begin{equation}
    T_2 \coloneqq \textnormal{GCYB}\left(w\right) \in (A\otimes A \otimes A)[\![z_1,z_2,z_3]\!]
\end{equation}
Now let us turn to
\begin{align*}
    \textnormal{GCYB}(r) &= \textnormal{GCYB}\left(w + s\right) = T_1 + T_2 \\&+
    {w^{13}s^{12}} + {s^{13}w^{12}} - {w^{12}s^{23}} - {s^{12}w^{23}} + {\overline{w}^{23}s^{13}} + {\overline{s}^{23}w^{13}}
    \\&= T_1 + T_2 + \underbrace{(s^{13}w^{12}-w^{12}s^{23})}_{\coloneqq T_3} - \underbrace{(s^{12}w^{23}-\overline{w}^{23}s^{13})}_{\coloneqq T_4} + \underbrace{(\overline{s}^{23}w^{13}+w^{13}s^{12})}_{\coloneqq T_5}
\end{align*}
Write \(s(x,y) = \sum_{k \in \bN} \sum_{i = 1}^d s_{k,\ell}^{i,j}x^ky^\ell b_i \otimes b_j \) and note that \(\overline{s}(x,y) = -\sum_{k \in \bN} \sum_{i = 1}^d s_{k,\ell}^{i,j}x^\ell y^k b_j \otimes b_i\). 
Using \(z^ka^{(1)}\gamma - \gamma a^{(2)}z^k = 0\) for all \(a\in A\) we see that
\begin{equation}
    T_3 = \frac{\lambda(z_1)z_2^n}{z_1-z_2}\sum_{k,\ell \in \bN} \sum_{i,j = 1}^d s_{k,\ell}^{i,j} (z_1^kb_i^{(1)}\gamma - \gamma b_i^{(2)}z_2^k) \otimes b_jz_3^\ell  \in (A \otimes A \otimes A)[\![z_1,z_2,z_3]\!].
\end{equation}
by virtue of \eqref{eq:series_vanishing_at_diagonal}.

Similarly, under consideration of \(\overline{w} = \frac{\lambda(y)x^n\gamma}{x-y}\), we get
\begin{equation}
    T_4 = \frac{1}{z_1-z_2}\sum_{k,\ell \in \bN} \sum_{i,j = 1}^d s^{k,\ell}_{i,j}b_iz_1^k \otimes (\lambda(z_2)z_3^n z_2^\ell b_j^{(1)}\gamma - \gamma b_j^{(2)}z_3^\ell\lambda(z_3)z_2^n)  \in (A \otimes A \otimes A)[\![z_1,z_2,z_3]\!]
\end{equation}
Using \(a^{(2)}\gamma = \gamma a^{(1)}\) and the notation \(\theta_a(b \otimes c) = b \otimes a \otimes c\) for all \(a \in A\), we obtain
\begin{equation}
    T_5 = \frac{\lambda(z_1)z_3^n}{z_1-z_3}\sum_{k,\ell \in \bN} \sum_{i,j = 1}^d s_{k,\ell}^{i,j}\theta_{b_jz_2^\ell}(-z_3^{k}b_i^{(2)}\gamma + \gamma b_i^{(1)}z_1^k) \in (A \otimes A \otimes A)[\![z_1,z_2,z_3]\!].
\end{equation}
Summarized, we have \(
    \textnormal{GCYB}(r) = T_1 + T_2 + T_3 + T_4 + T_5 \in (A\otimes A \otimes A)[\![z_1,z_2,z_3]\!].    
\)

Let us now write 
\begin{align*}
    r(x,y) &= \frac{y^n\lambda(x) \gamma}{x-y} + s(x,y) = \sum_{k \in \bN}\sum_{i = 1}^d r_{k,i}(x) \otimes b_iy^k
\end{align*}
and observe that
\begin{equation}%
\label{eq:gcybe_subalgebra}
\begin{aligned}
\textnormal{GCYB}(r) &= \sum_{k, \ell \in \bN} \sum_{i,j = 1}^d
    r_{\ell, j}r_{k,i} \otimes b_i z_2^k \otimes b_j z_3^\ell \\& - \sum_{k \in \bN} \sum_{i = 1}^d
    r_{k,i} \otimes \left(z_2^kb_i^{(1)}r(z_2,z_3) - \overline{r}(z_2,z_3)b_i^{(2)}z_3^k\right)
\end{aligned}
\end{equation}
holds.
Here, we used the fact that the embedding \eqref{eq:rational_r_as_series} induces a commutative diagram
\begin{equation}
    \xymatrix{(A \otimes A \otimes A)[\![z_1,z_2,z_3]\!] \ar[r]\ar[d]& (D_n(A) \otimes A \otimes A)[\![z_2,z_3]\!]\\(A \otimes A \otimes A)[\![z_1,z_2,z_3]\!]\left[\frac{1}{z_1-z_3}\right]\ar[ur]&
    }.
\end{equation}
Applying \(\beta_{(n,\lambda)}^{\otimes 3}(\overline{r}_{k_1,i_1}\otimes r_{k_2,i_2} \otimes r_{k_3,i_3},-)\) to \eqref{eq:gcybe_subalgebra}, where \(\overline{r} = \sum_{k = 0}^\infty \sum_{i = 1}^d\overline{r}_{k,i} \otimes b_iy^k\) yields 
\begin{equation}
    \beta_{(n,\lambda)}^{\otimes 3}(\overline{r}_{k_1,i_1}\otimes r_{k_2,i_2} \otimes r_{k_3,i_3},\textnormal{GCYB}(r)) = \beta_{(n,\lambda)}(\overline{r}_{k_1,i_1},r_{k_3,i_3}r_{k_2,i_2})
\end{equation}
since \(\overline{r}_{k_1,i_1} \in A(\overline{r}) = A(r)^\bot\). This concludes the proof, since \(\{r_{k,i}\mid k \in \bN_0,i \in \overline{1,d}\}\) (resp.\ \(\{\overline{r}_{k,i}\mid k \in \bN_0,i \in \overline{1,d}\}\)) is a basis of \(A(r)\) (resp.\ \(A(\overline{r})\)).

\subsubsection{Proof of Corollary \ref{thm:solutions_of_CYBE_and_manin_triples}}\label{sec:proof_of_corollary_solutions_of_CYBE_and_subspaces}
Under consideration of Theorem \ref{thm:series_and_subspaces} and the remarks after this theorem, it remains to prove that, if \(A\) is simple, any solution \(r\) of the \(A\)-CYBE of type \((n,\lambda)\) is automatically skew-symmetric. 

The equality \(\textnormal{CYB}(r) = 0\) implies that \(A(r) \subseteq D_n(A)\) is a subalgebra by rewriting the CYBE similarly to \eqref{eq:gcybe_subalgebra}. Therefore, (3) implies that \(r\) solves the \(A\)-GCYBE \eqref{eq:gcybe}. This implies
\begin{equation}
    0 = \textnormal{CYB}(r)-\textnormal{GCYB}(r) = (r^{23}-\overline{r}^{23})r^{13}
\end{equation}
Multiplying by \(z_1-z_3\) and setting \(z_1 = z_3\) we obtain \((r^{23}-\overline{r}^{23})\gamma^{13} = 0\). This implies that \((r - \overline{r})b_i^{(2)} = 0\) for all \(i \in \overline{1,d}\). Since \(A\) is simple, we have \(a A = \{0\}\) implies \(a = 0\). Therefore, \(r = \overline{r}\), which concludes the proof.

\subsection{Connection to topological \(D\)-bialgebra structures}
Let \(r\in (D_n(A)\otimes A)[\![y]\!]\) be a solution of the \(A\)-CYBE \eqref{eq:cybe} of type \((n,\lambda)\). The identity \(\textnormal{CYB}(r) = 0\) can be rewritten as
\begin{equation}\label{eq:NCYB_and_delta}
    \begin{split}
        &\sum_{k,\ell = 0}^\infty \sum_{i,j = 1}^d r_{k,i}(z_1)r_{\ell,j}(z_1) \otimes b_jz_2^\ell \otimes b_iz_3^k  \\&= \sum_{k = 0}^\infty \sum_{i = 1}^d r_{k,i}(z_1) \otimes \left( \left(b_iz_2^k\right)^{(1)}r(z_2,z_3) - r(z_2,z_3)\left(b_iz_3^k\right)^{(2)}\right) 
    \end{split}
\end{equation}
Therefore, under consideration of \eqref{eq:series_vanishing_at_diagonal}, we can deduce that
\begin{equation}\label{eq:def_delta}
    \delta_r(a)(x,y) \coloneqq  r(x,y)a(x)^{(1)} - a(y)^{(2)}r(x,y) = -\left(\overline{a(x)^{(1)}r(x,y) - r(x,y)a(y)^{(2)}}\right)
\end{equation}
defines a continuous linear map \(\delta_r \colon A[\![z]\!] \to (A\otimes A)[\![x,y]\!]\). Applying
\begin{equation}
    \beta_{(n,\lambda)}^{\otimes 3}(b_{i_1}z_1^{k_1} \otimes r_{k_3,i_3}\otimes r_{k_2,j_2}, -)
\end{equation}
to \eqref{eq:NCYB_and_delta} results in
\begin{equation}
    \beta_{(n,\lambda)}\left(b_{i_1}z^{k_1},r_{k_2,i_2}r_{k_3,j_3}\right) = \beta_{(n,\lambda)}^{\otimes 2}\left(\delta_r\left(b_{i_1}z^{k_1}\right), r_{k_2,i_2} \otimes r_{k_3,j_3}\right)
\end{equation}
This proves that \(\delta_r\) is determined by \(((D_n(A),\beta_{(n,\lambda)}),A[\![z]\!],A(r))\) and thus defines a topological \(D\)-bialgebra structure in any full subcategory {\tt C} of {\tt Alg}\(_\Bbbk\) that is closed under taking subalgebras and contains \(D_n(A)\). Therefore,
\[((D_n(A),\beta_{(n,\lambda)}),A[\![z]\!],A(r)) \cong (D(A[\![z]\!],\delta_r),\textnormal{ev}),A[\![z]\!],A[\![z]\!]^\vee),\]
so \((A[\![z]\!],\delta_r)\) is a non-degenerate topological \(D\)-bialgebra in {\tt C}. In fact, the results in Section \ref{sec:solutions_of_CYBE_and_manin_triples} imply that every non-degenerate topological \(D\)-bialgebra structure in {\tt C} is of this form.

\subsection{Equivalence of solutions of the \(A\)-CYBE} 
Let \(n \in \bN\) and \((A,\beta)\) be a finite-dimensional metric \(\Bbbk\)-algebra. We call two solutions \(r_1,r_2 \in (A \otimes A)[\![x,y]\!][(x-y)^{-1}]\) of the \(A\)-CYBE \eqref{eq:cybe} \emph{equivalent}, written \(r_1 \sim r_2\), if there exists \(\varphi \in \textnormal{Aut}_{\Bbbk[\![z]\!]\textnormal{-alg}}(A[\![z]\!])\) and \(u \in z\Bbbk[\![z]\!]^\times\) such that 
\begin{equation}
    (\varphi(x) \otimes \varphi(y))r_1(u(x),u(y)) = r_2(x,y).
\end{equation}

\begin{lemma}\label{lem:equivalence_manintriple_rmatrices}
Let \(n \in \bN\), \((A,\beta)\) be a finite-dimensional, central, simple, metric \(\Bbbk\)-algebra, \(\lambda_\epsilon \in \Bbbk[\![z]\!]^\times\) and let \(r_\epsilon \in (A \otimes A)[\![x,y]\!][(x-y)^{-1}]\) be a solution of the \(A\)-CYBE of type \((n,\lambda_\epsilon)\) for \(\epsilon\in\{1,2\}\).

Then the following statements are equivalent:
\begin{itemize}
    \item \(r_1\) and \(r_2\) are equivalent.
    \item \((A[\![z]\!],\delta_{r_1}) \cong (A[\![z]\!],\delta_{r_2})\) as topological \(D\)-bialgebra structures in any full subcategory {\tt C} of {\tt Alg}\(_\Bbbk\) that is closed under taking subalgebras and contains \(D_n(A)\).
    \item \(((D_n(A),\beta_{(n,\lambda_1)}),A[\![z]\!],A(r_1)) \cong ((D_n(A),\beta_{(n,\lambda_2)}),A[\![z]\!],A(r_2))\).
\end{itemize}
\end{lemma}
\begin{proof}
The equivalence of the latter two items is already dealt with in Remark \ref{rem:isomorphism_of_top_Dbialgebras}. For the equivalence of the first two items, recall that any \(\phi \in \textnormal{Aut}_{\Bbbk\textnormal{-alg}}(A[\![z]\!])\) is of the form \(\phi(a)(z) = \varphi(z)a(u(z))\) for some \(\varphi \in \textnormal{Aut}_{\Bbbk[\![z]\!]\textnormal{-alg}}(A[\![z]\!])\) and \(u \in z\Bbbk[\![z]\!]^\times\); see \cite[Theorem 3.3]{abedin_maximox_stolin_zelmanov}. Now, its easy to see that
\begin{equation}
    (\varphi(x) \otimes \varphi(y))r_1(u(x),u(y)) = r_2(x,y).
\end{equation}
is equivalent to \((\phi \otimes \phi)\delta_{r_1} \phi^{-1} = \delta_{r_2}\), which means that \((A[\![z]\!],\delta_{r_1}) \cong (A[\![z]\!],\delta_{r_2})\) as topological \(D\)-bialgebras.
\end{proof}

\subsection{Solutions of the \(A\)-CYBE and triangular \(D\)-bialgebra structures}\label{rem:triangular_bialgebras_and_rmatrices}
It is also possible to consider solutions to the \(A\)-CYBE \eqref{eq:cybe} of the form \(r \in (A \otimes A)[\![x,y]\!]\). Namely, the assignment \(r \mapsto A(r)\) defines a bijection between:

\begin{itemize}
    \item skew-symmetric solutions \(r \in (A \otimes A)[\![x,y]\!]\) of the \(A\)-CYBE \eqref{eq:cybe} and
    \item subspaces \(W \subseteq A\otimes R_\infty\) such that \(((A \otimes R_\infty,\beta \otimes t_\infty),A[\![z]\!],W)\) is a Manin triple.
\end{itemize}
Moreover, \((A[\![z]\!],\delta_r)\) is a topological \(D\)-bialgebra structure (in any category of algebras closed under taking subalgebras that contains \(A \otimes R_\infty\)) determined by \(((A \otimes R_\infty,\beta \otimes t_\infty),A[\![z]\!],W)\). Therefore,
\begin{equation}\label{eq:triangular_Dbialgebras}
    ((A \otimes R_\infty,\beta \otimes t_\infty),A[\![z]\!],W) \cong ((D_n(A[\![z]\!],\delta_r),\textnormal{ev}),A[\![z]\!],A[\![z]\!]^\vee),
\end{equation}
so \(\delta_r\) is a triangular topological \(D\)-bialgebra structure. On the other hand, all triangular topological \(D\)-bialgebra structures on \(A[\![z]\!]\) are of this form.

Recall that a Lie bialgebra structure \((L,\delta)\) is called triangular if \(\delta = \delta_r\) for some skew-symmetric solution \(r \in L \otimes L\) of the CYBE. If \(( \fg[\![z]\!],\delta)\) is a topological Lie bialgebra structure for some Lie algebra \(\fg\), it is natural to replace \(\fg[\![z]\!]\otimes \fg[\![z]\!]\) by its completion \((\fg \otimes \fg)[\![x,y]\!]\) in this definition. In particular, it is natural to call \(\delta\) triangular if \(\delta = \delta_r\) for a skew-symmetric solution \(r \in (\fg \otimes \fg)[\![x,y]\!]\) of the CYBE. The \(D\)-bialgebra structures satisfying \eqref{eq:triangular_Dbialgebras} are then called triangular in analogy to their Lie counterparts.

\section{Refined categorization of non-degenerate topological \(D\)-bialgebras}\label{sec:categorization_general_whole_section}

In this section, we refine Theorem \ref{thm:categorization_of_manin_triples} for so-called strongly geometrically admissible algebras over algebraically closed fields of characteristic 0. The main result of this section, Theorem \ref{thm:categorization_refined}, can be seen as an analog of the main results from \cite{abedin_maximox_stolin_zelmanov} for a large class of non-Lie algebras. The proof relies on refining the geometric approach already used in the proof of Theorem \ref{thm:categorization_of_manin_triples}.

Throughout the remainder of this paper, \(\Bbbk\) is an algebraically closed field of characteristic 0.

\subsection{The main theorem}\label{sec:geometrically_M_admissible}
We call a metric \(\Bbbk\)-algebra \((A,\beta)\) \emph{strongly geometrically admissible} if
\begin{enumerate}
    \item \((A,\beta)\) is geometrically admissible in the sense of Subsection \ref{sec:geometrically_admissible_metrics};

    \item For any ringed \(A\)-lattices \((O,W)\) and any maximal ideal \(\mathfrak{m} \subseteq O\) such that
    \begin{itemize}
        \item \(W_{\mathfrak{m}}\) is free as \(O_{\mathfrak{m}}\)-module and

        \item the pairing \(W_{\mathfrak{m}}\times W_{\mathfrak{m}} \to O_{\mathfrak{m}}\) induced by \(\beta\) is perfect,
    \end{itemize}
    we have \(W/\mathfrak{m}W \cong A\).
\end{enumerate}

\noindent
As we will see in Corollary \ref{cor:geometrically_M_admissible} below, many central simple \(\Bbbk\)-algebras are strongly geometrically admissible, e.g.\ all finite-dimensional simple associative, Lie and Jordan algebras.
 
The rest of this section is dedicated to proving the following result.

\begin{theorem}\label{thm:categorization_refined}
Let us fix the following notation:
\begin{itemize}
    \item \(\Bbbk\) is an algebraically closed field of characteristic 0;
    
    \item \((A,\beta)\) is a unital strongly geometrically admissible metric \(\Bbbk\)-algebra (e.g.\ a finite-dimensional simple Jordan or associative \(\Bbbk\)-algebra) and \(\gamma \in A \otimes A\) is its canonical \(A\)-invariant element (see Subsection \ref{sec:series_in_standard_form});

    \item \(((D_n(A),\beta_{(n,\lambda)}),A[\![z]\!],W)\) is the Manin triple associated to \((A,\beta)\) as well as some \(n \in \bN\) and \(\lambda \in \Bbbk[\![z]\!]^\times\)in Subsection \ref{sec:manin_triples_over_series_explicit}; 

    \item \(r\) is the solution of the \(A\)-CYBE associated to the Manin triple \(((D_n(A),\beta_{(n,\lambda)}),A[\![z]\!],W)\) via Theorem \ref{thm:solutions_of_CYBE_and_manin_triples}.
\end{itemize}
Precisely one of the following cases occurs:
\begin{enumerate}
    \item \(n = 0\) and \(r\) is either:
    \begin{enumerate}
        \item \emph{Trigonometric} in the sense that there exists a \(\beta\)-orthogonal \(\sigma \in \textnormal{Aut}_{\Bbbk\textnormal{-alg}}(A)\) of order \(m \in \bN\) and \(s \in L(A,\sigma) \otimes L(A,\sigma)\) such that \(r\) is equivalent to
    \begin{equation*}
    \frac{1}{\exp\left(x-y\right)-1}\sum_{j = 0}^{m-1}\textnormal{exp}\left(\frac{x-y}{m}\right) \gamma_j + s\left(\exp\left(\frac{x}{m}\right),\exp\left(\frac{y}{m}\right)\right).
    \end{equation*}
    Here, \(L(A,\sigma) \subseteq A[\widetilde{v},\widetilde{v}^{-1}]\) is the loop algebra twisted by \(\sigma\) (see Proposition \ref{thm:classification_sheaves_of_algebras} for the definition) and \(\gamma_j \in A \otimes A\) is uniquely determined by \(\gamma = \sum_{j = 1}^d \gamma_j\) and \((\sigma \otimes 1)\gamma_j = \varepsilon^j\gamma_j\) for some primitive \(m\)-th root of unity \(\varepsilon \in \Bbbk\);
    
        \item \emph{Rational} in the sense that there exists \(t \in (A\otimes A)[x,y]\) such that \(r\) is equivalent to \[\frac{\gamma}{x-y} + t(x,y).\]
    \end{enumerate}
    
        \item \(n = 1\) and \(r\) is \emph{quasi-trigonometric} in the sense that there exists \(t \in (A\otimes A)[x,y]\) such that \(r\) is equivalent to \[\frac{y\gamma}{x-y} + t(x,y).\]
    
        \item \(n = 2\) and \(r\) is \emph{quasi-rational} in the sense that there exists \(t \in (A\otimes A)[x,y]\) such that \(r\) is equivalent to \[\frac{y^2\gamma}{x-y} + t(x,y).\]
    \end{enumerate}
In particular, every solution of the \(A\)-CYBE \eqref{eq:cybe} of type \((n,\lambda)\) is, up to equivalence, either trigonometric, rational, quasi-trigonometric or quasi-rational.
\end{theorem}

\begin{corollary}
Let \(\Bbbk\) be an algebraically closed field of characteristic 0, \((A,\beta)\) be a strongly geometrically admissible \(\Bbbk\)-algebra, and {\tt C} be a full subcategory of {\tt Alg}\(_\Bbbk\) closed under taking subalgebras and satisfying \(D_n(A) \in {\tt C}\). 

Then every non-degenerate topological \(D\)-bialgebra \((A[\![z]\!],\delta)\) in {\tt C} satisfies, up to isomorphism, \(\delta = \delta_r\) for a solution \(r\) of the \(A\)-CYBE which is either trigonometric, rational, quasi-trigonometric, or quasi-rational.    
\end{corollary}

\begin{remark}\label{rem:unitality}
    Let us note that the unitality assumption in Theorem \ref{thm:categorization_refined} is actually a rather weak assumption. Indeed, if a strongly geometrically admissible algebra is power-associative and not anti-commutative, it is a non-nil (see \cite{shestakov}) trace-admissible algebra. These are automatically unital; see \cite{albert}. 
\end{remark}

\noindent 
The proof of Theorem \ref{thm:categorization_refined} is again based on the geometrization scheme from Subsection \ref{sec:geometrization}. However, to refine the geometric approach already used in the proof of Theorem \ref{thm:categorization_of_manin_triples}, we need to establish some facts about \'etale locally trivial sheaves of algebras in Subsection \ref{sec:etale_locally_trivial_sheaves_of_algebras}. There we also explain how examples of strongly geometrically admissible algebras can be constructed using the notion of rigidity.  The results from Subsection \ref{sec:etale_locally_trivial_sheaves_of_algebras} and the refinement of Theorem \ref{thm:categorization_of_manin_triples} in Subsection \ref{sec:geometric_categorization_general} are then used to associate more explicit geometric data to Manin triples of the form \eqref{eq:manin_triples_over_series_explicit}. Namely, so-called geometric \(A\)-CYBE data, which will be defined in Subsection \ref{sec:geometric_A_CYBE_data}. We will assign such a datum to any Manin triple of the form \eqref{eq:manin_triples_over_series_explicit} in Subsection \ref{sec:geometrization_of_manin_triples}. Theorem \ref{thm:categorization_refined} is then a consequence of the classification results for sheaves of algebras from Proposition \ref{thm:classification_sheaves_of_algebras}.

\subsection{\'Etale locally trivial sheaves of algebras.}\label{sec:etale_locally_trivial_sheaves_of_algebras}
Let \(\sheafA\) be a sheaf of algebras on a \(\Bbbk\)-scheme \(X\). We call \(\sheafA\) \emph{\'etale \(A\)-locally free at a point \(p\in X\)}, for some \(\Bbbk\)-algebra \(A\), if there exists an \'etale morphism \(f \colon Y \to X\) such that \(p \in f(Y)\) and \(f^*\sheafA\) is isomorphic to \(A\otimes \sheafO_Y\) as \(\sheafO_Y\)-algebras. Furthermore, \(\sheafA\) is called \emph{\'etale \(A\)-locally free} if \(\sheafA\) is \'etale \(A\)-locally free at all points of \(X\). Let us remark that an \'etale \(A\)-locally free sheaf of algebras is automatically quasi-coherent and, if \(A\) is finite-dimensional, coherent.

\'Etale local triviality can actually be checked on fibers by virtue of the following result, which is an algebro-geometric version of \cite{kirangi_lie_algebra_bundles}, see \cite[Theorem 2.10]{abedin_universal_geometrization}.

\begin{proposition}\label{prop:weak_locall_free_implies_etale}
Let \(\Bbbk\) be an algebraically closed field of characteristic 0, \(X\) be a reduced \(\Bbbk\)-scheme of finite-type, \(A\) be a finite-dimensional \(\Bbbk\)-algebra, and \(\sheafA\) be a quasi-coherent sheaf of algebras on \(X\). Then \(\sheafA\) is \'etale \(A\)-locally free if and only if \(\sheafA|_p \cong A\) for all closed points \(p \in X\). 
\end{proposition}

\noindent
It turns out that a sheaf of algebras which can be \'etale trivialized by a unital algebra is automatically unital, i.e.\ we have the following result.

\begin{lemma}\label{lemm:etale_locally_free_implies_unital}
Let \(A\) be a unital algebra over a field \(\Bbbk\) and \(\sheafA\) be an \'etale \(A\)-locally free sheaf of algebras on a \(\Bbbk\)-scheme \(X\). Then \(\sheafA\) is unital. In particular, \(\textnormal{h}^0(\sheafA) > 0\).
\end{lemma}
\begin{proof}
Let \(U \subseteq X\) be an open subset and assume that \(U\) has an affine open covering \(\{U_i\}_{i \in I}\) such that \(\Gamma(U_i,\sheafA)\) is unital for all \(i \in I\). Since \(\Gamma(\textnormal{D}(f),\sheafA) = \Gamma(U_i,\sheafA)_f\) and \(\Gamma(U_i,\sheafA) \to \Gamma(\textnormal{D}(f),\sheafA)\) as well as \(\Gamma(\textnormal{D}(f),\sheafA) \to \Gamma(\textnormal{D}(fg),\sheafA)\) are unital for all \(f,g \in \Gamma(U_i,\sheafO_X)\) and \(i \in I\), a gluing argument shows that \(\Gamma(U_i \cap U_j,\sheafA)\) and \(\Gamma(U_i,\sheafA) \to \Gamma(U_i\cap U_j,\sheafA)\) are unital. Therefore, a second gluing argument implies that \(\Gamma(U,\sheafA)\) is unital. A similar consideration shows that \(\Gamma(U,\sheafA) \to \Gamma(V,\sheafA)\) is unital for all \(V \subseteq U\). We conclude that \(\sheafA\) is unital if and only if every \(p\in X\) has an affine open neighbourhood \(U\) such that \(\Gamma(U,\sheafA)\) is unital.

For every \(p\in X\), We can chose an irreducible affine open neighbourhood \(U\) of \(p\), an irreducible affine scheme \(U'\), and a surjective \'etale morphism \(f\colon U' \to U\) such that there exists an isomorphism \(\psi \colon B\otimes_R S \to A \otimes S\) of \(S\)-algebras, where \(B \coloneqq \Gamma(U,\sheafA),R \coloneqq  \Gamma(U,\sheafO_X)\), and \(S \coloneqq \Gamma(U',\sheafO_{U'})\).
The element \(\psi^{-1}(1) \in B\otimes_R S\) is a unit and \(\psi\) is unital. Since \(f\) is faithfully flat of finite type, we can recover \(B\) from \(B\otimes_R S\) as
\begin{equation}
    B= \{a \in B \otimes_R S \mid \phi(a\otimes 1) = 1 \otimes a\}
\end{equation}
where \(\phi \colon (B \otimes_R S) \otimes_R S \to S \otimes_R (B \otimes_R S)\) is defined by \(\phi((b\otimes s) \otimes t) = s \otimes (b \otimes t)\); see e.g.\ \cite[Remark 2.21]{milne_etale_cohomology}. In particular, \(B\) can be identified with an subalgebra of the unital algebra \(A\otimes S\). Since \(\phi\) is an isomorphism, it is unital. Therefore, \(B = \Gamma(U,\sheafA)\) contains the unit of \(B\otimes_R S\). Thus, the argument in the beginning of this proof implies that \(\sheafA\) is unital.

Now \(\textnormal{h}^0(\sheafA) > 0\) follows from the fact that \(1 \in \textnormal{H}^0(\sheafA)\).
\end{proof}

\subsubsection{Rigidity and strongly geometrically admissible algebras}
Consider the affine variety \(\textnormal{Alg}(d,\Bbbk) = \textnormal{Hom}(\Bbbk^d\otimes \Bbbk^d,\Bbbk^d)\) of all possible multiplication maps on \(\Bbbk^d\). There is a natural action of the group of invertible \(d \times d\)-matrices \(\textnormal{GL}(d,\Bbbk)\) given by
\begin{align}\label{eq:actionofG}
    &(g \cdot \vartheta)(v \otimes w) = g^{-1}\vartheta(gv \otimes gw) &\forall g \in \textnormal{GL}(d,\Bbbk), \vartheta \in \textnormal{Alg}(d,\Bbbk), v,w \in \Bbbk^d.
\end{align}
The orbit of a multiplication map under this action corresponds to the isomorphism class of the associated algebra.

Let \(M \subseteq \textnormal{Alg}(d,\Bbbk)\) be a \(\textnormal{GL}_n(d,\Bbbk)\)-invariant affine subvariety and write \(A \in M\) for an algebra \(A = (\Bbbk^d,\mu)\), if \(\mu \in M\). A \(\Bbbk\)-algebra \(A = (\Bbbk^d,\mu)\) is called \emph{\(M\)-rigid} if \(A \in M\) and the orbit 
\begin{equation}\label{eq:orbit_definition}
    O(A) \coloneqq \{A' = (\Bbbk^d,\mu') \in M\mid \mu' = g\mu \textnormal{ for some }g\in \textnormal{GL}(d,\Bbbk)\}
\end{equation}
contains an open neighbourhood of \(A\) in \(M\). 

A sufficient condition for \'etale local triviality is the rigidity of the fiber, as the following result, which is a algebro-geometric version of a generalization of \cite[Lemma 2.1]{kirangi_semi_simple}, states. 

\begin{proposition}\label{lemm:semisimple_fiber_implies_etale_triviality}
Let \(\Bbbk\) be an algebraically closed field of characteristic 0 and \(M\) be a \(\textnormal{GL}(d,\Bbbk)\)-stable subvariety of \(\textnormal{Alg}(d,\Bbbk)\). Furthermore, let \(\sheafA\) be a locally free sheaf of algebras on a reduced \(\Bbbk\)-scheme \(X\) such that \(\sheafA|_q \in M\) for all \(q \in X\) closed.

If \(\sheafA|_p\) is \(M\)-rigid for some closed point \(p \in X\), \(\sheafA\) is \'etale \(\sheafA|_p\)-locally free in \(p\). In particular, \(\sheafA|_q \cong \sheafA|_p\) for all closed points \(q\in X\) in some neighbourhood of \(p\).
\end{proposition}

\noindent 
The proof is a straight forward adaptation of the proof of \cite[Theorem 2.11]{abedin_universal_geometrization} to this setting. 
A consequence of Lemma \ref{lemm:semisimple_fiber_implies_etale_triviality} is the following important criterion for strong geometric admissibility.

\begin{proposition}\label{prop:geometrically_M_admissible}
    Let \(\Bbbk\) be an algebraically closed field of characteristic 0,  \(M \subseteq \textnormal{Alg}(d,\Bbbk)\) be a \(\textnormal{GL}_n(d,\Bbbk)\)-invariant affine subvariety, and \((A,\beta)\) metric \(\Bbbk\)-algebra in \(M\). 
    
    Then \((A,\beta)\) is strongly geometrically admissible if:
    \begin{enumerate}
        \item \((A,\beta)\) is geometrically admissible in the sense of Subsection \ref{sec:geometrically_admissible_metrics};

        \item For any ringed \(A\)-lattices \((O,W)\) and any maximal ideal \(\mathfrak{m} \subseteq O\) such that
        \begin{itemize}
            \item \(W_{\mathfrak{m}}\) is free as \(O_{\mathfrak{m}}\)-module and

            \item the pairing \(W_{\mathfrak{m}}\times W_{\mathfrak{m}} \to O_{\mathfrak{m}}\) induced by \(\beta\) is perfect
        \end{itemize}
        the \(\Bbbk\)-algebra \(W/\mathfrak{m}W\) is \(M\)-rigid.
\end{enumerate}
\end{proposition}

\begin{proof}
Let \(((X,\sheafA),(p,c,\zeta))\) is the geometric datum associated to \((O,W)\) in Subsection \ref{sec:geometrization}. Furthermore, let \(U\) be the set of closed points \(q \in X\) such that 
\begin{itemize}
    \item \(\sheafA_q\) is a free \(\sheafO_{X,q}\)-module;

    \item The restriction \(\zeta(\sheafA_q) \times \zeta(\sheafA_q) \to c(\sheafO_{X,q})\) of \(\beta\) from \eqref{eq:beta_extension} is non-degenerate.
\end{itemize}
Then \(\sheafA|_q\) is \(M\)-rigid for all \(q \in U\) by assumption. Combining Proposition \ref{lemm:semisimple_fiber_implies_etale_triviality} and \(p \in U\), \(U\) is a non-empty open subset of the set of closed points of \(X\). In particular, \(U\) is connected since \(X\) is irreducible. 
 
Furthermore, every \(q \in U\) has an open neighbourhood \(U' \subseteq U\) such that \(\sheafA|_q \cong \sheafA|_{q'}\) holds for all \(q' \in U'\) by virtue of Proposition \ref{lemm:semisimple_fiber_implies_etale_triviality}. The connectedness of \(U\) therefore implies that \(\sheafA|_q \cong \sheafA|_p \cong A\) for all \(q \in U\). This implies that \((A,\beta)\) is strongly geometrically admissible.
\end{proof}

\noindent
Consider \(M \in \{\textnormal{Lie}_d,\textnormal{Ass}_d, \textnormal{Jor}_d\}\) where 
\(\textnormal{Lie}_d,\textnormal{Ass}_d, \textnormal{Jor}_d \subseteq \textnormal{Alg}(d,\Bbbk)\)
are the varieties of \(d\)-dimensional Lie, associative, and Jordan algebras respectively. 
In Subsection \ref{sec:geom_admissible_metrics_examples}, we discussed that any simple \(A \in M\) has an, up to multiplication by a scalar, unique algebra metric \(\beta\) and that the metric algebra \((A,\beta)\) is geometrically admissible. 

Let \((O,W)\) be a ringed \(A\)-lattice and \(\mathfrak{m} \subseteq O\) be a maximal ideal such that \(W_{\mathfrak{m}}\) is free. If the restriction \(W_{\mathfrak{m}} \times W_{\mathfrak{m}} \to O_{\mathfrak{m}}\) of \(\beta\) from \eqref{eq:beta_extension} is non-degenerate, the \(\Bbbk\)-algebra \(W/\mathfrak{m}W \in M\) inherits an algebra metric from \(\beta\). This algebra metric can be explicitly described using the formula in Subsection \ref{sec:geom_admissible_metrics_examples} and we see from this description that \(W/\mathfrak{m}W\) is semi-simple, i.e.\ a direct sum of simple subalgebras. If \(M = \textnormal{Lie}_d\) we use Cartan's criterion for semi-simplicity and if \(M \in \{\textnormal{Ass}_d,\textnormal{Jor}_d\}\) this is a consequence of general results on trace-admissible algebras; see e.g.\ \cite{albert}.  

Since semi-simple algebras in \(M\) are rigid (see  \cite{hazewinkel_gerstenhaber} for the case that \(M \in \{\textnormal{Lie}_d, \textnormal{Ass}_d\}\) and \cite{finston} for the case that \(M = \textnormal{Jor}_d\)), we see that \((A,\beta)\) satisfies the conditions of Proposition \ref{prop:geometrically_M_admissible}. Therefore, we obtain the following result.

\begin{corollary}\label{cor:geometrically_M_admissible}
    Any finite-dimensional simple Lie, associative, or Jordan algebra over an algebraically closed field of characteristic 0 is strongly geometrically admissible if equipped with its (up to scalar multiple) unique algebra metric.
\end{corollary}

\subsubsection{Sheaves of algebras on one-dimensional affine algebraic groups}
Recall that over an algebraically closed field of characteristic 0, a connected affine algebraic group over \(\Bbbk\) of dimension one is 
either isomorphic to the affine line or the punctured affine line. Let us conclude this subsection with a classification of all sheaves of algebras with constant fibers on these schemes; see \cite[Theorem 6.1.1]{abedin_thesis} for a proof.

\begin{proposition}\label{thm:classification_sheaves_of_algebras}
Let \(A\) be a finite-dimensional algebra over an algebraically closed field \(\Bbbk\) of characteristic 0.

\begin{enumerate}
    \item Let \(B\) be a \(\Bbbk[v,v^{-1}]\)-algebra satisfying \(B/(v-\lambda)B \cong A\) for all \( \lambda \in \Bbbk^\times\). Then there exists \(\sigma \in \textnormal{Aut}_{\Bbbk\textnormal{-alg}}(A)\) of order \(m \in \bN\) such that
    \begin{equation*}
        B \cong L(A,\sigma) \coloneqq \{a \in A[\widetilde{v},\widetilde{v}^{-1}]\mid a\left(\exp\left(2\pi i/m\right)\widetilde{v}\right) = \sigma(a(\widetilde{v}))\}  
    \end{equation*}
    as \(\Bbbk[v,v^{-1}]\)-algebras. Here, the \(\Bbbk[v,v^{-1}]\)-module structure of \(L(A,\sigma)\) is defined by \(\widetilde{v}^m = v\).
    
    \item Let \(B\) be an \(\Bbbk[z]\)-algebra satisfying \(B/(z-\lambda)B \cong A\) for all \( \lambda \in \Bbbk\). Then \(B \cong A[z]\) as \(\Bbbk[z]\)-algebras.
\end{enumerate}

\end{proposition}

\subsection{Geometric \(A\)-CYBE datum}\label{sec:geometric_A_CYBE_data}
We call a triple \((X,(\sheafA,\beta_\sheafA))\) \emph{geometric \(A\)-CYBE datum} for a finite-dimensional \(\Bbbk\)-algebra \(A\) if:
\begin{itemize}
    \item \(X\) is an irreducible cubic plane curve over \(\Bbbk\);
    
    \item \(\sheafA\) is a coherent sheaf of algebras on \(X\) such that: 
    \begin{enumerate}
        \item \(\textnormal{H}^0(\sheafA) = 0 = \textnormal{H}^1(\sheafA)\);
        
        \item \(\sheafA|_p \cong A\) for all smooth closed \(p \in X\);
    \end{enumerate}
    \item \(\beta_\sheafA\colon \sheafA \times \sheafA \to \sheafO_X\) is a symmetric, perfect, associative \(\sheafO_X\)-bilinear form.\smallskip
\end{itemize}

\noindent 
The name ``geometric \(A\)-CYBE datum'' will become clear in Subsection \ref{sec:geometric_rmatrix}.

\begin{remark}\label{rem:irreducible_cubic_plane}
It is well-known that any irreducible plane cubic curve \(X\) over \(\Bbbk\) is defined by an equation \(y^2 = x^3 + ax + b\) and precisely one of the following cases occurs:
\begin{enumerate}
    \item \(X\) is smooth if and only if \(4b^3 + 27a^2 \neq 0\), in which case \(X\) is an elliptic curve.
    \item \(X\) has a unique nodal singularity if \(4b^3 = -27a^2 \neq 0\). In this case, \(X\setminus \{s\}\) is isomorphic to the punctured affine line \(\textnormal{Spec}(\Bbbk[v,v^{-1}])\).
    \item \(X\) has a unique cuspidal singularity if \(a = b = 0\). In this case, \(X\setminus\{s\}\) is isomorphic to the affine line \(\textnormal{Spec}(\Bbbk[z])\).
\end{enumerate}
Let us note that, up to isomorphism, irreducible cubic plane curves are precisely irreducible projective curves over \(\Bbbk\) of arithmetic genus 1.
\end{remark}

\noindent
The following lemma is important for the identification of geometric \(A\)-CYBE data below.

\begin{lemma}\label{lemm:geometric_A_CYBE_datum}
Let \(\Bbbk\) be an algebraically closed field of characteristic 0, \(X\) be an irreducible plane cubic curve over \(\Bbbk\), and \(\sheafF\) be coherent sheaf on \(X\) satisfying:
\begin{itemize}
    \item \( \textnormal{H}^1(\sheafF) = 0\);
    
    \item There exists a non-degenerate symmetric \(\sheafO_X\)-bilinear form \(\beta_\sheafF\colon \sheafF \times \sheafF \to \sheafO_X\).
\end{itemize}

\noindent
Then \(\beta_\sheafF\) is perfect. In particular, \(\beta_\sheafF|_p\) is non-degenerate for all \(p \in X\).
\end{lemma}

\begin{proof}
Since \(\beta_\sheafF\) is non-degenerate, we have a short exact sequence
\begin{equation}
    0 \longrightarrow \sheafF \stackrel{\beta_\sheafF^\textnormal{a}}\longrightarrow \sheafF^* \longrightarrow \mathcal{C}\longrightarrow 0,
\end{equation}
where \(\mathcal{C} \coloneqq \textnormal{Cok}(\beta_\sheafF^\textnormal{a})\) is a torsion sheaf. We obtain the long exact sequence in cohomology
\begin{align}\label{eq:tildeKexactseq}
    0 \longrightarrow \textnormal{H}^0(\mathcal{F}) {\longrightarrow} \textnormal{H}^0(\mathcal{F}^*) \longrightarrow \textnormal{H}^0(\mathcal{C})\longrightarrow \textnormal{H}^1(\mathcal{F}) {\longrightarrow} \textnormal{H}^1(\mathcal{F}^*) \longrightarrow \textnormal{H}^1(\mathcal{C}) \longrightarrow 0.
\end{align}
The dualizing sheaf of any irreducible cubic plane curve is trivial, so Serre duality implies that \(\textnormal{h}^0(\mathcal{F}^*) = \textnormal{h}^1(\mathcal{F}) = 0\) and thus \(\textnormal{H}^0( \mathcal{C}) = 0\). Since \( \mathcal{C}\) is a torsion sheaf, we see that \(\mathcal{C} = 0\), so \(\beta_\sheafF^\textnormal{a}\) is an isomorphism. 
\end{proof}

\subsubsection{Geometric solutions of the \(A\)-CYBE}\label{sec:geometric_rmatrix}
In this subsection, we will repeat the construction of geometric solutions of the usual CYBE from \cite{burban_galinat} in our general setting. This will result in a construction of geometric solutions of a geometric \(A\)-CYBE from a geometric \(A\)-CYBE datum. In particular, this explains the name ``geometric \(A\)-CYBE datum''.

Let \(A\) be a finite-dimensional \(\Bbbk\)-algebra, \((X,(\sheafA,\beta_\sheafA))\) be a geometric \(A\)-CYBE datum, and \(C \subseteq X\) be the set of smooth points. Fix a global section \(\eta \in \textnormal{H}^0(\omega_X)\) of the dualizing sheaf \(\omega_X\) of \(X\). Let us remark that \(\omega_X\) can be identified with the sheaf of so-called Rosenlicht-regular 1-forms; see e.g. \cite[Section 5.2]{conrad}. Consider the \emph{diagonal residue sequence}
\begin{align}\label{eq:resseq}
    0 \longrightarrow \sheafO_{X \times C} \longrightarrow \sheafO_{X\times C}(\Delta) \stackrel{\textnormal{res}^\eta_\Delta}{\longrightarrow} \delta_*\sheafO_C\longrightarrow 0.
    \end{align}
Here, \(\Delta\subseteq X \times C\) is the image of \(\delta\colon C \to X \times C\) defined by \(p \mapsto (p,p)\). Furthermore, \(\textnormal{res}^\eta_\Delta\) is determined by \((u_1-u_2)^{-1} \mapsto \mu\) locally around any closed point \(p\in C\), where:
\begin{itemize}
    \item \(u\) is a local parameter of \(p\) defined on an affine open subset \(U\) of \(C\);
    \item \(\omega_C\) and \(\sheafO_{X\times C}(-\Delta)\) are locally generated by \(\textnormal{d}u\) and
    \[u_1-u_2 \coloneqq u \otimes 1 - 1 \otimes u \in \Gamma(U,\sheafO_X) \otimes \Gamma(U,\sheafO_X) \cong \Gamma(U \times U,\sheafO_{X \times X})\]
    respectively, after potentially shrinking \(U\);

    \item \(\eta_p = \mu \textnormal{d}u\) holds for some uniquely determined \(\mu \in \Gamma(U,\sheafO_X)\). \smallskip
\end{itemize}

\noindent
The tensor product of \eqref{eq:resseq} with \(\sheafA \boxtimes \sheafA|_C \coloneqq \textnormal{pr}_1^*\sheafA \otimes_{\sheafO_{X\times C}}\textnormal{pr}_2^*\sheafA|_C\) for the canonical projections \(X \stackrel{\textnormal{pr}_1}{\longleftarrow} X\times C \stackrel{\textnormal{pr}_2}{\longrightarrow} C\) gives the short exact sequence
\begin{align}\label{eq:Resseq}
    0 \longrightarrow \sheafA \boxtimes \sheafA|_C \longrightarrow \sheafA \boxtimes \sheafA|_C(\Delta) \longrightarrow \delta_*(\sheafA|_C \otimes_{\sheafO_C} \sheafA|_C)\longrightarrow 0.
\end{align}
Using the K\"unneth formula and \(\textnormal{H}^0(\sheafA) = 0=\textnormal{H}^1(\sheafA)\) results in
\begin{equation}
    \begin{split}
        &\textnormal{H}^0(\sheafA \boxtimes \sheafA|_C) = \textnormal{H}^0(\sheafA) \otimes \textnormal{H}^0(\sheafA|_C) =  0 \textnormal{ and }
        \\&\textnormal{H}^1(\sheafA \boxtimes \sheafA|_C) = \left(\textnormal{H}^1(\sheafA) \otimes \textnormal{H}^0( \sheafA|_C)\right) \oplus \left(\textnormal{H}^0(\sheafA) \otimes \textnormal{H}^1(\sheafA|_C)\right) = 0.
    \end{split}
\end{equation}
Therefore, the long exact sequence in cohomology induced by
\eqref{eq:Resseq} results in an isomorphism \(R\colon\textnormal{H}^0(\sheafA \boxtimes \sheafA|_C(\Delta)) \to \textnormal{H}^0(\sheafA|_C \otimes \sheafA|_C)\).

The pairing \(\beta_\sheafA\) of \(\sheafA\) induces an isomorphism \(
    B\colon \sheafA|_C\otimes_{\sheafO_C} \sheafA|_C \to \sheafEnd_{\sheafO_C}
   (\sheafA|_C)\) defined by 
\begin{equation}
    a \otimes b \longmapsto \beta_\sheafA(b,-)a \textnormal{ for all }U\subseteq C \textnormal{ affine open }a,b \in \Gamma(U,\sheafA).
\end{equation}
Combined with \(R\), we obtain an isomorphism \(\Phi\coloneqq BR\colon\textnormal{H}^0(\sheafA \boxtimes \sheafA|_C(\Delta)) \to \textnormal{End}_{\sheafO_C}(\sheafA|_C)\).

Consider the section \(\rho\coloneqq \Phi^{-1}(\textnormal{id}_{\sheafA|_C}) \in \textnormal{H}^0(\sheafA \boxtimes \sheafA|_C(\Delta))\). Then, if \(U \subseteq C\) is any affine open subset such that \(\eta = \mu^{-1}\textnormal{d}u\) for \(\mu,u \in \Gamma(U,\sheafO_X\), we can write 
\begin{equation}\label{eq:rho_local_form}
    \rho|_{U \times U} = \frac{(1 \otimes \mu)\chi}{u_1-u_2} + s
\end{equation}
for some \(s \in \Gamma(U\times U,\sheafA \boxtimes \sheafA)\), where \(\chi \in \Gamma(U\times U,\sheafA \boxtimes \sheafA)\) is any preimage of \(\textnormal{id}_{\sheafA|_U}\) under the surjective map 
\[\Gamma(U\times U,\sheafA \boxtimes \sheafA) \longrightarrow \Gamma(U,\sheafA \otimes \sheafA) \to \End_{\sheafO_U}(\sheafA|_U).\]
One should think of \eqref{eq:rho_local_form} as an analog of the standard form \eqref{eq:standard_form} of \((n,\lambda)\)-type series.

By repeating the arguments in \cite[Theorem 3.11 and Theorem 4.3]{burban_galinat}, we can see that
\begin{equation}
    \rho^{13}\rho^{12} - \rho^{12}\rho^{23} + \rho^{13}\rho^{23} = 0.
\end{equation}
Here, the summands on the left-hand side can be understood as a rational section of \(\sheafA \boxtimes \sheafA \boxtimes \sheafA\) by adapting the notations from Subsection \ref{sec:definition_CYBE} to this geometric setting. In particular, \(\rho\) is a solution of a geometric version of the \(A\)-CYBE.

\subsection{Geometrization of Manin triples}\label{sec:geometrization_of_manin_triples}
Recall that \(\Bbbk\) is an algebraically closed field of characteristic 0. Let \((A,\beta)\) be a strongly geometrically admissible \(\Bbbk\)-algebra, \(n \in \mathbb{N}\), \(\lambda \in \Bbbk[\![z]\!]^\times\), and \(((D_n(A),\beta_{(n,\lambda)}),A[\![z]\!],W)\) be the Manin triple associated to that data in Subsection \ref{sec:manin_triples_over_series_explicit}. Recall from Theorem \ref{thm:categorization_of_manin_triples} that \(n \in \{0,1,2\}\) and that we may assume \(\lambda = 1\).

The goal of this section is to assign a geometric \(A\)-CYBE datum to \(((D_n(A),\beta_{(n,1)}),A[\![z]\!],W)\).

\subsubsection{Geometrization in case \(n = 0\)}\label{sec:geometrization_n=0}
There exists a particular \(O \subseteq M \coloneqq\{f \in \Bbbk(\!(z)\!)\mid f W \subseteq W\}\) such that \(\dim(\Bbbk(\!(z)\!)/(\Bbbk[\![z]\!] + O)) = 1\); see Subsection \ref{sec:geometric_categorization_general}. Namely, the integral closure \(N\) of \(M\) either satisfies \(\dim(\Bbbk(\!(z)\!)/(\Bbbk[\![z]\!]+N)) = 1\) or \(N= \Bbbk[u]\) for some \(u \in z^{-1} + z\Bbbk[\![z]\!]^\times\), and then
\begin{equation}\label{eq:multipliers_explicit}
    O \coloneqq \begin{cases}
            N &\textnormal{if }\dim(\Bbbk(\!(z)\!)/(\Bbbk[\![z]\!]+N)) = 1,\\
            \Bbbk[u',u'u]&\textnormal{if }N= \Bbbk[u].
        \end{cases}
\end{equation}
Applying the geometrization procedure form Section \ref{sec:geometrization} to \((O,W)\) gives a geometric datum \(((X,\sheafA),(p,c,\zeta))\). Observe that \(\sheafA\) satisfies \(\zeta \colon \widehat{\sheafA}_p \stackrel{\cong}\longrightarrow A[\![z]\!]\) and \(\textnormal{H}^0(\sheafA) = 0 = \textnormal{H}^1(\sheafA)\) (combine \(A(\!(z)\!) = A[\![z]\!] \oplus W\) with \eqref{eq:cohomology}). 

If \(X\) is smooth it is a smooth irreducible cubic plane curve. If \(O = \Bbbk[u',u'u]\) for \(u \neq z^{-1}\), we have \(-u' = u^2 - a\) for some \(a \in \Bbbk \setminus\{0\}\). This equation is equivalent to \(u^{\prime,2}(a-u') = u^{\prime,2}u^2\), so putting \(y = u'u\) and \(x = u'\), we see that \(X\) is a nodal irreducible cubic plane curve. Finally, if \(O = \Bbbk[z^{-2},z^{-3}]\), \(X\) is clearly a cuspidal irreducible cubic plane curve. 

In order to see that \((X,\sheafA)\) gives rise to a geometric \(A\)-CYBE datum, we have to construct an appropriate \(\sheafO_X\)-bilinear map \(\beta_\sheafA \colon \sheafA \times \sheafA \to \sheafO_X\).

If \(X\) is smooth (which is equivalent to the fact that \(O\) is integrally closed), the geometrically admissible metric \(\beta\) defines a pairing \(\beta_\sheafA \colon \sheafA \times \sheafA \to \sheafO_X\); see
Subsection \ref{sec:properties_geometriaclly_admissible_pairing}.(2). 

Let us assume that \(X\) is singular, i.e.\ \(O = \Bbbk[u',u'u]\). Since \(\beta\) is geometrically admissible, it induces a pairing \(W \times W \to N = \Bbbk[u]\). Since \(W\) is Lagrangian, the image under this pairing lies in the kernel of \(\textnormal{res}_0\) restricted to \(\Bbbk[u]\). It is easy to see that this kernel is equal to \(O\), so the coefficient-wise application of \(\beta\) defines a map \(W \times W \to O\). It is now straight forward to see that for every \(U \subseteq X\), the commutative diagram
\begin{equation}
    \xymatrix{\Gamma(U,\sheafA)\times \Gamma(U,\sheafA) \ar[r]\ar[d]_{\zeta \times \zeta} & \Gamma(U,\sheafO_X)\ar[d]^{c} \\ A(\!(z)\!)\times A(\!(z)\!) \ar[r]_-{\beta} & \Bbbk(\!(z)\!)}
\end{equation}
defines a pairing \(\beta_\sheafA \colon \sheafA \times \sheafA \to \sheafO_X\). 

\begin{lemma} \label{lemm:sheafA_has_fiber_A}
The triple \((X,(\sheafA,\beta_{\sheafA}))\) is a geometric \(A\)-CYBE datum. In particular, \(\sheafA|_q \cong A\) for all smooth closed \(q \in X\).
\end{lemma}
\begin{proof}
    Since by construction the fiber of \(\beta_\sheafA\) is \(\beta\) at \(p \in X\), the kernel of the canonical morphism \(\sheafA \to \sheafA^*\) is a torsion subsheaf of the torsion free sheaf \(\sheafA\), hence zero. In other words, \(\beta_\sheafA\) is non-degenerate. By virtue of Lemma \ref{lemm:geometric_A_CYBE_datum}, \(\beta_\sheafA\) is perfect.

    It remains to prove that \(\sheafA|_q \cong A\) for all smooth closed \(q \in X\). Observe that \(\sheafA|_p \cong A\) already holds, so we may assume \(q \neq p\). Let \(\mathfrak{m} \subseteq O\) be the maximal ideal associated to
    \[q \in X \setminus \{p\} \cong \textnormal{Spec}(O).\]
    Then \(\beta_{\sheafA,q}\) is identified with the restriction \(W_{\mathfrak{m}} \times W_{\mathfrak{m}} \to O_{\mathfrak{m}}\) of \(\beta\) from \eqref{eq:beta_extension} through \(\sheafA_q \cong W_{\mathfrak{m}}\) and \(\sheafO_{X,q} \cong O_{\mathfrak{m}}\). Since \(A\) is strongly geometrically admissible (recall the definition from Subsection \ref{sec:geometrically_M_admissible}) and \(\beta_{\sheafA,q}\) is perfect, we obtain \(\sheafA|_q \cong W/\mathfrak{m}W \cong A\).
\end{proof}

\noindent
Let us note the following important consequence of Lemma \ref{lemm:sheafA_has_fiber_A}.

\begin{proposition}\label{lemm:concluding_bd_trich_case}
The following results are true.
\begin{enumerate}
    \item If \(X\) is elliptic, \(A\) is non-unital.

    \item Let \(X\) be singular and \(\rho\) be the section constructed in Subsection \ref{sec:geometric_rmatrix} from \((X,(\sheafA,\beta_\sheafA))\) and \(\eta \coloneqq dv \in \textnormal{H}^0(\omega_X)\), where \(v \coloneqq u/u'\) is the local generator of \(p\) associated to the representation \eqref{eq:multipliers_explicit}.
    Then image of \(\rho\) under the Taylor expansion \[\textnormal{H}^0(\sheafA \boxtimes \sheafA|_C(\Delta)) \longrightarrow \varprojlim_k\left(\Gamma(X\setminus\{p\},\sheafA)\otimes \widehat{\sheafA}_p/\mathfrak{m}_p\widehat{\sheafA}^k_p\right)\stackrel{\zeta \otimes \zeta}{\longrightarrow} (A(\!(z)\!)\otimes A)[\![z]\!]\]
    at \(X \times \{p\}\) trivialized by \((c,\zeta)\) is equivalent to the solution \(r\) of the \(A\)-CYBE associated to the Manin triple \(((A(\!(z)\!),\beta_{(0,0)}),A[\![z]\!],W)\).
     \end{enumerate}
\end{proposition}

\begin{proof}[Proof of (1)]
Assume \(X\) is elliptic and \(A\) is unital. According to Lemma \ref{lemm:sheafA_has_fiber_A} and Proposition \ref{prop:weak_locall_free_implies_etale}, \(\sheafA\) is \(A\)-\'etale locally free. By virtue of Lemma \ref{lemm:etale_locally_free_implies_unital}, this contradicts \(\textnormal{h}^0(\sheafA) = 0\). Therefore, \(A\) is non-unital if \(X\) is elliptic.
\end{proof}

\begin{proof}[Proof of (2)]
The proof is a straight forward repetition of the the proof of \cite[Theorem 3.17]{abedin_universal_geometrization} (see also \cite[Theorem 3.3.3]{abedin_thesis}).
\end{proof}

\subsubsection{Geometrization in case \(n = 1\)}\label{sec:geometrization_n=1}
Let \(W_+\) (resp.\ \(W_-\)) be the projection of 
\[W \subseteq D_1(A) = A(\!(z)\!) \times A\]
onto the \(A(\!(z)\!)\) (resp.\ \(A\)) component. By virtue of Section \ref{sec:categorization_general}.(4), \(\Bbbk[z^{-1}] W_+ \subseteq W_+\) and we can consider the geometrization \(
((Y,\mathcal{W}),(p,c,\zeta))\) of \((\Bbbk[z],W_+)\), where \(Y = \mathbb{P}^1\) and \( s_- \coloneqq p = (z)\). Let \(s_+ \in \mathbb{P}^1\) be the point corresponding to the ideal \((z^{-1}) \subseteq \Bbbk[z^{-1}]\) via \(c(\Gamma(\mathbb{P}^1\setminus\{s_-\},\mathcal{O}_X)) = \Bbbk[z^{-1}]\).

Let the sheaf of algebras \(\mathcal{V}\) on \(\mathbb{P}^1\) be defined as the pull-back 
\begin{equation}
    \xymatrix{\mathcal{V} \ar[r] \ar[d] 
    & W_- \ar[d] \\
    \mathcal{W} \ar[r] &  \mathcal{W}|_{s_-} \cong A}
\end{equation}
where \(A\), \(W_-\), and \(\mathcal{W}|_{s_-}\) are understood as skyscraper sheaves at \(s_-\).
In other words, \(\mathcal{V}\) fits into the short exact sequence
\begin{equation} \label{eq:short_exseq_VW}
    0 \longrightarrow \mathcal{V} \longrightarrow \mathcal{W} \oplus W_- \longrightarrow A \longrightarrow 0.
\end{equation}
Since the morphism \(W_- \to A\) is injective, the morphism \(\mathcal{V}\to\mathcal{W}\) is too and we can identify \(\mathcal{V}\) with a subsheaf of \(\mathcal{W}\). Let \(\beta_{\mathcal{W}} \colon \mathcal{W} \times \mathcal{W} \to \sheafO_{\mathbb{P}^1}\) be the pairing induced by \(\beta\) in Subsection \ref{sec:properties_geometriaclly_admissible_pairing}.(2) and \(\beta_{\mathcal{V}} \colon \mathcal{V} \times \mathcal{V} \to \sheafO_{\mathbb{P}^1}\) be the restriction to \(\mathcal{V}\).

\begin{lemma}\label{lemm:thetas}
The following is true:
\begin{enumerate}
    \item \(\textnormal{H}^0(\mathcal{V}) \cong \iota(\fg[\![x]\!]) \cap (W_+ \times W_-) \), \(\textnormal{H}^1(\mathcal{V}) = 0\)
    and \(\mathcal{V}|_{\mathbb{P}^1\setminus\{s_-\}} = \mathcal{W}|_{\mathbb{P}^1\setminus\{s_-\}}\);
    
    \item There exist canonical surjective morphisms \(\theta_\pm \colon \mathcal{V}|_{s_\pm} \to W_\pm/W_\pm^\bot\)
    such that \[\beta_{\mathcal{V}}|_{s_\pm}(a,b) = \beta_{(1,1)}^\pm(\theta_\pm(a),\theta_\pm(b))\]
    holds for all \(a,b \in \mathcal{V}|_{s_\pm}\). 
\end{enumerate}
\end{lemma}
\begin{proof}[Proof of (1)]
The restriction of the short exact sequence \eqref{eq:short_exseq_VW} to \(\mathbb{P}^1\setminus\{s_-\}\) results in \(\mathcal{V}|_{\mathbb{P}^1\setminus\{s_-\}} = \mathcal{W}|_{\mathbb{P}^1\setminus\{s_-\}}\).
Since the first cohomology group of torsion sheaves vanishes and \(A[\![z]\!] + W_+ = A(\!(z)\!)\) implies \(\textnormal{H}^1(\mathcal{W}) = 0\) because of \eqref{eq:cohomology}, the long exact sequence of \eqref{eq:short_exseq_VW} in cohomology reads
\begin{equation}\label{eq:long_exact_cohomology_V}
    0 \longrightarrow \textnormal{H}^0(\mathcal{V}) \longrightarrow \textnormal{H}^0(\mathcal{W}) \oplus W_- \longrightarrow A \longrightarrow \textnormal{H}^1(\mathcal{V}) \longrightarrow 0.
\end{equation}
In particular, \(\textnormal{H}^0(\mathcal{W})  \cong A[\![z]\!]\cap W_+\) implies
\[\textnormal{H}^0(\mathcal{V}) \cong A[\![z]\!] \cap (W_+ \times W_-).\]
The image \(\overline{W}_+\) of \(A[\![z]\!]\cap W_+ \to A\) under the canonical map \(A[\![z]\!] \to A\) satisfies \(\overline{W}_+ + W_- = A\). Therefore, \(\textnormal{H}^0(\mathcal{W}) \oplus W_- \to A\) in 
\eqref{eq:long_exact_cohomology_V}
is surjective. Consequently, \(\textnormal{H}^1(\mathcal{V}) = 0\).
\end{proof}

\begin{proof}[Proof of (2)]
The algebra \(W_+\) is a torsion-free as \(\Bbbk[z^{-1}]\)-module, so it is a free \(\Bbbk[z^{-1}]\)-module. Since \(\beta\) is geometrically admissible, this implies \(\beta(W_+,W_+)\subseteq \Bbbk[z^{-1}]\).  In particular,
\begin{equation}
    \beta_{(1,1)}^+(z^{-1}a,b) = \textnormal{res}_0 z^{-2}\beta(a,b)  = 0
\end{equation}
for all \(a,b \in W_+\), so \(z^{-1}W_+ \subseteq W_+^\bot\). Therefore, we have a surjective morphism 
\[
\theta_+\colon \mathcal{V}|_{s_+} \cong W_+/z^{-1}W_+ \longrightarrow W_+/W_+^\bot
\] 
intertwining the corresponding forms.

On the other hand, the construction of \(\mathcal{V}\) as pull-back gives a canonical map \(\mathcal{V} \to W_-\) which is surjective since
\(\mathcal{W} \to \mathcal{W}|_{s_-}\)
is surjective. This morphism factors through
a surjective morphism \(\mathcal{V}|_{s_-} \to W_-\) which respects the forms and this map induces \(\theta_-\). 
\end{proof}

\noindent
Let \(X\) be an irreducible cubic plane curve with nodal singularity \(s\) and chose the normalization \(\nu \colon \mathbb{P}^1 \to X\) in such a way that \(\nu^{-1}(s) = \{s_+,s_-\}\). Let us understand \(W_\pm/W_\pm^\bot\) as skyscraper sheaf at \(s_\pm\) and let \(\theta\) be the direct image under \(\nu\) of the morphism 
\begin{equation}
     \mathcal{V} \longrightarrow \mathcal{V}|_{s_+}\times \mathcal{V}|_{s_-} \stackrel{(\theta_+,\theta_-)}\longrightarrow W_+/W_+^\bot \times W_-/W_-^\bot
\end{equation}
for \(\theta_\pm\) from Lemma \ref{lemm:thetas}.
Let \(\mathcal{A}\) be defined as pull-back 
\begin{equation}
    \xymatrix{\mathcal{A}\ar[r]\ar[d]&
    W/(W_+^\bot \times W_-^\bot)\ar[d]\\ \nu_*\mathcal{V}\ar[r]_-{\theta}&\nu_*(W_+/W_+^\bot \times W_-/W_-^\bot)}
\end{equation}
where \(W/(W_+^\bot \times W_-^\bot)\) is viewed as skyscraper sheaves at \(s\in X\). Again, this is equivalent to the short exact sequence
\begin{equation}\label{eq:A_qt_case_kernel_def}
    0 \longrightarrow \mathcal{A} \longrightarrow \nu_*\mathcal{V} \oplus (W/(W_+^\bot \times W_-^\bot)) \longrightarrow W_+/W_+^\bot \times W_-/W_-^\bot \longrightarrow 0.
\end{equation}
Therefore, \(\sheafA\) can be identified with a subsheaf of \(\nu_*\mathcal{V}\). Let \(\beta_\sheafA\colon \mathcal{A}\times \mathcal{A}\to \nu_*\mathcal{O}_{\mathbb{P}^1}\) be the restriction of \(\nu_*\beta_{\mathcal{V}}\) to \(\mathcal{A}\), where we recall that \(\beta_{\mathcal{V}}\colon \mathcal{V}\times \mathcal{V} \to \mathcal{O}_{\mathbb{P}^1}\) is obtained by restriction from \(\beta_{\mathcal{W}}\).

\begin{lemma}\label{lemm:A_weakly_g_locally_free_qtcase}
The datum \((X,(\sheafA,\beta_\sheafA))\) is a geometric \(A\)-CYBE datum. In particular, \(\sheafA|_q \cong A\) for all smooth closed \(q \in X\).
\end{lemma}
\begin{proof}
The long exact sequence in cohomology of \eqref{eq:A_qt_case_kernel_def} is given by
\begin{equation}\label{eq:long_exact_sequence_cohomology_A_qt}
    0 \longrightarrow \textnormal{H}^0(\mathcal{A}) \longrightarrow \textnormal{H}^0(\mathcal{V}) \oplus (W/(W_+^\bot \times W_-^\bot)) \longrightarrow W_+/W_+^\bot \times W_-/W_-^\bot \longrightarrow \textnormal{H}^1(\mathcal{A})\longrightarrow 0.
\end{equation}
Here, we used that the first cohomology group of torsion sheaves vanishes and \(\textnormal{H}^1(\mathcal{V}) = 0\); see Lemma \ref{lemm:thetas}.(1). 
The canonical map \(\textnormal{H}^0(\mathcal{V})\to W_+/W_+^\bot \times W_-/W_-^\bot\) thereby coincides with the inclusion 
\[
A[\![z]\!] \cap (W_+ \times W_-) \longrightarrow W_+/W_+^\bot \times W_-/W_-^\bot
\]
under the identification \(\textnormal{H}^0(\mathcal{V}) \cong A[\![z]\!] \cap (W_+ \times W_-)\).
Therefore, Subsection \ref{lem:manin_triples}.(3) implies that the middle arrow in \eqref{eq:long_exact_sequence_cohomology_A_qt} is an isomorphism. Consequently,
\(\textnormal{H}^0(\mathcal{A}) = 0 = \textnormal{H}^1(\mathcal{A})\), so property (1) of in the definition of a geometric \(A\)-CYBE datum in Section \ref{sec:geometric_A_CYBE_data} is satisfied.

Next, we want to see that \(\beta_\sheafA\) actually takes values in \(\mathcal{O}_X \subseteq \nu_*\mathcal{O}_{\mathbb{P}^1}\). Let \(a,b\in\mathcal{A}|_s\) and \(a_\pm,b_\pm \in W_\pm\) be representatives of the images of \(a,b\) under the canonical maps \(\mathcal{A}|_s \to \mathcal{V}|_{s_\pm} \to W_\pm/W_\pm^\bot\). Then
\begin{equation}
    \beta_\sheafA|_s(a,b) = (\beta^+_{(1,1)}(a_+,b_+),\beta_{(1,1)}^-(a_-,b_-)) \in \Bbbk \times \Bbbk \cong \nu_*\mathcal{O}_{\mathbb{P}^1}|_s
\end{equation}
holds, since the \(\theta_\pm \colon \mathcal{V}|_{s_\pm} \to W_\pm/W_\pm^\bot\) intertwine
the forms \(\beta_{(1,1)}^\pm\) and \(\beta_{\mathcal{V}}|_{s_\pm}\). 
The definition of \(\mathcal{A}\) implies that \(
    (a_+,a_-),(b_+,b_-) \in W
    \),
and the Lagrangian property of \(W\) gives
\begin{equation}
    0 = \beta_{(1,1)}((a_+,a_-),(b_+,b_-)) = \beta_{(1,1)}^+(a_+,b_+) - \beta^-_{(1,1)}(a_-,b_-).
\end{equation}
We obtain \(\beta_{\sheafA}|_s(a,b)\in \{(\lambda,\lambda)\mid \lambda \in \Bbbk\}\). This implies that \(\beta_{\sheafA}\) takes values in \(\mathcal{O}_X\).

Since \(\beta_{\mathcal{W}}|_{s_-}\) can be identified with the algebra metric \(\beta\) on \(\mathcal{W}|_{s_-} \cong A\), \(\beta_{\mathcal{W}}\), and consequently \(\beta_{\sheafA}\), is non-vanishing on an open subset of \(X\). Combined with Proposition \ref{lemm:semisimple_fiber_implies_etale_triviality} this implies that there exists a closed point \(q \in \mathbb{P}^1\setminus\{s_+,s_-\}\) such that \(A \cong \mathcal{W}|_q \cong \mathcal{A}|_{\nu(q)}\) and \(\beta_{\sheafA}|_{\nu(q)}\) is a non-zero associative bilinear form on this space. In particular, \(\beta_{\sheafA}|_q\) is automatically non-degenerate. Therefore, the kernel of the canonical morphism \(\sheafA \to \sheafA^*\) induced by \(\beta_\sheafA\) is a torsion subsheaf of the torsion free sheaf \(\sheafA\), hence zero. Consequently, \(\beta_\sheafA\) is non-degenerate.

Lemma \ref{lemm:geometric_A_CYBE_datum} now states that \(\beta_\sheafA\) is perfect. Consequently, for all closed \(q \in \mathbb{P}^1\setminus\{s_+,s_-\}\) the bilinear form \(\beta_{\mathcal{W},q}\), which can be identified with \(\beta_{\sheafA,\nu(q)}\) via \(\mathcal{W}_q   \cong\sheafA_{\nu(q)}\), is perfect. Since \((A,\beta)\) is strongly geometrically admissible, this implies that \(W_+/(z-a)W_+ \cong \mathcal{W}|_q   \cong\sheafA|_{\nu(q)}\) is isomorphic to \(A\). Here, for \(a \in \Bbbk^\times\) the maximal ideal \((z-a) \subseteq \Bbbk[z,z^{-1}]\) defines the point \(q \in \mathbb{P}^1\setminus\{s_+,s_-\} \cong \textnormal{Spec}(\Bbbk[z,z^{-1}])\). In conclusion, \((X,(\sheafA,\beta_\sheafA))\) is a geometric \(A\)-CYBE datum.
\end{proof}

\begin{proposition}\label{lemm:concluding_qtcase}
There exists a \(\varphi \in \textnormal{Aut}_{\Bbbk[\![z]\!]\textnormal{-alg}}(A[\![z]\!])\), unique 
\[\{t_{k,i} \in A[z] \mid k \in \bN,i\in\overline{1,n}\},\]
and \(N \in \bN\) such that
\begin{equation}
    \varphi(W) = \textnormal{Span}_{\Bbbk}\{w_{k,i} + t_{k,i} \mid k \in \bN,i\in \overline{1,n}\}
\end{equation}
and \(t_{k,i} = 0\) for all \( k\ge N\), \(i \in \overline{1,n}\). Here, the \(w_{k,i}\) were defined in \eqref{eq:def_wki}.
\end{proposition}
\begin{proof}
Lemma \ref{lemm:A_weakly_g_locally_free_qtcase} implies \( \mathcal{W}|_{q}\cong \mathcal{A}|_{\nu(q)} \cong A\) for all \(q \in \mathbb{P}^1\setminus\{s_+,s_-\}\). Combined with \(\mathcal{W}|_{s_-} \cong A\), this implies that \(B \coloneqq \zeta(\Gamma(\mathbb{P}^1\setminus\{s_+\},\mathcal{W}))\subseteq A[\![z]\!]\) is a free \(\Bbbk[z] = c(\Gamma(\mathbb{P}^1\setminus\{s_+\},\mathcal{O}_{\mathbb{P}^1}))\)-algebra satisfying \(B/(z-\lambda)B \cong A\) for all \(\lambda \in \Bbbk\). Therefore, \(B \cong A[z]\) by virtue of Theorem \ref{thm:classification_sheaves_of_algebras}.(2). Completing said automorphism in the \((z)\)-adic topology yields \(\varphi \in \textnormal{Aut}_{\Bbbk[\![x]\!]\textnormal{-alg}}(A[\![z]\!])\) with the property \(\varphi(B) = A[z]\). Since \(\mathcal{W}\) is a sheaf, we have
\begin{equation}
    \varphi(W_+) = \varphi(\zeta(\Gamma(\mathbb{P}^1\setminus\{s_-\},\mathcal{W})) \subseteq \varphi(\zeta(\Gamma(\mathbb{P}^1\setminus\{s_+,s_-\},\mathcal{W}))) = \varphi(B)[z^{-1}] = A[z,z^{-1}].
\end{equation}
This, combined with the fact that \(W_+\) is a free \(\Bbbk[z^{-1}]\)-module, implies that \(\varphi(W_+) \subseteq z^{N-1}A[z^{-1}]\) holds for a sufficiently large integer \(N \in \bN\). Consequently,
\begin{equation}\label{eq:W_bounded_qt_case}
    z^{-N}A[z^{-1}] \subseteq \varphi(W_+)^\bot \subseteq \varphi(W_+) \subseteq z^{N-1}A[z^{-1}].  
\end{equation}
Since \(A[\![z]\!] \oplus \varphi(W_+) = D_1(A)\), we can now write
\begin{equation}
    \varphi(W) = \textnormal{Span}_{\Bbbk}\{w_{k,i} + t_{k,i} \mid k \in \bN,i\in \overline{1,n}\} 
\end{equation}
for uniquely determined \(\{t_{k,i} \in A[\![z]\!] \mid k \in \bN,i\in\overline{1,n}\}\). Equation \eqref{eq:W_bounded_qt_case} now implies that \(t_{k,i} \in A[z]\) and \(t_{k,i} = 0\) for all \(k \ge N\).
\end{proof}

\subsubsection{Geometrization in case \(n = 2\)}\label{sec:geometrization_n=2}
Similar to the previous case, 
\[
W \subseteq D_2(A) = A(\!(z)\!) \times A[z]/x^2A[z]    
\]
and we denote by
\( W_+ \) (resp.\ \(W_-\)) the projection of \(W\) to
\( A(\!(z)\!)\) (resp.\
\(A[z]/z^2A[z]\)).

\begin{lemma}\label{lemm:quasirational_lemma}
The following facts are true.
\begin{enumerate}
    \item \(W = W_+ \times W_-\);
    
    \item \(W_+ \cap z^2A[\![z]\!] = \{0\}\), so \(W_+\cap A[\![z]\!]\) can be identified with a subalgebra of \(A[z]/z^2A[z]\);
    
    \item \((W_+ \cap A[\![z]\!]) \oplus W_- = A[z]/z^2A[z] \).
\end{enumerate}
\end{lemma}
 \begin{proof}
For (1), observe that \(\beta(W_+,W_+) \subseteq \Bbbk[z^{-1}]\) holds since  \((A,\beta)\) is geometrically admissible and \(W_+\) is a free \(\Bbbk[z^{-1}]\)-algebra by virtue of Section \ref{sec:geometric_categorization_general}.(5). Therefore, \(z^{-2}\beta(W_+,W_+) \subseteq z^{-2}\Bbbk[z^{-1}]\) implies \(
    \beta_{(2,1)}(a,b) = \textnormal{res}_0 z^{-2}\beta(a,b) = 0
    \)
for all \(a,b \in W_+\). Consequently, \(W_+ \subseteq W_+^\bot\). Together with \(W_+^\bot \subseteq W_+\) we arrive at \(W_+ = W_+^\bot\). Subsection \ref{lem:manin_triples}.(3) implies \(W_- = W_-^\bot\), so 
\[W_+^\bot \times W_-^\bot \subseteq W \subseteq W_+ \times W_-\] concludes the proof of (1).

The identities \(\{0\} = (A[\![z]\!] + W_+)^\bot = z^2A[\![z]\!] \cap W_+^\bot = z^2A[\![z]\!] \cap W_+\) imply (2). Part (3) now follows from (2) and \(A[\![z]\!] \oplus (W_+ \times W_-) = A(\!(z)\!) \times A[z]/z^2A[z]\).
\end{proof}

\noindent
Consider the geometrization \(((Y,\mathcal{W}),(p,c,\zeta))\) of \((\Bbbk[z^{-1}],W_+)\), where as in the last section we have \(Y = \mathbb{P}^1\). 
Let \(X\) be an irreducible plane cubic curve with cuspidal singularity \(s\) and chose the normalization \(\nu \colon \mathbb{P}^1 \to X\) in such a way that \(\nu(p) = s\). 
The isomorphism 
\(\zeta \colon \widehat{\mathcal{W}}_p \to A[\![z]\!]\) implies that 
\begin{equation}
    \nu_*\mathcal{W}|_s \cong \zeta(\widehat{\mathcal{W}}_p)/z^2\zeta(\widehat{\mathcal{W}}_p) = A[z]/z^2A[z].  
\end{equation}
This yields a surjective morphism \(\nu_* \mathcal{W} \to A[z]/z^2A[z]\).
Let \(\mathcal{A}\) be the sheaf of algebras defined as the pull-back
\begin{equation}\label{eq:def_sheafA_qrcase}    \xymatrix{\mathcal{A} \ar[r]\ar[d] & W_- \ar[d] \\ \nu_*\mathcal{W} \ar[r] & A[z]/z^2A[z]}
\end{equation}
where \(A[z]/z^2A[z]\) and \(W_-\) are understood as skyscraper sheaves at \(s\). Equivalently, \(\sheafA\) fits into the short exact sequence
\begin{equation}\label{eq:short_exact_sequence_quasirational_case}
    0 \longrightarrow \mathcal{A} \longrightarrow \nu_*\mathcal{W} \oplus W_- \longrightarrow A[x]/z^2A[z] \longrightarrow 0.
\end{equation}
Let \(\beta_{\mathcal{W}} \colon \mathcal{W} \times \mathcal{W}\to \sheafO_{\mathbb{P}^1}\) be the pairing induced by \(\beta\) in Subsection \ref{sec:properties_geometriaclly_admissible_pairing} and let \(\beta_\sheafA\colon \mathcal{A}\times \mathcal{A} \to \nu_*\mathcal{O}_{\mathbb{P}^1}\) be the the restriction of \(\nu_*\beta_{\mathcal{W}}\) to \(\mathcal{A} \subseteq \nu_*\mathcal{\mathcal{W}}\).

\begin{lemma}\label{lem:geometic_A_CYBE_datum_qrcase}
The datum \((X,(\sheafA,\beta_\sheafA))\) is a geometric \(A\)-CYBE datum. In particular, \(\sheafA|_q \cong A\) for all smooth closed \(q \in X\). Furthermore, there exists \( \varphi\in \textnormal{Aut}_{\Bbbk(\!(z)\!)\textnormal{-alg}}(A(\!(z)\!))\) such that the identity \(\varphi(W_+) = A[z^{-1}]\) holds.
\end{lemma}
\begin{proof}
The global section of \(\nu_* \mathcal{W} \to A[z]/z^2A[z]\) coincides with the canonical morphism \[A[\![z]\!]\cap W_+ \to A[z]/z^2A[z]\]
if  \(\textnormal{H}^0(\mathcal{W})\) is identified with \(A[\![z]\!] \cap W_+\). Therefore, the middle arrow in the long exact sequence in cohomology 
\begin{equation}\label{eq:long_exact_sequence_cohomology_A_qr}
    0 \longrightarrow \textnormal{H}^0(\mathcal{A}) \longrightarrow \textnormal{H}^0(\mathcal{W}) \oplus W_- \longrightarrow A[z]/z^2A[z] \longrightarrow \textnormal{H}^1(\mathcal{A})\longrightarrow 0
\end{equation}
of \eqref{eq:short_exact_sequence_quasirational_case} is an isomorphism by virtue of Lemma \ref{lemm:quasirational_lemma}.(3). 
Here we used again that:
\begin{itemize}
    \item \(\textnormal{H}^1(\mathcal{W}) = 0\) by virtue of Subsection \ref{lem:manin_triples}.(2) and \eqref{eq:cohomology};
    \item The first cohomology group of torsion sheaves vanishes.
\end{itemize}
Consequently, \(\textnormal{H}^0(\mathcal{A}) = 0 = \textnormal{H}^1(\mathcal{A})\).

Let us now show that \(\beta_\sheafA \colon \sheafA \times \sheafA \to \nu_*\sheafO_X\) takes values in \(\sheafO_X\).
For any \(a,b \in \mathcal{A}|_s\) we have  
\begin{equation}
    \nu_*\beta_{\mathcal{W}}|_{s}(a,b) = \beta(a_1,b_1) + [z](\beta(a_1,b_2) +  \beta(a_2,b_1))  \in \Bbbk[z]/(z^2),
\end{equation}
where \(a_1 + [z]a_2\) and \(b_1 + [z]b_2 \in A[z]/z^2A[z]\) are the images of \(a \) and \(b\) respectively
under 
\[\mathcal{A}|_s \to \nu_*\mathcal{W}|_s \cong A[z]/z^2A[z].
\]
By definition of \(\mathcal{A}\), \(a_1 + [z]a_2\), \(b_1 + [z]b_2 \in W_- \) and \(\beta(a_1,b_2) +  \beta(a_2,b_1) = 0\) since \(W_-\subseteq A[z]/z^2A[z]\) is Lagrangian. Therefore, \(\beta_\sheafA|_{s}(a,b) = \beta(a_1,b_1)\in \Bbbk\),
implying that \(\beta_\sheafA\) takes values in \(\mathcal{O}_X \subseteq \nu_*\mathcal{O}_{\mathbb{P}^1}\). 

Repeating the arguments in the end of the proof of Lemma \ref{lemm:A_weakly_g_locally_free_qtcase}, we can deduce that \((X,(\sheafA,\beta_\sheafA))\) is a geometric \(A\)-CYBE datum and \(\sheafA|_q\cong A\) for all smooth closed \(q\in X\). Now \eqref{eq:def_sheafA_qrcase} implies that \( \mathcal{W}|_{q}\cong \mathcal{A}|_{\nu(q)} \cong A\) for all \(q \in \mathbb{P}^1\setminus\{p\}\). Consequently,
\begin{equation*}
    W_+ = \zeta(\Gamma(\mathbb{P}^1\setminus\{p\},\mathcal{W}))\subseteq A(\!(z)\!)
\end{equation*}
is a free \(\Bbbk[z^{-1}] = c(\Gamma(\mathbb{P}^1\setminus\{p\},\mathcal{O}_{\mathbb{P}^1}))\)-algebra satisfying \(W_+/(z^{-1}-\lambda)W_+ \cong A\) for all \(\lambda \in \Bbbk\). Therefore, \(W_+ \cong A[z^{-1}]\) by virtue of Theorem \ref{thm:classification_sheaves_of_algebras}.(2). This induces the automorphism \(\varphi\).
\end{proof}

\noindent
We can now copy the arguments of Proposition \ref{lemm:concluding_qtcase} to deduce that.
\begin{proposition}\label{lemm:concluding_qrcase}
There exists a \(\varphi \in \textnormal{Aut}_{\Bbbk[\![z]\!]\textnormal{-alg}}(A[\![z]\!])\), a set
\begin{equation*}
    \{t_{k,i} \in A[z] \mid k \in \bN,i\in\overline{1,n}\},
\end{equation*}
and a natural number \(N \in \bN\) such that
\begin{equation}
    \varphi(W) = \textnormal{Span}_{\Bbbk}\{w_{k,i} + t_{k,i} \mid k \in \bN,i\in \overline{1,n}\}
\end{equation}
and \(t_{k,i} = 0\) for all \( k\ge N\), \(i \in \overline{1,n}\). Here, the \(w_{k,i}\) were defined in \eqref{eq:def_wki}.
\end{proposition}

\subsection{Proof of Theorem \ref{thm:categorization_refined}}\label{sec:categorization_general}
Recall the notation and statement of Theorem \ref{thm:categorization_refined}:
\begin{itemize}
    \item \(\Bbbk\) is an algebraically closed field of characteristic 0;
    
    \item \((A,\beta)\) is a unital strongly geometrically admissible metric \(\Bbbk\)-algebra and \(\gamma \in A \otimes A\) is its canonical \(A\)-invariant element;
    
    \item \(((D_n(A),\beta_{(n,\lambda)}),A[\![z]\!],W)\) is a Manin triple of the form \ref{eq:manin_triples_over_series_explicit} for some \(n \in \bN\) and \(\lambda \in \Bbbk[\![z]\!]^\times\).
    
    \item \(r\) is the solution of the \(A\)-CYBE associated to the Manin triple \(((D_n(A),\beta_{(n,\lambda)}),A[\![z]\!],W)\) via Theorem \ref{thm:solutions_of_CYBE_and_manin_triples}. 
\end{itemize}
Then precisely one of the following cases occurs: \begin{enumerate}
    \item If \(n = 0\), the curve \(X\) from the \(A\)-CYBE datum \((X,\sheafA)\) of \(((D_n(A),\beta_n),A[\![z]\!],W)\) constructed in Subsection \ref{sec:geometrization_n=0} is either a nodal or cuspidal irreducible cubic plane curve. Furthermore:
    \begin{enumerate}
        \item \(X\) is nodal if and only if \(r\) is  {trigonometric} in the sense of Theorem \ref{thm:categorization_refined};
    
    \item \(X\) is cuspidal if and only if \(r\) is {rational} in the sense in the sense of Theorem \ref{thm:categorization_refined};
    \end{enumerate}
    
    \item \(n = 1\) if and only if \(r\) is {quasi-trigonometric} in the sense of Theorem \ref{thm:categorization_refined};
    
    \item \(n = 2\) if and only if \(r\) is {quasi-rational} in the sense of Theorem \ref{thm:categorization_refined}.
\end{enumerate}

\subsubsection{Proof of (1)}
First of all, since \(A\) is unital, \(X\) cannot be elliptic by virtue of Proposition \ref{lemm:concluding_bd_trich_case}.(1). Therefore, \(X\) is either a nodal or a cuspidal irreducible plane cubic curve. Let \(s \in X\) be the unique singularity in both cases.

Let \(\eta\) and \(\rho\) be as in Proposition \ref{lemm:concluding_bd_trich_case}.(2) and chose isomorphisms
\begin{equation}
    C \coloneqq X \setminus \{s\} \stackrel{f}\longrightarrow \begin{cases}
        \textnormal{Spec}(\Bbbk[v,v^{-1}]) & \textnormal{if }X \textnormal{ is nodal}; \\
        \textnormal{Spec}(\Bbbk[z]) & \textnormal{if }X \textnormal{ is cuspidal}
    \end{cases}
\end{equation}
such that 
\begin{equation}
    \eta = \begin{cases}
        v^{-1}dv &\textnormal{if }X \textnormal{ is nodal};\\
        dz&\textnormal{if }X\textnormal{ if cuspidal}.
    \end{cases}
\end{equation}
In both cases we can chose \(U = C\) in \eqref{eq:rho_local_form} in order to obtain
\begin{equation}
    \rho|_{C \times C} = \frac{(1 \otimes \mu)\chi}{u_1 - u_2} + s
\end{equation}
where \(u = v\) (resp.\ \(u = z\)) and \(\mu = v\) (resp.\ \(\mu = 1\)) if \(X\) is nodal (resp.\ if \(X\) is cuspidal). Recall that \(s\) is some element in \(\textnormal{H}^0(\sheafA|_C \boxtimes \sheafA|_C) = \textnormal{H}^0(\sheafA|_C) \otimes \textnormal{H}^0(\sheafA|_C)\) and \(\chi\) is some preimage of \(\textnormal{id}_{\sheafA|_C}\) under
\[\textnormal{H}^0(\sheafA|_C \boxtimes \sheafA|_C) \longrightarrow \textnormal{H}^0(\sheafA|_C \otimes_{\sheafO_C}\sheafA|_C) \longrightarrow \End_{\sheafO_C}(\sheafA|_C).\]

Using Lemma \ref{lemm:sheafA_has_fiber_A} and Theorem \ref{thm:classification_sheaves_of_algebras} we can see that there exists a \(f^\sharp\)-equivariant isomorphism
\begin{equation}
    \textnormal{H}^0(\sheafA|_{C}) \stackrel{\phi_1}\longrightarrow \begin{cases}
        L(A,\sigma) &\textnormal{if }X\textnormal{ is nodal};\\
        
        A[z]&\textnormal{if }X\textnormal{ is cuspidal},
    \end{cases}
\end{equation}
where in the nodal case \(\sigma \in \textnormal{Aut}_{\Bbbk\textnormal{-alg}}(A)\) is of finite order.
Here, \(f^\sharp\) is the map \(\Bbbk[v,v^{-1}] \to \Gamma(C,\sheafO_X)\) (resp.\ \(\Bbbk[z] \to \Gamma(C,\sheafO_X)\)) defined by \(f\) if \(X\) is nodal (resp.\ cuspidal). Let us conclude the proof of (1) in a case by case fashion.

\begin{proof}[Case (a): \(X\) is nodal]
Let \(A_j \coloneqq \{a \in A \mid \sigma(a) = \varepsilon^ja\}\) for the \(m\)-th root of unity
\(\varepsilon \in \Bbbk\) from Theorem \ref{thm:categorization_refined}. 
Note that \(\beta_\sheafA\) induces an algebra metric \(L(A,\sigma) \times L(A,\sigma) \to \Bbbk[v,v^{-1}]\) defined by the coefficient-wise application of \(\beta\). In particular, since \(v = \widetilde{v}^m\) and
\begin{equation}
    \beta(\widetilde{v}^ka,\widetilde{v}^\ell) = \beta(a,b)\widetilde{v}^{k+\ell} \in \Bbbk[v,v^{-1}]
\end{equation}
holds for all \(a\in A_k,b \in A_\ell\), we have \(\beta(A_k,A_\ell) = \{0\}\) if \(k + \ell \notin m \bZ\). Furthermore,
\[\beta(\sigma(a),b) = \varepsilon^k \beta(a,b) = \varepsilon^{k + \ell - \ell}\beta(a,b) = \beta(a,\sigma^{-1}(b))\]
holds for \(k+\ell \in m\bZ\). Combined, we see that \(\sigma\) is orthogonal with respect to \(\beta\).

Since \(\sigma\) is orthogonal with respect to \(\beta\), it is easy to see that \(\gamma = \sum_{j = 0}^{m-1}\gamma_j \in \bigoplus_{j = 0}^{m-1}(A_j \otimes A_{-j})\). We can choose \(\chi\) as the preimage of 
\begin{equation}
    \sum_{j = 0}^{m-1}\left(\frac{\widetilde{v}}{\widetilde{w}}\right)^j \gamma_j \in L(A,\sigma) \otimes L(A,\sigma)
\end{equation}
under the isomorphism \(\phi_1 \otimes \phi_1 \colon \textnormal{H}^0(\sheafA|_C) \otimes \textnormal{H}^0(\sheafA|_C) = \textnormal{H}^0(\sheafA|_C \boxtimes \sheafA|_C) \to L(A,\sigma) \otimes L(A,\sigma)\). Then
\begin{equation}
     (\phi_1 \otimes \phi_1)\rho|_{C \times C} = \frac{1}{(v/w)-1}\left(\frac{\widetilde{v}}{\widetilde{w}}\right)^j \gamma_j + t 
\end{equation}
holds for \(t \coloneqq (\phi_1\times \phi_1)s \in L(A,\sigma) \otimes L(A,\sigma)\). 

Let \(\exp\) be the completion of \(\Bbbk[v,v^{-1}] \to \Bbbk[\![z]\!]\), \(v \mapsto \exp(z)\) with respect to the ideal \((v-1)\) and \(\phi_2\in \textnormal{Aut}_{\Bbbk\textnormal{-alg}}(A[\![z]\!])\) be the \(\exp\)-equivariant isomorphism obtained by completing the map \(L(A,\sigma) \to A[\![z]\!]\), \(f \mapsto f(\exp(z/m))\) at the same ideal. Using Proposition \ref{lemm:concluding_bd_trich_case}.(2), we can see that the automorphism \(\phi \coloneqq \phi_2\phi_1\zeta^{-1} \in \textnormal{Aut}_{\Bbbk\textnormal{-alg}}(A[\![z]\!])\) satisfies 
\begin{equation}
    (\phi \otimes \phi)r(x,y) =\frac{1}{\exp\left(x-y\right)-1}\sum_{j = 0}^{m-1}\textnormal{exp}\left(\frac{x-y}{m}\right) \gamma_j + s\left(\exp\left(\frac{x}{m}\right),\exp\left(\frac{y}{m}\right)\right).
\end{equation}
This concludes the proof in the nodal case.
\end{proof}

\begin{proof}[Case (b): \(X\) is cuspidal.] We can chose \(\chi \in \textnormal{H}^0(\sheafA|_C) \otimes \textnormal{H}^0(\sheafA|_C)\) as the preimage of \(\gamma \in (A \otimes A)[x,y]\) under the isomorphism \(\phi_1 \otimes \phi_1\).
Then  
\begin{equation}
    (\phi_1 \otimes \phi_1)\rho|_{C\times C} = \frac{\gamma}{x-y} + t
\end{equation}
holds for \(t \coloneqq (\phi_1 \otimes \phi_1)s \in (A \otimes A)[x,y]\). 

Let \(\phi_2 \in \textnormal{Aut}_{\Bbbk[\![z]\!]\textnormal{-alg}}(A[\![z]\!])\) be the completion of \(A[z] \to A[\![z]\!]\). Using Proposition \ref{lemm:concluding_bd_trich_case}.(2), we can see that 
\begin{equation}
    (\phi \otimes \phi)r = \frac{\gamma}{x-y} + t
\end{equation}
holds for \(\phi \coloneqq \phi_2 \phi_1 \zeta^{-1} \in \textnormal{Aut}_{\Bbbk[\![z]\!]\textnormal{-alg}}(A[\![z]\!]) \). This concludes the proof in the cuspidal case.
\end{proof}

\subsubsection{Proof of (2) and (3)}
By virtue of Proposition \ref{lemm:concluding_qtcase} and Proposition \ref{lemm:concluding_qrcase} there exist 
\[\{t_{k,i} \in A[z] \mid k \in \bN,i\in\overline{1,n}\}\]
and \(N \in \bN\) such that, up to isomorphism of Manin triples,
\begin{equation}
    W = \textnormal{Span}_{\Bbbk}\{w_{k,i} + t_{k,i} \mid k \in \bN,i\in \overline{1,n}\}
\end{equation}
and \(t_{k,i} = 0\) for all \( k\ge N\). Here, the \(w_{k,i} \in D_n(A)\) are defined in \eqref{eq:def_wki}.
The solution \(r\) of the \(A\)-CYBE of \(W\) can now be determined by 
\begin{equation}
    r(x,y) = \sum_{k = 0}^\infty \sum_{i = 1}^d(w_{k,i} + t_{k,i}) \otimes b_iy^k = \frac{y^n\gamma}{x-y} + t(x,y),
\end{equation}
where \(t = \sum_{k = 0}^N \sum_{i = 1}^d t_{k,i}(x) \otimes b_iy^k \in (A\otimes A)[x,y]\).

\section{Classification of associative \(D\)-bialgebra structures over series}\label{sec:classification_associatice_Dbialgebras}

\subsection{Non-triangular topological associative \(D\)-bialgebras on series are non-degenerate}\label{sec:associative_topological_Dbialgebras_nondegenerate}
The final goal of this paper is the classification of all non-triangular topological associative \(D\)-bialgebra structures on \(A[\![z]\!]\) (i.e.\ topological \(D\)-bialgebra structures in the category of associative algebras) for any finite-dimensional simple associative algebra \(A\) over an algebraically closed field \(\Bbbk\) of characteristic 0. 
Recall that these are exactly the co-opposites of (non-triangular) topological balanced infinitesimal \(D\)-bialgebra structures on \(A[\![z]\!]\). Therefore, the classification of the latter is equivalent.

In order to use Theorem \ref{thm:categorization_refined}, we begin by proving that, as in the case of a simple Lie algebra over \(\Bbbk\), these are all non-degenerate.

\begin{proposition}\label{prop:associative_doubles}
Let \(\Bbbk\) be algebraically closed of characteristic 0 and \((A,\beta)\) be a finite-dimensional, simple, associative, metric \(\Bbbk\)-algebra, i.e.\ \(A \cong \textnormal{M}_{n}(\Bbbk)\) is the space of \(n \times n\)-matrices with entries in \(\Bbbk\) and \(\beta\) is a scalar multiple of the algebra metric defined by the trace of matrices.

Any non-triangular topological associative \(D\)-bialgebra structure \(\delta \colon A[\![z]\!] \to (A\otimes A)[\![x,y]\!]\) is non-degenerate in the sense of Section \ref{sec:manin_triples_over_series_explicit}.
\end{proposition}

\subsubsection{Proof of Proposition \ref{prop:associative_doubles}}
Let us begin by proving.

\begin{lemma}\label{lem:extensions_of_scalers}
Let \(\Bbbk\) be algebraically closed of characteristic 0 and \(A\) be a finite dimensional associative \(\Bbbk\)-algebra.
Every associative algebra \(B\) containing \(A\) as subalgebra is isomorphic to \(A \otimes R\) for some unital associative \(\Bbbk\)-algebra \(R\). Furthermore, if \(B\) is equipped with an algebra metric \(\widetilde{\beta}\), then for all \(a,b \in A\) and \(r,s \in R\)
    \begin{equation}
        \widetilde{\beta}(a \otimes f,b\otimes g) = \beta(a,b) t(rs) 
    \end{equation}
    for some \(t \colon R \to \Bbbk\) such that the associated pairing \((r,s) \mapsto t(rs)\) is an algebra metric of \(R\).
\end{lemma}

\begin{proof} 
The algebra \(B\) splits into a direct sum of irreducible \(A\)-bimodules: \(B = \bigoplus_{i \in I} A r_i A\), where \(I \coloneqq \{r \in B\mid ArA\textnormal{ is irreducible}\}/\sim\) for \(r \sim s\) if \(ArA = AsA\) and \(i \mapsto r_i\) is some choice function \(I \to R\). The modules \(Ar_iA\) are all isomorphic to \(A\) itself, so \(B \cong A \otimes R\) as \(A\)-bimodule, for the vector space \(R\) over \(\Bbbk\) with basis \(\{r_i\}_{i\in I}\). Let us write the original copy of \(A\) in \(B\) as \(A \otimes 1\) for some distinguished element \(1 \in R\) and note that \((a \otimes 1)(b \otimes r) = ab \otimes r\) for all \(a,b \in A\) and \(r \in R\) by construction. Consider \((1 \otimes r_i)(1 \otimes r_j) = \sum_{k \in I}a_k \otimes r_k\), where only finitely many \(a_k\) are non-zero. Now \([a \otimes 1,1 \otimes r_i] = 0 = [a\otimes 1,1 \otimes r_j]\) implies \([a,a_k] = 0\) for all \(a \in A, k \in I\). Therefore, \(a_k \in \Bbbk 1 \subseteq A\) for all \(k \in I\), so 

\begin{equation}
    (1 \otimes r_i)(1 \otimes r_j) = 1 \otimes \sum_{k \in I} C_{ij}^k r_k
\end{equation}
for some \(\{C_{ij}^k\}_{k \in I} \subseteq \Bbbk\) which are almost all 0. In particular, \(R\) is a \(\Bbbk\)-algebra with multiplication determined by \((1 \otimes r)(1 \otimes s) = 1 \otimes rs\). Then \(1 \in R\) is a unit and since \(B\) is associative, \(R\) is too. 

For the second part of the statement, note that \(\widetilde{\beta}(a \otimes 1,b \otimes 1) = \lambda \beta(a,b)\) for some \(\lambda \in \Bbbk^\times\), so
\begin{equation}
    t(r) \coloneqq \frac{1}{n\lambda}\widetilde{\beta}(1 \otimes 1, 1\otimes r) 
\end{equation}
is the desired map \(t \colon R \to \Bbbk\).
\end{proof}

\begin{lemma}\label{lem:alternative_are_commutative}
Let \(R\) be an alternative algebra over a field of characteristic larger then 3 equipped with a linear map \(t \colon R \to \Bbbk\) such that \((r,s) \mapsto t(rs)\) is an algebra metric.
Assume there exists a reduced, commutative, associative subalgebra \(S \subseteq R\) satisfying \(S^\bot \subseteq S\). 

The algebra \(R\) is commutative and associative. 
\end{lemma}
\begin{proof}
Let \(p,q \in S\) and \(r,s \in R\) be arbitrary elements. The identities
\begin{equation}
    t(p(qr)) = t((pq)r) = t((qp)r) = t(r (qp)) = t((r q) p) = t(p (r q))
\end{equation}
show that \(t(p [q,r]) = 0\). As a consequence we see that \([S,R] \subseteq S^\bot \subseteq S\). Furthermore, since \(R\) is alternative, the subalgebra \(\Bbbk[q,r] \subseteq R\) is associative 
and we see that
\begin{align*}
    0 &= [q,[q,r^2]] = 
    [q,[q,r]r + r[q,r]] = [q,[q,r]]r + [q,r]^2 + [q,r]^2 + r[q,[q,r]] =
    2[q,r]^2,
\end{align*}
where we used that \([q,r],[q,r^2] \in S\) implies \([q,[q,r^2]] = 0 = [q,[q,r]]\). Since \(R\) is reduced, we deduce that \([S,R] = 0\). Consequently, 
\begin{equation}
    t(q [r,s]) = t([qr,s]) = 0    
\end{equation}
so \([R,R] \subseteq S^\bot \subseteq S\). Combined with \([R,S] = 0\), this implies \([[r,s],s] = 0 = [[r,sr],s]\), so
\begin{equation}
    [r,s]^2 = [[r,s]r,s] = [[r,sr],s] = 0
\end{equation}
holds. Since \(R\) is reduced, we deduce that \([R,R] = 0\) and the fact that any unital commutative associative algebra over a field of characteristic larger 3 is associative concludes the proof. 
\end{proof}

\noindent
We can now proof Proposition \ref{prop:associative_doubles}. By virtue of Lemma \ref{lem:extensions_of_scalers}, we have \(D(A[\![z]\!],\delta) \cong A \otimes R\) for some unital associative \(\Bbbk\)-algebra \(R\) and \(\textnormal{ev}(a \otimes r, b \otimes s) = \beta(a,b)t(rs)\) for some \(t \colon R \to \Bbbk\) which defines an algebra metric. Since \(A[\![z]\!] \subseteq D(A[\![z]\!],\delta)\) is a Lagrangian subalgebra, \(\Bbbk[\![z]\!]\subseteq R\) is a Lagrangian subalgebra. Therefore, Lemma \ref{lem:alternative_are_commutative} implies that \(R\) is commutative. It is now easy to see that \((R,t)\) is a trace extension of \(\Bbbk[\![z]\!]\) in the sense of Section \ref{sec:trace_extension} and Proposition \ref{prop:trace_extensions} concludes the proof.

\subsection{Categorization of topological associative \(D\)-algebra structures on series}\label{sec:categorization_associative}
Let \(\Bbbk\) be an algebraically closed field of characteristic 0.
Recall that any finite-dimensional simple associative \(\Bbbk\)-algebra is isomorphic to the algebra \(A = \textnormal{M}_{n}(\Bbbk)\) of \(n \times n\)-matrices with entries in \(\Bbbk\) and the bilinear form \(\beta \colon A \times A \to A\) defined by the trace \((a,b) \mapsto \textnormal{tr}(ab)\) is strongly geometrically admissible; see Corollary \ref{cor:geometrically_M_admissible}.

Theorem \ref{thm:categorization_refined} states that
we have four different types of non-triangular associative topological \(D\)-bialgebra structures on \(A[\![z]\!]\). Namely, those associated to solutions of the \(A\)-CYBE which are either trigonometric, rational, quasi-trigonometric, or quasi-rational. 

In Subsection \ref{sec:trigonometric_solutions_ACYBE}, we will show that there are no trigonometric nor quasi-trigonometric solutions of the \(A\)-CYBE. So we are left with two different types of non-triangular associative topological \(D\)-bialgebra structures on \(A[\![z]\!]\). Namely, those associated to solutions of the \(A\)-CYBE which are either rational or quasi-rational. 

In the subsections \ref{sec:rational_A_CYBE_soltuions} and \ref{sec:qrational_A_CYBE_soltuions}, we will establish the structure theory of (quasi-)rational solutions of the \(A\)-CYBE by combining the methods from \cite{aguiar_associative} with the approach of \cite{stolin_sln} to the structure theory of (quasi-)rational solutions of the \(\mathfrak{sl}_n(\Bbbk)\)-CYBE.

An important result to conclude this plan will be the following adaptation of Stolin's theory of maximal orders from \cite{stolin_sln} over \(\mathfrak{sl}_n(\Bbbk)\) to the complete matrix algebra \(A = \textnormal{M}_n(\Bbbk)\).

\begin{proposition}\label{prop:maximal_order}
Let \(W \subseteq A[z,z^{-1}]\) be a subalgebra satisfying \(z^{\pm N}A[z^{\pm 1}] \subseteq W \subseteq z^{\mp N}A[z^{\pm 1}]\) for some \(N \in \bN\). Then there exists \(g \in \textnormal{SL}_n(\Bbbk[z,z^{-1}])\) such that \(\textnormal{Ad}(g)W \subseteq A[z^{-1}]\)
\end{proposition}
\begin{proof}
Without loss of generality, we prove the ``+'' case, i.e.\ \(z^{-N}A[z^{-1}] \subseteq W \subseteq z^NA[z^{-1}]\) for some \(N \in \bN\) implies the existence of \(g \in \textnormal{SL}(n,\Bbbk[z,z^{-1}])\) such that \(\textnormal{Ad}(g)W \subseteq A[z^{-1}]\).

Without loss of generality, we may assume that \(\Bbbk[z^{-1}] \subseteq W\), since we can pass to the algebra \(\Bbbk[z^{-1}]W + \Bbbk[z^{-1}]\) which contains \(W\) and satisfies \(z^{-N}A[z^{-1}] \subseteq \Bbbk[z^{-1}]W + \Bbbk[z^{-1}] \subseteq z^NA[z^{-1}]\). Now, \(W = \pi(W) \oplus \Bbbk[z^{-1}]\) as vector spaces, where \(\pi \colon A[z,z^{-1}] \to \fg[z,z^{-1}]\) is the coefficient-wise application of \(a \mapsto a -\frac{\textnormal{tr}(a)}{n} \in \fg \coloneqq \mathfrak{sl}_n(\Bbbk)\).

The subalgebra \(\pi(W) \subseteq \fg[z,z^{-1}]\) satisfies \(z^{-N}\fg[z^{-1}] \subseteq \pi(W) \subseteq z^N\fg[z^{-1}]\). By virtue of \cite[Theorem 4']{stolin_geometry} and the description of maximal orders for \(\fg = \mathfrak{sl}_n(\Bbbk)\) from \cite{stolin_sln}, there exists \(g \in \textnormal{SL}(n,\Bbbk[z,z^{-1}])\) such that \(\textnormal{Ad}(g)\pi(W) \subseteq \fg[z^{-1}]\). Since \(\textnormal{Ad}(g)\Bbbk[z^{-1}] = \Bbbk[z^{-1}]\), this implies that \(\textnormal{Ad}(g)W \subseteq A[z^{-1}]\).
\end{proof}

\subsection{Absence of (quasi-)trigonometric solutions of the \(A\)-CYBE} \label{sec:trigonometric_solutions_ACYBE}
Let \(\Bbbk\) be an algebraically closed field of characteristic 0, \(A = \textnormal{M}_n(\Bbbk)\) be the \(\Bbbk\)-algebra of \(n\times n\)-matrices, and \(\beta\) be the trace pairing of \(A\). In this section, we prove the following result.

\begin{theorem}\label{thm:no_trigonometric}
    There are no quasi-trigonometric nor trigonometric solutions of the \(A\)-CYBE.
\end{theorem}

We will thereby proceed in two steps. First, we show that (quasi-)trigonometric solutions of the \(A\)-CYBE define certain subalgebras of \(A[v,v^{-1}] \times A[v,v^{-1}]\). Then, using similar methods as in the classification of trigonometric solutions of the \(\fg\)-CYBE for simple Lie algebras \(\fg\) from \cite{belavin_drinfeld_solutions_of_CYBE_paper}, we prove that these subalgebras cannot exist.

\subsubsection{(Quasi-)trigonometric solutions of the \(A\)-CYBE and subalgebras of \(L \times L\) for \(L\coloneqq A[v,v^{-1}]\)} 
Consider \(L \coloneqq A[v,v^{-1}]\). Let us prove that any (quasi-)trigonometric solution \(r\) of the \(A\)-CYBE defines a subspace \(W_r \subseteq L \times L\) such that:
    \begin{enumerate}
            \item \(W_r\) is a subalgebra complementary to the diagonal \(D \coloneqq \{(a,a)\mid a \in L\}\), i.e.\ \(L \times L = D \oplus W_r\);

            \item \(W_r\) is Lagrangian with respect to the algebra metric \(\widetilde{\beta}\) on \(L \times L\) defined by 
            \begin{equation}\label{eq:LxL_bilinearform}
                \widetilde{\beta}((a_1,a_2),(b_1,b_2)) \coloneqq \beta_{(1,1)}^+(a_1,b_1) - \beta_{(1,1)}^+(a_2,b_2)
            \end{equation}
            for \(a_1,a_2,b_1,b_2 \in L\), where \(\beta_{(1,1)}^+\) is defined in \eqref{eq:beta_nlambda_pm};

            \item \(W_r\) is commensurable with \(V \coloneqq A[v] \times A[v^{-1}]\) in the sense that \(\dim((W_r + V)/(W_r \cap V)) < \infty\).
        \end{enumerate}
        
\begin{proof}[Construction of \(W_r\)]
Since all automorphisms of \(A\) are inner, \(L(A,\sigma) \cong A[v,v^{-1}]\) for all \(\sigma \in \textnormal{Aut}_{\Bbbk\textnormal{-alg}}(A)\) of finite order (see \cite{pianzola}). Therefore, both trigonometric and quasi-trigonometric solutions of the \(A\)-CYBE are described by expressions of the form
\begin{equation}\label{eq:trig_sol_of_ACYBE}
    r(v,w) = \frac{w\gamma}{v-w} + t(v,w) \textnormal{ for some }t \in (A\otimes A)[v,v^{-1},w,w^{-1}]
\end{equation}
such that \(r(\exp(x),\exp(y))\) is a solution of the \(A\)-CYBE.
We construct a subspace \(W_r\) to \(r\) in a similar fashion as subalgebras were associated to solutions of the \(A\)-CYBE in Section \ref{sec:proof_solutions_of_CYBE_and_manin_triples}. 

Note that the natural embedding \(L \otimes L \to (L \otimes A)(\!(w^{\pm 1})\!)\) extends to
\begin{equation}\label{eq:laurent_expansion}
    (L \otimes L)[(v-w)^{-1}] \longrightarrow (L \otimes A)(\!(w^{\pm 1})\!)
\end{equation}
by interpreting \((v-w)^{-1}\) as 
\begin{equation}
    \sum_{k \in \bN}v^{-k-1}w^k \in \Bbbk[v,v^{-1}](\!(w)\!) \textnormal{ and }-\sum_{k \in \bN}v^{k}w^{-k-1} \in \Bbbk[v,v^{-1}](\!(w^{-1})\!)
\end{equation}
respectively. These embeddings can be understood as the Laurent expansions in \(w = 0\) and \(w = \infty \) respectively.

Let us consider an \(r\) of the form \eqref{eq:trig_sol_of_ACYBE},
chose an orthonormal basis \(\{b_i\}_{i = 1}^d\subseteq A\) with respect to the trace pairing \(\beta\), and let 
\begin{equation}
    \sum_{k \in \bN}\sum_{i = 1}^d r_{k,i}^+(v) \otimes b_iw^k \in (L \otimes A)(\!(w^{-1})\!) \textnormal{ and }\sum_{k \in \bN}\sum_{i = 1}^d r_{k,i}^-(v) \otimes b_iw^k \in (L \otimes A)(\!(w)\!)
\end{equation}
be the Laurent expansions of \(r\) in \(w = \infty\) and \(w = 0 \) respectively.

If \(t = \sum_{k \in \bZ}\sum_{i = 1}^d t_{k,i}(v) \otimes w^kb_i\), where only finitely many \(t_{k,i}\) are non-zero, and
\begin{equation}
w^-_{k,i} \coloneqq \begin{cases} b_i v^{-k} & k > 0\\ 0 & k \le 0 
\end{cases} \textnormal{ and } w^+_{k,i} \coloneqq \begin{cases} 0 & k > 0\\ -b_iv^{-k} & k \le 0 
\end{cases}
\end{equation}
we have \(r^\pm_{k,i} = w^\pm_{k,i} + t_{k,i}\) for \(k \in \bZ\) and \(i \in \overline{1,d}\).

Let us define 
\begin{equation}\label{eq:W_trigonometric}
    W_r \coloneqq \textnormal{Span}_{\Bbbk}\{(r_{k,i}^+,r_{k,i}^-) \mid k \in \bZ, i \in \overline{1,d}\}.
\end{equation}
Clearly, \(W_r\) is commensurable with \(V\), so we have to verify that conditions (1) and (2) hold.
\end{proof}

\begin{proof}[\(W_r\) satisfies (1)]
It is easy to see that \(L \times L = D \oplus W_r\), so we have to show that \(W_r \subseteq L \times L\) is a subalgebra. Similar to Section \ref{sec:definition_CYBE}, we can define for every 
\[s \in (L \otimes L)[(v-w)^{-1}] = (A \otimes A)[v,v^{-1},w,w^{-1},(v-w)^{-1}]\] the expression
\begin{equation}\label{eq:A_cybe_trig}
    \textnormal{CYB}(s) = s^{13}s^{12} - s^{12}s^{23} + s^{23}s^{13} \in (L \otimes L \otimes L)\left[\frac{1}{(v_1-v_2)(v_1-v_3)(v_2-v_3)}\right].
\end{equation}
by using the notations \eqref{eq:ij_notations_constant} coefficient-wise. Then \(s\) satisfies the \(\textnormal{CYB}(s) = 0\) if and only if \(s(\exp(x),\exp(y))\) satisfies the usual \(A\)-CYBE \eqref{eq:cybe}. In particular, we can see that \(\textnormal{CYB}(s) = 0\) implies already that \(s\) is skew-symmetric. 

Similar arguments as in the Section \ref{sec:proof_solutions_of_CYBE_and_manin_triples} show that, since \(r\) is skew-symmetric, we have
\begin{equation}
    \textnormal{CYB}(r) \in L \otimes L \otimes L.
\end{equation}
Therefore, we can rewrite this expression using the Laurent expansions \eqref{eq:laurent_expansion} in \(v_3 = \infty\) and \(v_3 = 0\) to obtain
\begin{equation}\label{eq:cybe_subalgebra_trigonometric}
    \begin{split}
     \textnormal{CYB}^\pm(r) &=  \sum_{k, \ell \in \bZ} \sum_{i,j = 1}^d
    r^\pm_{\ell, j}r^\pm_{k,i} \otimes b_i z_2^k \otimes b_j z_3^\ell \\& - \sum_{m \in \bN} \sum_{i = 1}^d
    r^\pm_{k,i} \otimes \left(z_2^kb_i^{(1)}r(z_2,z_3) - r(z_2,z_3)b_i^{(2)}z_3^k\right).
    \end{split}
\end{equation}

If \(\textnormal{CYB}(r) = 0\), then \(\textnormal{CYB}^+(r) = 0 = \textnormal{CYB}^-(r)\) and \eqref{eq:cybe_subalgebra_trigonometric} implies that \(W_r \subseteq L \times L\) is a subalgebra. 
\end{proof}

\begin{proof}[\(W_r\) satisfies (2)]
The fact that \(r\) is skew-symmetric is equivalent to \(t = \overline{t} - \gamma\), which means
\begin{equation}\label{eq:t_skew_trig}
     t_{k,i}^{\ell,j} = - t_{\ell,j}^{k,i} - \delta_{ij}\delta_{k0}\delta_{\ell0}
\end{equation}
if \(t = \sum_{k,\ell \in \bZ}t_{k,i}^{\ell,j}b_jv^\ell \otimes b_iw^k\). Furthermore, the identities
\begin{equation}
    \begin{split}
        \widetilde{\beta}((w_{k,i}^+,w_{k,i}^-),(w_{\ell,j}^+,w_{\ell,j}^-)) = \delta_{ij}\delta_{k0}\delta_{\ell0} \textnormal{ and }
        \widetilde{\beta}((w_{k,i}^+,w_{k,i}^-),(b_jv^\ell,b_jv^\ell) ) = -\delta_{ij}\delta_{k\ell}
    \end{split}
\end{equation}
are easily verified.

This implies that, if \(t\) is identified with its image under \(L \otimes L \cong D \otimes D\),
\begin{equation}
    \begin{split}
        \widetilde{\beta}((r_{k,i}^+,r_{k,i}^-),(r_{\ell,j}^+,r_{\ell,j}^-)) &=  \underbrace{\widetilde{\beta}((w_{k,i}^+,w_{k,i}^-),(w_{\ell,j}^+,w_{\ell,j}^-))}_{= \delta_{ij}\delta_{k0}\delta_{\ell0}} + \underbrace{\widetilde{\beta}((t_{k,i},t_{k,i}),(t_{\ell,j},t_{\ell,j}))}_{= 0} \\&+ \underbrace{\widetilde{\beta}((w_{k,i}^+,w_{k,i}^-),(t_{\ell,j},t_{\ell,j}) )}_{t_{\ell,j}^{k,i}} + \underbrace{\widetilde{\beta}((w_{k,i}^+,w_{k,i}^-),(t_{\ell,j},t_{\ell,j}))}_{t_{k,i}^{\ell,j}}\\
    &= t_{k,i}^{\ell,j} + t_{\ell,j}^{k,i} + \delta_{ij}\delta_{k0}\delta_{\ell0} = 0.
    \end{split}
\end{equation}
We conclude that \(W_r\subseteq W_r^\bot\). This, \(L \times L = D \oplus W_r\), and \(D^\bot = D\) imply \(W_r = W_r^\bot\).
\end{proof}

\subsubsection{Proof of Theorem \ref{thm:no_trigonometric}} We will now apply the methods used to classify trigonometric solutions of the \(\fg\)-CYBE for simple Lie algebras \(\fg\) from \cite{belavin_drinfeld_solutions_of_CYBE_paper} (see also \cite[Section 8.2]{abedin_thesis}) to our associative setting and conclude that the constructed \(W \coloneqq W_r\) cannot exist.

Let us define \(R_\pm \in \textnormal{End}_{\Bbbk}(L)\) by \(R_\pm(v^{k}b_i) = r_{k,i}^\pm\) for all \(k \in \bZ\) and \(i \in \overline{1,d}\). Since by definition \(r_{k,i}^\pm = w_{k,i}^\pm + t_{k,i}\), we have \(R_- = 1+R_+\). Then \(W = W^\bot\) can be rewritten as
\begin{equation}
    \begin{split}
        0 &= \widetilde\beta(((R_--1)a),R_-a),((R_--1)b,R_-b)) \\&= \beta_{(1,1)}^+((R_--1)a,(R_--1)b) - \beta_{(1,1)}^+(R_-a,R_-b) \\&= -\beta_{(1,1)}^+(R_-a,b) - \beta_{(1,1)}^+(a,R_-b) +  \beta_{(1,1)}^+(a,b),
    \end{split}  
\end{equation}
for all \(a,b \in L\). Therefore, \(R_- + R_-^* = 1\), where \((\cdot)^*\) denotes the adjoint with respect to \(\beta_{(1,1)}^+\) from \eqref{eq:beta_nlambda_pm}. Furthermore, \(W = \{(R_+a,R_-a)\mid a\in L\}\) holds by definition. Write \(W_\pm \coloneqq \textnormal{Im}(R_\pm)\). Then \(W_\pm^\bot = \textnormal{Ker}(R_\mp) 
\subseteq W_\pm\) follows immediately from \(R_-^* + R_- = 1\) and \(R_+ = R_- - 1\). Moreover, 
\begin{equation}\label{eq:W_cayley_form}
    \begin{split}
        W &= \{(R_+a,R_-b)\mid a,b\in L \textnormal{ s.t.\ }\theta(a + W_+^\bot) = b + W_-^\bot\} \\&= \{(R_+a,R_-b)\mid R_-(a-b) \in W_-^\bot\}
    \end{split}
\end{equation}
holds, where \(\theta \colon W_+/W_+^\bot \to W_-/W_-^\bot\) is defined by \(R_+a + W_+^\bot \mapsto R_-a + W_-^\bot\). Indeed, `` \(\subseteq\) '' is clear and `` \(\supseteq\) '' follows from the fact that \((R_--1)a = R_+a = R_-b\) and \(R_-(a-b)\in W_-^\bot = \textnormal{Ker}(R_+)\) implies \(a = R_-(a-b) \in \textnormal{Ker}(R_+)\), so \(R_-b = R_+a = 0\), which means that 
\begin{equation}
    \{(R_+a,R_-b)\mid R_-(a-b) \in \textnormal{Ker}(R_+)\} \cap D = \{0\}.
\end{equation}

Let us write now \(R \coloneqq R_-\) and recall that \(R_+ = R - 1\). The fact that \(W \subseteq L \times L\) is a subalgebra implies immediately that \(\theta\) is a \(\Bbbk\)-algebra isomorphism and \(W_+, W_- \subseteq L\) are subalgebras. In particular, \(W_- = \textnormal{Im}(R)\) is a subalgebra, so for all \(a,b\in L\) exists \(c \in L\) such that \((Ra)(Rb) = Rc\). Applying \(\theta^{-1}\) gives \(((R-1)a)((R-1)b) = (R-1)c + d\) for some \(d \in W_+^\bot = \textnormal{Ker}(R)\). Subtracting the second from the first equation and applying \(R\) gives
\begin{equation}\label{eq:R_equation_trig}
    (Ra) (Rb) = R((Ra) b + a (Rb) - ab)
\end{equation}
for all \(a ,b \in L\). From this one can deduce 
\begin{equation}\label{eq:R_equation_lamba_mu_form}
    ((R-\lambda)a) ((R-\mu)b) = (R-\mu)((R-\lambda)a b) + (R-\lambda)(a (R-\mu)b) + ((\lambda + \mu -1)R - \lambda \mu) (ab)
\end{equation}
for all \(\lambda,\mu \in \Bbbk\) and \(a,b\in L\).

Let \(L^\lambda \coloneqq \bigcup_{k = 1}^\infty\textnormal{Ker}((R-\lambda)^k)\) be the generalized eigenspace of \(R\) to \(\lambda \in \Bbbk\). Since only finitely many \(t_{k,i}\) are non-zero, we have \(L = \bigoplus_{\lambda \in \Bbbk}L^\lambda\),
\begin{equation}\label{eq:generalized_eigenspaces_bounded}
    v^{N}A[v] \subseteq L^0 \subseteq v^{-N}A[v] \textnormal{ and }v^{-N}A[v^{-1}] \subseteq L^1 \subseteq v^{N}A[v^{-1}].
\end{equation}
Furthermore, \(L^\lambda = \textnormal{Ker}((R-\lambda)^{k})\) for some sufficiently large \(k = k(\lambda) \in \bN\) and \(R + R^* = 1\) implies 
\begin{equation}\label{eq:orthogonality_generalized_eigenspaces}
    \beta_{(1,1)}^+(L^\lambda,L^\mu) \neq 0 \textnormal{ if and only if }\lambda + \mu = 1,    
\end{equation}
where we recall that \(\beta_{(1,1)}^+\) was defined in  \eqref{eq:beta_nlambda_pm}.

Using induction on \(\ell = k_1 + k_2\) and \eqref{eq:R_equation_lamba_mu_form} shows that 
\((R-\lambda)^{k_1}a = 0 = (R-\mu)^{k_2}b\) implies
\begin{equation}
    ((\lambda + \mu -1)R - \lambda \mu)^\ell(ab) = 0.
\end{equation}
As a consequence, we have 
\begin{equation}\label{eq:multiplication_generalized_eigenspaces}
    \lambda + \mu \neq 1 \implies L^\lambda L^\mu = L^{\frac{\lambda\mu}{\lambda + \mu -1}} \textnormal{ and } \lambda \neq \{0,1\} \implies L^\lambda L^{1 - \lambda} = \{0\}.
\end{equation}
Therefore, \(L' \coloneqq \bigoplus_{\lambda \in \Bbbk\setminus\{0,1\}} L^\lambda \subseteq L\) is a finite-dimensional subalgebra. Furthermore, by construction \(R|_{L'},(R-1)|_{L'} \colon L' \to L'\) are both invertible and one can deduce from \eqref{eq:R_equation_trig} that \(R|_{L'}((R-1)|_{L'})^{-1}\) defines an automorphism of \(L'\) without fixed point. In particular, \(L'\) is solvable as Lie algebra by virtue of \cite[Section 9]{belavin_dirnfeld_triangle}.

Combining \eqref{eq:generalized_eigenspaces_bounded} and \eqref{eq:multiplication_generalized_eigenspaces} we see that \(L^0 \oplus L' \subseteq L\) is a subalgebra satisfying \(v^NA[v] \subseteq L^0 \oplus L' \subseteq v^{-N}A[v]\) for some sufficiently large \(N\in \bN\).
This and Proposition \ref{prop:maximal_order} implies that
\(\textnormal{Ad}_{g_0}(L^0 \oplus L') \subseteq A[v]\) holds for some \(g_0 \in \textnormal{SL}_n(\Bbbk[v,v^{-1}])\). Furthermore, \eqref{eq:orthogonality_generalized_eigenspaces} implies that 
\(L^0 = (L^0 \oplus L')^\bot\), so 
\begin{equation}
    vA[v] = A[v]^\bot \subseteq \textnormal{Ad}_{g_0}(L^0) \subseteq \textnormal{Ad}_{g_0}(L^0 \oplus L') \subseteq A[v].
\end{equation}
Consider \(K = \textnormal{Ad}_{g_0}(L^0)/vA[v] \subseteq A[v]/vA[v] = A\). Then \(L^{0,\bot} = L^0 \oplus L' \supseteq L^0\) implies that \(K\) is a subalgebra of \(A\) such that \(0 = \beta(a,b) = \textnormal{tr}(ab)\) for all \(a,b\in K\). Therefore, \(K\) is a solvable Lie ideal of the Lie algebra \(K \oplus \textnormal{Ad}_{g_0}(L') \subseteq A\) by Cartan's criterion and the quotient \(\textnormal{Ad}_{g_0}(L')=(K \oplus \textnormal{Ad}_{g_0}(L'))/K\) is solvable as Lie algebra as mentioned above. In conclusion, \(K \oplus \textnormal{Ad}_{g_0}(L')\) is a solvable sub-Lie algebra of \(A\). Hence, there exists \(g_0' \in \textnormal{SL}_n(\Bbbk)\) such that \(\textnormal{Ad}_{g_0'}(K\oplus \textnormal{Ad}_{g_0}(L'))\) is contained in the standard Borel subalgebra \(B_+ \subseteq A\) of upper triangular matrices. Consequently, \(\textnormal{Ad}_{\widetilde{g}_0}(L^0 \oplus L') \subseteq B_+ \oplus vA[v]\) for \(\widetilde{g}_0 \coloneqq g_0'g_0\).

In a similar fashion we can see that \(\textnormal{Ad}_{\widetilde{g}_1}(L^1 \oplus L') \subseteq B_- \oplus v^{-1}A[v^{-1}]\) for an appropriate \(\widetilde{g}_1 \in \textnormal{SL}_n(\Bbbk[v,v^{-1}])\). 
In particular,
\begin{equation}
    L = L^0 \oplus L' \oplus L^1 = \textnormal{Ad}_{\widetilde{g}_0^{-1}}(B_+ \oplus vA[v]) + \textnormal{Ad}_{\widetilde{g}_1^{-1}}(B_- \oplus v^{-1}A[v^{-1}])
\end{equation}
or \(L = (B_+ \oplus vA[v]) + \textnormal{Ad}_{\widetilde{g}}(B_- \oplus v^{-1}A[v^{-1}])\) for \(\widetilde{g} = \widetilde{g}_0\widetilde{g}_1^{-1}\). By virtue of the Birkhoff decomposition for \(\textnormal{SL}_n(\Bbbk[v,v^{-1}])\), there exists a factorization \(\widetilde{g} = b_+wb_-\) such that 
\begin{equation}
    \textnormal{Ad}_{b_\pm}(B_\pm \oplus v^{\pm 1}A[v^{\pm 1}]) = B_\pm \oplus v^{\pm 1}A[v^{\pm 1}]    
\end{equation}
and \(w = P\textnormal{diag}(v^{k_1},\dots,v^{k_n})\) for \(k_1,\dots,k_n \in \bZ\) and a permutation matrix \(P \in \textnormal{SL}_n(\Bbbk)\); see e.g.\ \cite[6.8.2 Theorem and Section 13.2.2]{kumar}. Therefore, we obtain 
\begin{equation}
    L = (B_+ \oplus vA[v]) + \textnormal{Ad}_{w}(B_- \oplus v^{-1}A[v^{-1}]).    
\end{equation}
It is easy to see that this can only hold if \(\textnormal{Ad}_w = 1\). Consequently, we have 
\begin{equation}\label{eq:conjugate_to_borel}
    \textnormal{Ad}_g(L^0 \oplus L') \subseteq B_+ \oplus vA[v] \textnormal{ and }\textnormal{Ad}_g(L^1 \oplus L') \subseteq B_- \oplus v^{-1}A[v^{-1}]  
\end{equation}
for \(g = b_+\widetilde{g}_0\), where \(B_- \subseteq A\) is the Borel subalgebra of lower triangular matrices.

Let us replace \(R\) with \(\textnormal{Ad}_gR\textnormal{Ad}_{g^{-1}}\), then \(W\) is replaced with \((\textnormal{Ad}_g \times \textnormal{Ad}_g)W\) and \(L^\lambda\) is replaced with \(\textnormal{Ad}_gL^\lambda\). In particular, we can assume \(g = 1\) in \eqref{eq:conjugate_to_borel}.

Let \(\widetilde{L}^\lambda = \{a \in L \mid [a, L^{\lambda}] \subseteq L^\lambda\}\) be the Lie normalizer of \(L^\lambda\). By induction on \(k \in \bN\), we can prove that for every \(a,b\in L\) satisfying \((R-\lambda)^kb = 0\) and \(aL^{\lambda} \subseteq L^\lambda\), we have \((Ra)b \in L^\lambda\). Indeed, we have
\begin{equation}
    (R-\lambda)((Ra)b) = (Ra)((R-\lambda)b) - R(a(Rb) - ab) \in L^\lambda
\end{equation}
by induction assumption, where \eqref{eq:R_equation_trig} and the fact that \(R(L^\lambda) \subseteq L^\lambda\) were used. Similarly, we can prove that \(L^{\lambda}a \subseteq L^\lambda\) implies \(b(Ra) \in L^\lambda\). This proves that 
\begin{equation}\label{eq:R_and_idealizers}
    R(\widetilde{L}^\lambda)\subseteq \widetilde{L}^\lambda.    
\end{equation}
Since 
\begin{equation}
    [B_+,B_+] \oplus vA[v] = (B_+ \oplus vA[v])^\bot \subseteq (L^0 \oplus L')^\bot = L^0 \subseteq L^0\oplus L' \subseteq B_+ \oplus vA[v]
\end{equation}
holds, the Lie normalizer of \(L^0\oplus L'\) is \(B_+ \oplus vA[v]\). Similarly, the Lie normalizer of \(L^1 \oplus L'\) is \(B_- \oplus v^{-1}A[v^{-1}]\). In conclusion, 
\begin{equation}
    B_+ \oplus vA[v] = \bigoplus_{\lambda \neq 1}\widetilde{L}^\lambda\textnormal{ and }B_- \oplus v^{-1}A[v^{-1}] = \bigoplus_{\lambda \neq 0}\widetilde{L}^\lambda
\end{equation}
and \eqref{eq:R_and_idealizers} implies \(R(B_\pm \oplus v^{\pm 1}A[v^{\pm1}]) \subseteq B_\pm \oplus v^{\pm 1}A[v^{\pm1}]\) and consequently \(R(H) \subseteq H\) for the subalgebra \(H \subseteq A\) of diagonal matrices. Therefore,
\begin{equation}\label{eq:Wpm_triangular_decomposition}
    W_\pm = ((N_+ \oplus vA[v]) \cap W_\pm) \oplus (H \cap W_\pm) \oplus ((N_- \oplus v^{-1}A[v^{-1}]) \cap W_\pm) 
\end{equation}
since \(W_+ = \textnormal{Im}(R-1)\) and \(W_- = \textnormal{Im}(R)\), where \(N_+\) (resp.\ \(N_-\)) is the subalgebra of \(A\) of strictly upper (resp. lower) matrices.

Let \(H_\pm \coloneqq H\cap W_\pm\).
Since \(W_\pm^\bot \subseteq W_\pm\) with respect to \(\beta^+_{(1,1)}\) from \eqref{eq:beta_nlambda_pm}, we can deduce from \eqref{eq:Wpm_triangular_decomposition} that \((H^\bot_\pm \cap H)\subseteq H_\pm\) holds with respect to the trace pairing \(\beta\). However, \(H_\pm\) are both Artinian \(\Bbbk\)-algebras without nilpotents, hence \(H_\pm \cong \Bbbk^{\ell_\pm}\) for some \(\ell_\pm \le n\) as algebras. Observe that any embedding \(\Bbbk \to H\) is of the form \(a \mapsto a\delta_I\), where \(I = \{i_1,\dots,i_k\}\subseteq \overline{1,n}\), \(\delta_I \coloneqq h_{i_1} + \dots + h_{i_k}\), and \(h_i\) is the diagonal matrix with 1 at the \(i\)-th row and column as only non-zero entry. Therefore, 
\[H_\pm = \textnormal{Span}_\Bbbk\left\{\delta_{I^\pm_1},\dots,\delta_{I^\pm_{\ell^\pm}}\right\}\] for some pairwise disjoint \({I^\pm_1},\dots,{I^\pm_{\ell^\pm}} \subseteq \overline{1,n}\). This implies \(H = H_\pm \oplus (H_\pm^\bot \cap H)\) and combined with \((H_\pm^\bot \cap H) \subseteq H_\pm\) we are left with \(H_\pm = H\). 

Finally, we see that \(1 \in W_+ \cap W_-\) and, since \(\theta\) is a \(\Bbbk\)-algebra isomorphism, \(\theta\) is unital and \(\theta(1) = 1\). Consequently, \((1,1) \in W\) which contradicts \(W \cap D = \{0\}\). In conclusion, \(W\) cannot exist. This proves that \(r\) cannot exist. Since \(r\) was an arbitrary (quasi-)trigonometric solution of the \(A\)-CYBE, we have proven Theorem \ref{thm:no_trigonometric}.

\subsection{Structure theory of rational \(D\)-bialgebra structures over \(A\)}\label{sec:rational_A_CYBE_soltuions}
As in the previous subsection, let \(\Bbbk\) be an algebraically closed field of characteristic 0, \(A = \textnormal{M}_n(\Bbbk)\) be the \(\Bbbk\)-algebra of \(n\times n\)-matrices, and \(\beta\) be the trace pairing of \(A\). 

The assignment \(r \mapsto A(r)\) defines a bijection between rational solutions of the \(A\)-CYBE and Lagrangian subalgebras \(W \subseteq A(\!(z)\!)\) satisfying \(A[\![z]\!] \oplus W = A(\!(z)\!)\) and \(z^{-N}A[z^{-1}] \subseteq A(r) \subseteq z^NA[z^{-1}]\) for some sufficiently large \(N \in \bN_0\). 
By virtue of Proposition \ref{prop:maximal_order}, there exists \(g \in \textnormal{SL}(n,\Bbbk(\!(z)\!))\) such that \(\textnormal{Ad}(g)A(r) \subseteq A[z^{-1}]\). By virtue of e.g.\ \cite[2.2 Sauvage Lemma]{stolin_sln}, there exists \(d = \textnormal{diag}(z^{d_1},\dots, z^{d_n})\) such that \(g = g_-dg_+\) for \(g_+ \in \textnormal{SL}(n,\Bbbk[\![z]\!])\) and \(g_- \in \textnormal{SL}(n,\Bbbk[z^{-1}])\). Therefore, up to equivalence, \(A(r) \subseteq d^{-1} A[z^{-1}] d\). The fact that \(A[\![z]\!] + A(r) = A(\!(z)\!)\) holds implies that \(0 \le d_i \le 1\) for all \( i \in \overline{1,n}\). Thus, after reordering the indices, we can assume that \(d = d_k \coloneqq (1,\dots,1, z,\dots, z)\), where \(z\) appears \(k\)-times on the right hand side. 

We call \(r\) rational solution of \emph{type \(k\)}, if \(A(r) \subseteq N_k \coloneqq d_k^{-1} A[z^{-1}] d_k\), where we remark that
\begin{equation}\label{eq:N_k}
N_k \coloneqq \left\{\begin{pmatrix} A & B \\ C & D \end{pmatrix}\in L = \textnormal{M}_n(\Bbbk[z,z^{-1}])\,\Bigg|\, \substack{ A \in \textnormal{M}_{n-k}(\Bbbk[z^{-1}]),\, B \in z\textnormal{M}_{(n-k)\times k}(\Bbbk[z^{-1}])\\  C \in z^{-1}\textnormal{M}_{k\times (n-k)}(\Bbbk[z^{-1}]),\,D\in \textnormal{M}_{k}(\Bbbk[z^{-1}])}\right\}.
\end{equation}
We will now show that these solutions are parametrized by associative versions of Stolin pairs, which parameterize rational solutions of the \(\mathfrak{sl}_n(\Bbbk)\)-CYBE. To this end, let 
\begin{equation}
    P_k \coloneqq \left\{\begin{pmatrix} A & B \\ 0 & D \end{pmatrix}\in A = \textnormal{M}_n(\Bbbk)\,\Bigg|\, A \in \textnormal{M}_{n-k}(\Bbbk),\, B \in \textnormal{M}_{(n-k)\times k}(\Bbbk), \textnormal{ and }  D\in \textnormal{M}_{k}(\Bbbk)\right\}
\end{equation}
Then \((S,\chi)\) is called \emph{associative Stolin pair of type \(k\)} if \(S \subseteq A\) is a subalgebra and \(\chi \colon S \times S \to \Bbbk\) is a bilinear form such that

\begin{itemize}
    \item \(S + P_k = A\);
    
    \item \(\chi\) is a \emph{Connes 2-cocycle}, i.e.\ \(\chi\) is skew-symmetric and
    \begin{equation*}
        \chi(a_1a_2,a_3) + \chi(a_2a_3,a_1) + \chi(a_3a_1,a_2) = 0
    \end{equation*}
    holds for all \(a_1,a_2,a_3 \in S\),
    and \(\chi\) restricts to a non-degenerate bilinear form on \(S \cap P_k\).
\end{itemize}
By adjusting the arguments in \cite{stolin_sln}, we will prove the following result.

\begin{theorem}\label{thm:classification_rational_case}
Rational solutions of the \(A\)-CYBE of type \(k\) are in bijection with associative Stolin pairs of type \(k\).
\end{theorem}

\begin{remark}
Let us note that Stolin pairs of type 0 are simply subalgebras \(S \subseteq A\) which admit a non-degenerate Connes 2-cocycle.

Theorem \ref{thm:classification_rational_case}.(1) states that these are in bijection with rational solutions \(r\) of the \(A\)-CYBE satisfying \(A(r) \subseteq A[z^{-1}]\). It is easy to see that \(r(x,y) = \frac{\gamma}{x-y} + t\) for a constant tensor \(t \in A\otimes A\). Then \(r\) is a solution of the \(A\)-CYBE if and only if \(t\) is a skew-symmetric solution of the \(A\)-CYBE. The fact that these are in bijection with Stolin pairs of type 0 is actually exactly \cite[Proposition 2.7]{aguiar_associative}.
\end{remark}

\subsubsection{Proof of Theorem \ref{thm:classification_rational_case}}
It suffices to prove that there is a bijection between Lagrangian subalgebras \(W \subseteq N_k\) such that \(A(\!(z)\!) = A[\![z]\!] \oplus W\) and Stolin pairs of type \(k\). 

Observe that the image of \(A[\![z]\!] \cap N_k\) in \(D_{\epsilon} \coloneqq N_k/z^{-2}N_k \cong A[\epsilon]/\epsilon^2A[\epsilon] = A \oplus \epsilon A\) is exactly \(P_k \oplus \epsilon P_k^\bot\) and \(D_{\epsilon}\) inherits the algebra metric
\begin{equation}
    \beta_\epsilon(a_1 + \epsilon a_2,b_1 + \epsilon b_2) \coloneqq \beta(a_1,b_2) + \beta(a_2,b_1)
\end{equation}
from \(A(\!(z)\!)\).

Since \(z^{-2}N_k = N_k^\bot \subseteq W^\bot = W \subseteq N_k\) holds, we can see that \(W \mapsto W/z^{-2}N_k\) defines a bijection between Lagrangian subalgebras \(W \subseteq N_k\) such that \(A(\!(z)\!) = A[\![z]\!] \oplus W\) and Lagrangian subalgebras \(V \subseteq D_\epsilon\) such that \(D_\epsilon = (P_k \oplus \epsilon P_k^\bot) \oplus V\).
Therefore, it suffices to establish a bijection between the latter Lagrangian subalgebras and associative Stolin pairs of type \(k\).

Let \(V \subseteq D_\epsilon\) be any Lagrangian subspace and \(S\) be the image of \(V\) under \(\epsilon \mapsto 0\). Then 
\begin{equation}
    \epsilon S^\bot = (S \oplus \epsilon A)^\bot \subseteq V^\bot = V \subseteq S \oplus \epsilon A.
\end{equation}
A dimension argument implies that the mapping \(\epsilon \mapsto 0\) defines an isomorphism \(V/\epsilon S^\bot \to S\). In other words, there exists a linear map \(f\colon S \mapsto  A/ S^\bot\) such that \(V/\epsilon S^\bot = \{a + \epsilon f(a) \mid a \in S \}\). Consider the bilinear form on \(S\) defined by \(\chi(a,b) \coloneqq \beta(f(a),b)\) for \(a,b \in S\). Observe that, since \(\beta\) pairs \(S\) and \(A/S^\bot\) non-degenerately, \(f\) is completely determined by \(\chi\). Furthermore, \(\chi\) is skew-symmetric since
\begin{equation}
    0 = \beta_{\epsilon}(a + \epsilon f(a), b + \epsilon f(b)) = \chi(a,b) + \chi(b,a).
\end{equation} 
Note that \(V\) can be uniquely reconstructed from \(S\) and \(f\) and hence from the pair \((S,\chi)\). This establishes a bijection between Lagrangian subspaces \(V \subseteq D_\epsilon\) and pairs \((S,\chi)\) consisting of subspaces \(S \subseteq A\) equipped with a skew-symmetric bilinear form \(\chi\).

It remains to prove that \(V \subseteq D_\epsilon\) is a subalgebra satisfying \(D_\epsilon = (P_k \oplus \epsilon P_k) \oplus V\) if and only if \((S,\chi)\) is a Stolin pair of type \(k\).

Observe that \(V \subseteq D_\epsilon\) is a subalgebra if and only if for all \(a,b\in S\)
\begin{equation}
    (a + \epsilon f(a))(b + \epsilon f(b)) = ab + \epsilon (f(a)b + af(b)) \in V/\epsilon S^\bot.
\end{equation}
and this is equivalent to \(f(ab) = f(a)b + af(b)\).
Now let us note that
\begin{equation}
    \begin{split}
        &\chi(a_1a_2,a_3) + \chi(a_2a_3,a_1) + \chi(a_3a_1,a_2) = \chi(a_1a_2,a_3) - \chi(a_1,a_2a_3) - \chi(a_2,a_3a_1) \\&= \beta(f(a_1a_2),a_3) - \beta(f(a_1),a_2a_3) - \beta(f(a_2),a_3a_1) \\&= \beta(f(a_1a_2),a_3) - \beta(f(a_1)a_2,a_3) - \beta(a_1f(a_2),a_3) \\&= \beta(f(a_1a_2) - f(a_1)a_2 - a_1f(a_2),a_3),
    \end{split}
\end{equation}
where we used the skew-symmetry of \(\chi\) and the associativity of \(\beta\).
Since \(S\) and \(A/S^\bot\) are non-degenerately paired by \(\beta\), this identity shows that \(f(ab) = f(a)b + af(b)\) for all \(a,b\in S\) is equivalent to the fact that \(\chi\) is a Connes 2-cocycle. 

To conclude the proof, we have to show that \(D_\epsilon = (P_k \oplus \epsilon P_k) \oplus V\) is equivalent to the facts that \(S + P_k = A\) holds and \(\chi\) is non-degenerate on \(S \cap P_k\). 

Assume first that \(D_\epsilon = (P_k \oplus \epsilon P_k) \oplus V\) and
observe that \(S+ P_k = A\) immediately follows from \(D_\epsilon = (P_k \oplus \epsilon P_k^\bot) + V\). Assume that \(a \in S \cap P_k\) satisfies \(\chi(a,b) = 0\) for all \(b \in S \cap P_k\). In other words, \(a_f \in (S \cap P_k)^\bot = S^\bot + P_k^\bot\) for any representative \(a_f\) of \(f(a)\), so \(a_f = a_1 - a_2\) for \( a_1 \in P_k^\bot\) and \(a_2 \in S^\bot\). Then \(a + \epsilon( a_f + a_2) \in V \cap (P_k \oplus \epsilon P_k^\bot) = \{0\}\). This proves that \(\chi\) is non-degenerate on \(S \cap P_k\). 

Conversely, assume that \(S + P_k = A\) and \(\chi\) is non-degenerate on \(S \cap P_k\). Let 
\[a + \varepsilon(a_f + a^\bot) = p + \epsilon p^\bot \in (P_k \oplus \epsilon P_k) \cap V,\]
for \(a \in S, a^\bot \in S^\bot, p \in P_k\), \(p^\bot \in P_k^\bot\), and a representative \(a_f \in A\) of \(f(a) \in A/S^\bot\). Then \(a = p \in S\cap P_k\) and \(a_f = p^\bot - a^\bot \in S^\bot + P_k^\bot = (S \cap P_k)^\bot\). This implies that \(\chi(a,b) = \beta_\epsilon(f(a),b) = 0\) for all \(b \in S \cap P_k\), so \(a = 0\) since \(\chi\) is non-degenerate on \(S \cap P_k\). Therefore, 
\[a_f + a^\bot = p^\bot \in S^\bot \cap P_k^\bot = (S + P_k)^\bot = \{0\}.\]
Summarized, \((P_k \oplus \epsilon P_k) \cap V = \{0\}\) and by dimension reasoning we see that \(D_\epsilon = (P_k \oplus \epsilon P_k^\bot) \oplus V\).

\subsection{Structure theory of quasi-rational \(D\)-bialgebra structures over \(A\)}\label{sec:qrational_A_CYBE_soltuions}
Recall that \(\Bbbk\) is an algebraically closed field of characteristic 0, \(A = \textnormal{M}_n(\Bbbk)\) is the \(\Bbbk\)-algebra of \(n\times n\)-matrices, and \(\beta\) is the trace pairing of \(A\).

The assignment \(r \mapsto ((D_2(A),\beta_{(2,1)}),A[\![z]\!],A(r))\) defines a bijection between quasi-rational solutions of the \(A\)-CYBE and Manin triples \(((D_2(A),\beta_{(2,1)}),A[\![z]\!],W)\) satisfying \(z^{-N}A[z^{-1}] \subseteq W_+ \subseteq z^NA[z^{-1}]\) for some sufficiently large \(N \in \bN\). Here, \(W_+\) is the projection of \(W \subseteq D_2(A) = A(\!(z)\!) \times A[z]/z^2A[z]\) onto \(A(\!(z)\!)\). 

Repeating the arguments in Section \ref{sec:rational_A_CYBE_soltuions}, we obtain \(W_+ \subseteq N_k = d_k^{-1}A[z^{-1}]d_k\) for some \(k \in \overline{1,n}\) up to equivalence. Here, \(N_k\) is explicitly given in \eqref{eq:N_k}.

We call a quasi-rational solution \(r\) of the \(A\)-CYBE of \emph{type \(k\)}, if \(A(r) \subseteq N_k \times A[z]/z^2A[z]\). 

\begin{theorem}\label{thm:classification_qrational_case}
Quasi-rational solutions of the \(A\)-CYBE of type \(k\) are in bijection with associative Stolin pairs of type \(k\).
\end{theorem}

\begin{remark}
In general, the rational and quasi-rational solution of the \(A\)-CYBE associated to the same Stolin pair have no obvious connection. 

However, if \((S,\chi)\) is a Stolin pair of type \(0\), then the associated rational solution is \(r(x,y) = \frac{\gamma}{x-y} + t\) for some \(t \in A \otimes A\) and the associated quasi-rational solution is \(\widetilde{r}(x,y) = \frac{xy\gamma}{x-y} + t = \frac{y^2\gamma}{x-y} - x\Omega + t\). In particular, \(r(x^{-1},y^{-1}) = \widetilde{r}(x,y)\). Observe that \(z\mapsto z^{-1}\) is not an admissible coordinate transformation of \(A(\!(z)\!)\) and \(D_2(A)\).
\end{remark}

\subsubsection{Proof of Theorem \ref{thm:classification_qrational_case}}
Let \(r\) be a quasi-rational solution of the \(A\)-CYBE of type \(k\) and \(((D_2(A),\beta_{(2,1)}),A[\![z]\!],W)\) be the associated Manin triple.

Recall from Lemma \ref{lemm:quasirational_lemma}.(1) and its proof that \(W = W_+ \times W_-\) for some Lagrangian subalgebras \(W_+ \subseteq A(\!(z)\!)\) and \(W_- \subseteq A[z]/z^2A[z]\). Since \(A[z^{-1}]\) and consequently \(N_k\) is Lagrangian in \(A(\!(z)\!)\), where \(A(\!(z)\!)\) is equipped with \(\beta_{(2,1)}^+\) from \eqref{eq:beta_nlambda_pm}, \(W_+ \subseteq N_k\) implies \(N_k = N_k^\bot \subseteq W_+^\bot = W_+\) and thus \(W_+ = N_k\). 

Now Lemma \ref{lemm:quasirational_lemma}.(2),(3) states that \(W_+ \cap A[\![z]\!]\) can be embedded into \(A[z]/z^2A[z]\) in such a way that \((W_+ \cap A[\![z]\!]) \oplus W_- = A[z]/z^2A[z]\). But we have seen in the proof of Theorem \ref{thm:classification_rational_case} that this image of \(W_+ \cap A[\![z]\!]\) is precisely \(P_k \oplus [z]P_k^\bot\) and that the decompositions \((P_k \oplus [z]P_k^\bot) \oplus W_- = A[z]/z^2A[z]\) into Lagrangian subalgebras are in bijection with Stolin pairs of type \(k\).

Since all steps made are invertible, we obtain the desired bijection.

\subsection{Outlook on Jordan bialgebras}\label{sec:outlook_jordan}
Let \(A\) be a finite-dimensional simple algebra over an algebraically closed field \(\Bbbk\) of characteristic 0. If \(A\) is a Lie algebra, the classification of non-triangular topological Lie bialgebra structures \((A[\![z]\!],\delta)\) was concluded in \cite{abedin_maximox_stolin_zelmanov} and is based on several previous results in the theory of the classical Yang-Baxter equation and topological Lie bialgebra structures such as e.g.\ \cite{belavin_drinfeld_solutions_of_CYBE_paper, stolin_sln, stolin_maximal_orders, montaner_stolin_zelmanov, abedin_universal_geometrization}. In this section, we considered the case of an associative algebra \(A\) and established the analogous classification of non-triangular topological associative \(D\)-bialgebra structures \((A[\![z]\!],\delta)\) by adapting the above mentioned results to this setting. 

It is natural to ask for the analogous result in the Jordan case. In particular, if \(A\) is a Jordan algebra, is it possible to achieve a similar classification of non-triangular topological Jordan bialgebra \((A[\![z]\!],\delta)\)? In order to achieve such a classification, one would first need to check that, similar to the associative and Lie case, \((A[\![z]\!],\delta)\) is automatically non-degenerate in the sense of Section \ref{sec:manin_triples_over_series_explicit}. Then the classification would reduce to the classification of trigonometric, rational, quasi-trigonometric, and quasi-rational solutions of the \(A\)-CYBE \eqref{eq:cybe} by virtue of Theorem \ref{thm:categorization_refined}. 

We conjecture that, similar to the associative case considered in this section, the unitality of \(A\) obstructs the existence of trigonometric and quasi-trigonometric solutions of the \(A\)-CYBE. If so, the classification could then be concluded by adapting the structure theory of rational and quasi-rational solutions from \cite{stolin_sln,stolin_maximal_orders} to the Jordan setting.

A potentially fruitful approach to this plan could be by using the Kantor-Koecher-Tits construction, which associates to the Jordan algebra \(A\) a Lie algebra. This might reduce some of the above problems to the Lie algebra case. Indeed, it was shown in \cite{zhelyabin} that there exists a compatibility between the Kantor-Koecher-Tits construction, Jordan bialgebra structures, and Lie bialgebra structures. On the other hand, the absence of (quasi-)trigonometric solutions of the \(A\)-CYBE could also be approached by trying to adapt the proof of Theorem \ref{thm:no_trigonometric} directly to the Jordan setting.

We will consider these ideas in future research.

\appendix
\section{Notations and conventions} \label{sec:notation}

\noindent
Throughout this document \(\Bbbk\) denotes the base field we are working over. From Section \ref{sec:general_categorization} onward it will be of characteristic 0 and from Section \ref{sec:categorization_general_whole_section} onward it will be additionally algebraically closed. By convention the set of natural numbers \(\bN = \{0,1,2,\dots\}\) include 0 and we use the notation \(\overline{m,n} = \{m, \dots, n\}\) for the set of natural numbers between a number \(m\) and larger number \(n\).

\subsubsection*{Commutative algebra} In this text, rings are always unital, associative, and commutative. For a ring \(R\) and \(R\)-modules \(M,N\), the space of \(R\)-linear maps \(M \to N\) (resp. \(M \to M\)) is denoted by \(\Hom_R(M,N)\) (resp. \(\End_R(M)\)), while the tensor product of \(M\) and \(N\) is written as \(M \otimes_R N\). For \(R = \Bbbk\) the indices are omitted. The invertible elements of \(R\) are denoted by \(R^\times\), and \(M^* \coloneqq \Hom_R(M,R)\) is the dual module of \(M\).  If \(R\) is a domain, \(\textnormal{Q}(R) \coloneqq (R\setminus \{0\})^{-1}R\) denotes its quotient field and we write \(\textnormal{Q}(M) \coloneqq M \otimes_R \textnormal{Q}(R)\). Let \(f\colon R \to \widetilde{R}\) be a morphism of rings and \(\widetilde{M}\) be an \(\widetilde{R}\)-module. We say that a map \(g \colon M\to \widetilde{M}\) is \(f\)-equivariant if it is a group homomorphism satisfying \(g(r m) = f(r) g(m)\) for all \(r\in R,m\in M\).

\subsubsection*{Non-associative algebra} Let \(R\) be a ring.
In this text, an \(R\)-algebra \(A\) satisfies no additional assumptions if not mentioned explicitly, i.e.\ \(A = (A,\mu_A)\) consists of an \(R\)-module \(A\) equipped with a multiplication map \(\mu_A \colon A \otimes_R A \to A\). The left (resp.\ right) multiplication maps with respect to an element \(a \in A\) are denoted by \(L_a\) (resp.\ \(R_a\)), i.e.
\begin{equation}
    L_a(b) = ab = R_b(a) \textnormal{ for all }a,b \in A.
\end{equation}
The group of invertible \(R\)-algebra endomorphisms of \(A\), i.e.\ invertible \(R\)-linear maps \(f \colon A \to A\) satisfying \(f\mu_A = \mu_A(f\otimes f)\), will be denoted by \(\Aut_{R\textnormal{-alg}}(A)\). We note that `` \(\oplus\) '' will always denote the direct sum of modules and not of algebras, while `` \(\times\) '' is used for the latter.
For any \(a,a_1,\dots,a_n \in A\), we write
\begin{equation}
    \begin{split}
        &a^{(i)}(a_1\otimes \dots \otimes a_n) = a_1\otimes \dots \otimes aa_i \otimes \dots \otimes a_n\\&(a_1\otimes \dots \otimes a_n)a^{(i)} = a_1\otimes \dots \otimes a_ia \otimes \dots \otimes a_n.    
    \end{split}
\end{equation}
We say that a map \(\beta \colon A \times A \to R\) is an algebra metric if it is a non-degenerate symmetric \(R\)-bilinear map such that
\begin{align}
    \beta(ab,c) = \beta(a,bc) \textnormal{ for all }a,b,c \in A.
\end{align}
In this case, we call the pair \((A,\beta)\) metric \(R\)-algebra.

\subsubsection*{Formal series.} For a module \(M\) over a ring \(R\), the module of formal power series in the formal variable \(z\) with coefficients in \(M\) is denoted by 
\begin{equation}
    M[\![z]\!] \coloneqq \left\{m = \sum_{k \in \bN}m_kz^k\,\Bigg|\, m_k \in M\textnormal{ for }k \in \bN\right\}.    
\end{equation}
Furthermore, we write \(M[\![z_1,\dots,z_k]\!] \coloneqq M[\![z_1]\!]\dots[\![z_k]\!]\).
The \(R\)-module \(R[\![z]\!]\) (resp.\ \(R[\![z_1,\dots,z_k]\!]\))
is a ring extension of \(R\) and \(M[\![z]\!]\) (resp.\ \(M[\![z_1,\dots,z_k]\!]\)) is an \(R[\![z]\!]\)-module (resp.\ \(R[\![z_1,\dots,z_k]\!]\)-module). Then \(M(\!(z)\!) \coloneqq M[\![z]\!][z^{-1}] = \textnormal{Q}(M[\![z]\!])\) is the module of formal Laurent series. We note that if \(M\) is an \(R\)-algebra the module \(M[\![z]\!]\) (resp. \(M(\!(z)\!)\)) is naturally an \(R[\![z]\!]\)-algebra  (resp.\ \(R(\!(z)\!)\)-algebra). Elements \(p\) in \(M(\!(z)\!)\) (resp. \(M(\!(z_1)\!)\dots(\!(z_k)\!)\)) will sometimes be denoted with the formal variable (resp. variables) for convenience, i.e.\ \(p = p(z)\) (resp. \(p = p(z_1,\dots,z_k)\)). A generic element \(p \in M(\!(z)\!)\) is written \(p(z) = \sum_{k \in \bZ}p_kz^k\) and \(p'(z) = \sum_{k \in \bZ} kp_kz^{k-1}\) denotes the formal derivative of \(p\). If \(p(z) \in mz^{-k} + z^{-k+1}M[\![z]\!]\), it is said to be of order \(k\) with main part \(mz^{-k}\). Finally, 
\begin{equation*}
    M[\![z_1,\dots,z_k]\!]^\vee \coloneqq \{f \in M[\![z_1,\dots,z_k]\!]^*\mid f((z_1,\dots,z_k)^mM[\![z_1,\dots,z_k]\!]) = \{0\} \textnormal{ for some }m\in\bN\}
\end{equation*}
is the continuous dual of \(M[\![z_1,\dots,z_k]\!]\).

\subsubsection*{Algebraic geometry.} Let \(X = (X,\sheafO_X)\) be a ringed space and \(\sheafF,\sheafG\) be two \(\sheafO_X\)-modules. For a morphism \(f\colon X \to Y = (Y,\sheafO_Y)\) of ringed spaces, we denote the additional structure morphism by \(f^{\flat}\colon \sheafO_Y \to f_*\sheafO_X\) and write \(f^\sharp \colon f^{-1}\sheafO_Y \to \sheafO_X\) for the induced morphism. The set of \(\sheafO_X\)-module homomorphisms \(\sheafF \to \sheafG\) (resp. \(\sheafF \to \sheafF\)) is denoted by \(\Hom_{\sheafO_X}(\sheafF,\sheafG)\) (resp. \(\End_{\sheafO_X}(\sheafF)\)) while its sheaf counterpart is denoted by \(\sheafHom_{\sheafO_X}(\sheafF,\sheafG)\) (resp. \(\sheafEnd_{\sheafO_X}(\sheafF))\). In particular, we write \(\sheafF^* \coloneqq \sheafHom_{\sheafO_X}(\sheafF,\sheafO_X)\). The tensor product of \(\sheafF\) and \(\sheafG\) is written as \(\sheafF\otimes_{\sheafO_X}\sheafG\).

Assume that \(X\) and \(Y\) are \(S\)-schemes. The fiber product of \(X\) and \(Y\) over \(S\) is denoted by \(X\times_S Y\) and \(\sheafF|_p\) is the fiber of \(\sheafF\) in a point \(p\in X\). If \(S = \textnormal{Spec}(\Bbbk)\), the index \(S\) is omitted and \(\textnormal{H}^n(\sheafF)\) denotes the \(n\)-th global cohomology group of \(\sheafF\), while \(\textnormal{h}^n(\sheafF)\) denotes its dimension over \(\Bbbk\), if said space is finite-dimensional. In particular, \(\textnormal{H}^0(\sheafF) = \Gamma(X,\sheafF)\) is the space of global sections of \(\sheafF\).

\printbibliography
\end{document}